\documentclass[11pt, reqno]{amsart}
  
\usepackage{graphicx}
\usepackage{amscd,amsmath,amsopn,amssymb, amsthm,multicol, wasysym} 
\usepackage{hyperref}
\usepackage[all]{xy}
\usepackage{amscd}
\usepackage{lscape}
\usepackage{slashed}
\usepackage{graphicx}
\usepackage{lscape}
\usepackage{multirow}
\usepackage{listings}
\usepackage{booktabs}
 \usepackage{cancel}
   \usepackage{color}
 \usepackage{young}
\newtheorem{theorem}{Theorem}[section]
\newtheorem{lemma}[theorem]{Lemma}
\newtheorem{prop}[theorem]{Proposition}
\newtheorem{corol}[theorem]{Corollary}
\newtheorem{remark}[theorem]{Remark}

\theoremstyle{definition}
\newtheorem{definition}[theorem]{Definition}
\newtheorem{example}[theorem]{Example}

\theoremstyle{remark}

\numberwithin{equation}{section}

 \input ulem.sty
 
\DeclareMathOperator{\Ad}{Ad}
\DeclareMathOperator{\ad}{ad}
\DeclareMathOperator{\Aut}{Aut}

\DeclareMathOperator{\Id}{Id}
\DeclareMathOperator{\tr}{tr}

\DeclareMathOperator{\Sym}{Sym}

\DeclareMathOperator{\Cl}{C\ell}
\DeclareMathOperator{\Ric}{Ric}
\DeclareMathOperator{\Sca}{Scal}
\DeclareMathOperator{\Hom}{Hom}

\DeclareMathOperator{\Span}{span}
\DeclareMathOperator{\Spin}{Spin}

\DeclareMathOperator{\Iso}{Iso}
\DeclareMathOperator{\Jac}{Jac}
\newcommand{\fr}{\mathfrak}
\newcommand{\La}{\Lambda}
\newcommand{\al}{\alpha}
\newcommand{\be}{\beta}

\newcommand{\bb}{\mathbb}
\newcommand{\cal}{\mathcal}

\DeclareMathOperator{\SO}{SO}
\DeclareMathOperator{\Sp}{Sp}
 \DeclareMathOperator{\SU}{SU}
\DeclareMathOperator{\U}{U}
\DeclareMathOperator{\G}{G}
\DeclareMathOperator{\F}{F}

\DeclareMathOperator{\E}{E}
\DeclareMathOperator{\Ss}{S}
\DeclareMathOperator{\Ed}{End}
\DeclareMathOperator{\Gl}{GL}
\newcommand{\thickline}{\noalign{\hrule height 1pt}}

\begin{document}  
 

\title{Invariant connections and $\nabla$-Einstein structures on  isotropy irreducible   spaces}
\author{Ioannis Chrysikos, Christian O'Cadiz Gustad  and Henrik Winther}
 \address{Faculty of Science,  University of Hradec Kr\'alov\'e,  Rokitanskeho 62 Hradec Kr\'alov\'e
 50003, Czech Republic}
 \email{ioannis.chrysikos@uhk.cz}
    \address{Institute of Mathematics and Statistics, University
of Troms\o, Troms\o  \ 90-37, Norway }
 \email{cgustad@gmail.com}
    \address{Institute of Mathematics and Statistics, University
of Troms\o, Troms\o  \  90-37, Norway }
 \email{henrik.winther@uit.no}



 
 \medskip

 \begin{abstract}
This paper is devoted to a systematic study and classification of  invariant  affine or metric connections on certain classes of naturally reductive spaces.  For  any non-symmetric, effective, strongly isotropy irreducible homogeneous Riemannian manifold $(M=G/K, g)$, we  compute the dimensions of the spaces of  $G$-invariant affine and metric connections. For such manifolds we also describe the space of  invariant metric connections with skew-torsion.  For the compact Lie group  $\U_{n}$ we classify all bi-invariant metric connections, by introducing  a new family of bi-invariant connections  whose torsion is of vectorial type. Next we  present  applications related with the notion of $\nabla$-Einstein manifolds with skew-torsion. In particular,   we classify all such invariant structures on any  non-symmetric  strongly isotropy irreducible homogeneous space.

 

 \end{abstract}
\maketitle  

\section*{Introduction}
 \noindent{\bf Motivation.}  Given a homogeneous space $M=G/K$ with a reductive decomposition $\fr{g}=\fr{k}\oplus\fr{m}$, a $G$-invariant  affine connection $\nabla$ is nothing but  a connection in the frame bundle $F(M)=G\times_{K}{\rm GL}(\fr{m})$  of $M$ which is also $G$-invariant.
 The first studies  of  invariant connections were perfomed  by  Nomizu \cite{N},  Wang \cite{Wang} and Kostant \cite{K} during the fifties.   After that, homogeneous connections on principal bundles have    attracted  the interest   of both mathematicians and physicists,  with  several different perspectives; for example Cartan connections  and parabolic geometries \cite{Cap}, Lie triple systems and Yamaguti-Lie algebras \cite{Eldu,    Ben1}, Yang--Mills and gauge theories \cite{Itoh, Laqq}, etc.   From another    point of view,   invariant  connections are crucial  in the holonomy  theory of naturally reductive spaces and Dirac operators, mainly due to the special properties of the {\it canonical connection} (or the characteristic connection in  terms of special structures, see \cite{Kob2,  Olmos, OReg,  CSwan, Agr03, AF, AFH}). 
 {\small  
  \begin{table}
 \caption{Invariant connections on  compact irreducible symmetric spaces due to \cite{Laq1, Laq2}}
 \label{table:one}
 \begin{tabular}{l l l}
 \hline
  Type I & $M=G/K$ &   invariant connections $\cal{A}ff_{G}(F(M))$  \\
  \thickline
  AI & $\SU_{n}/\SO_{n}$ $(n\geq 3)$  & 1-dimensional  family\\
 AII & $\SU_{2n}/\Sp_{n}$ $(n\geq 3)$ & 1-dimensional family\\
 EIV & $\E_6/\F_4$ & 1-dimensional    family\\
 & all the other cases & canonical connection  $\equiv$ Levi-Civita connection  \\
 \hline
Type II & $M=(G\times G)/\Delta G$ &  bi-invariant connections $\cal{A}ff_{G\times G}(F(M))$  \\ 
  \thickline 
  & $\SU_{n}$ $( n\geq 3)$ & 2-dimensional family \\
  & all the other simple Lie groups & 1-dimensional family (inducing the flat $\pm 1$-connections) \\
                          \hline
              \end{tabular}    
                       \end{table}}

    According to \cite{Wang},  given a $G$-homogeneous principal bundle $P\to G/K$ with structure group $U$, there is a bijective correspondence  between $G$-invariant connections on $P$ and certain linear maps $\Lambda : \fr{g}\to\fr{u}$, where $\fr{g}, \fr{u}$  are  the Lie algebras of $G$ and $U$, respectively.  Wang's correspondence was  successfully used  by Laquer  \cite{Laq1, Laq2} during the nineties  to describe the set  of invariant affine connections,  denoted by $\cal{A}ff_{G}(F(G/K))$, on compact irreducible Riemannian symmetric spaces $M=G/K$.   For most  cases, Laquer proved that  $\cal{A}ff_{G}(F(G/K))$  consists of   the canonical connection (simple Lie groups admit a line of canonical connections), except for  a few   cases where  new  1-parameter families  arise, see Table \ref{table:one}.  By contrast, much less is known about  invariant connections on  {\it non-symmetric}   homogeneous spaces, even in the isotropy irreducible case. For example, the first author   in \cite{Chrysk}, considered invariant connections on manifolds $G/K$   diffeomorphic to a symmetric space, which however do not induce a symmetric  pair $(G, K)$,  e.g. $\G_2/\SU_3\cong \Ss^{6}$ and $\Spin_{7}/\G_2\cong\Ss^{7}$.  There, it was shown that   the space of   $\G_2$-invariant affine or metric connections on the sphere $\Ss^{6}=\G_2/\SU(3)$ is 2-dimensional, while  the space of $\Spin_{7}$-invariant affine or metric  connections on the 7-sphere $\Ss^{7}=\Spin_{7}/\G_2$  is 1-dimensional.  This  is a remarkable result,  since the only   $\SO_{7}$- (resp. $\SO_{8}$)-invariant  affine (or metric) connection on the symmetric space  $\Ss^{6}=\SO_{7}/\SO_{6}$ (resp. $\Ss^{7}=\SO_{8}/\SO_{7}$) is the canonical connection.

  Motivated by this simple result, in this article  we classify  invariant  affine connections on  {\it (compact) non-symmetric  strongly isotropy irreducible homogeneous Riemannian manifolds}.
       A connected  effective  homogeneous  space  $G/K$ is called isotropy irreducible if  $K$ acts irreducibly on $T_{o}(G/K)$ via the isotropy representation. If the identity component $K_{0}$ of $K$ also acts irreducibly  on $T_{o}(G/K)$, then $G/K$ is called {\it strongly isotropy irreducible}.  Obviously, any  strongly isotropy irreducible space  (SII space for short)  is also isotropy irreducible but the converse is false, see \cite{Bes}.
 Non-symmetric  strongly isotropy irreducible homogeneous spaces  were originally classified by Manturov (see for example \cite{Bes}) and   were later studied by Wolf \cite{Wolf} and others.  Any SII space admits a unique invariant Einstein metric, the so-called Killing metric and  in the non-compact case  such a manifold is  a symmetric space of non-compact type.   In fact, SII spaces  share many properties with symmetric spaces and indeed, any irreducible (as Riemannian manifold) symmetric space is strongly isotropy irreducible.
  A conceptual relationship between symmetric spaces and SII spaces was explained in \cite{Wa5}.  More recently, isotropy irreducible homogeneous   spaces  endowed with their canonical connection $\nabla^{c}$ was shown  to have a special relationship with   geometric structures with torsion (see \cite{CSwan, Cleyton}).

  \smallskip
\noindent {\bf Outline and classification results.} 
    After recalling  preliminaries in Section \ref{I} about (invariant) metric connections and their torsion types,  in  Section \ref{derivat} we fix a reductive homogeneous space $(M=G/K, \fr{g}=\fr{k}\oplus\fr{m})$ and    introduce the notion of generalized derivations of a tensor   $F  : \otimes^{p}\fr{m} \to\fr{m}$.  When  $F$ is  $\Ad(K)$-invariant and $\mu\in{\rm Hom}_{K}(\fr{m}\otimes\fr{m}, \fr{m})$ is a $K$-intertwining map, we prove that $\mu$ induces a generalized derivation of $F$  if and only if $F$ is $\nabla^{\mu}$-parallel, where $\nabla^{\mu}$ is the invariant connection on $M$ associated to $\mu$ (Theorem \ref{mdiffer}).  Moreover, we conclude that  for an invariant tensor field  $F$    the operation induced by   a generalized derivation $\mu\in{\rm Hom}_{K}(\fr{m}\otimes\fr{m}, \fr{m})$, coincides with the covariant derivative $\nabla^{\mu}F$ (Corollary \ref{corder}). 
    Then we consider   derivations on $\fr{m}$  and  provide  necessary and sufficient conditions for their existence   (Theorem \ref{deriv2}), generalizing   results from \cite{Chrysk}. 
    
    Next, in Section \ref{natural}   we present a series of new results related to invariant connections and their torsion type, on compact,  effective, naturally reductive Riemannian manifolds.  In particular,  we examine both the symmetric and non-symmetric case and  we develop some theory available for the classification of all $G$-invariant metric connections, with respect to a naturally reductive metric (see  Lemma \ref{nicea}, Lemma  \ref{dimskew}, Theorem \ref{plethysm}).   In fact, in this way we correct  some wrong conclusions given in  \cite{AFH, Chrysk}. 
       For example, for the compact Lie group $\U_{n}$ $(n\geq 3)$ endowed with a bi-invariant metric we present a   class of  bi-invariant metric connections  whose torsion is not a 3-form, but  of  vectorial type  (Theorem \ref{mtr1}, Proposition \ref{un}).  
    
    In Section \ref{II}  we focus on (compact) non-symmetric, strongly isotropy irreducible homogeneous Riemannian manifolds  $(M=G/K, g=-B|_{\fr{m}})$ with aim  the classification of all $G$-invariant affine or metric connections.  We always work with an effective $G$-action, and based on our previous results on effective naturally reductive spaces   we first  prove that a $G$-invariant metric connection on  $(M=G/K, g=-B|_{\fr{m}})$  cannot admit a component of vectorial type (Proposition \ref{right1}).  Then, in the spin case we  describe an application   about  the formal self-adjointness of Dirac operators associated to invariant metric connections  on such types of homogeneous spaces   (Corollary \ref{spin}).     

  Notice now that any (effective) non-symmetric SII space  $M=G/K$ admits a family of invariant metric connections  induced by the $\Ad(K)$-invariant bilinear map   $\eta^{\al} : \fr{m}\times\fr{m}\to\fr{m}$ with $\eta^{\al}(X, Y):=\frac{1-\al}{2}[X, Y]_{\fr{m}}$.  In full details, this family, which we call the {\it Lie bracket family}, has the form 
\[
\nabla^{\al}_{X}Y=\nabla^{c}_{X}Y+\eta^{\al}(X, Y)=\nabla^{g}_{X}Y+\frac{\al}{2}T^{c}(X, Y),
\]
 where $\nabla^{c}$ denotes the canonical connection associated to $\fr{m}$ and $\nabla^{g}$ the Levi-Civita connection of the  Killing metric.\footnote{The parameter $\alpha$ can be a real or complex number, depending on the   type   of the isotropy representation $\fr{m}$. } Hence, its torsion is  an invariant  3-form  on $M$, given by $T^{\al}=\al\cdot T^{c}$, where $T^{c}$ is the torsion of $\nabla^{c}$ (see \cite{Agr03, Chrysk}).   However,  we will show that in general the family $\nabla^{\al}$ does not exhaust all  $G$-invariant metric connections, even with skew-torsion.
  In particular, for the classification of invariant connections on $M=G/K$ one needs to decompose  the  modules $\Lambda^{2}(\fr{m})$ and $\Sym^{2}(\fr{m})$ into irreducible submodules.   For  such a procedure    we mainly use  the \text{LiE} program\footnote{\url{http://wwwmathlabo.univ-poitiers.fr/~maavl/LiE/}}, but   also provide  examples of how such spaces can be treated only by pure representation theory arguments,  without the aid of a computer (see paragraph \ref{examples}). 
    As a result, for any effective non-symmetric (compact)  SII homogeneous Riemannian manifold $(M=G/K, g=-B|_{\fr{m}})$ we state    the dimension of the space ${\rm Hom}_{K}(\fr{m}\otimes\fr{m}, \fr{m})$ (see Theorem \ref{class}, Tables \ref{table:four}, \ref{table:five}).   
 In addition to this,    for any such homogeneous space we present the space of $G$-invariant torsion-free connections  and classify the dimension of the space of $G$-invariant metric connections. Moreover, we state the multiplicity of the (real) trivial representation inside the space $\Lambda^{3}(\fr{m})$ of  3-forms.  This last step yields finally the presentation of the subclass of  $G$-invariant metric connections with skew-torsion. Note that all these desired multiplicities  were also obtained in \cite{Cleyton}, up to some  errors/omissions,   see Remark \ref{wCley} and Table \ref{table:four} for corrections.  
      We summarize our results as follows:

   \medskip
    \noindent {\bf Theorem A.1.} \label{MAINNEW1} {\it Let $(M=G/K, g=-B|_{\fr{m}})$ be an effective   non-symmetric  SII space. Then: }\\
{\it  (i) The  family $\{\nabla^{\al} : \al\in\bb{R}\}$ exhausts all $G$-invariant affine or metric connections on $M=G/K$, if and only if $G=\Sp_{n}$, or $M=G/K$ is one of  the manifolds \[
\begin{tabular}{lllll}
$\SO_{14}/\Sp_{3}$, &   $\SO_{4n}/\Sp_{n}\times\Sp_{1} \ (n\geq 2)$, & $\SO_{7} /\G_{2}$, & $\SO_{16} /\Spin_{9}$, & $\G_{2} /\SO_{3}$, \\
$\F_{4} /(\G_{2} \times\SU_{2})$,  & $\E_{7} /(\G_{2} \times \Sp_{3})$, & $\E_{7} /(\F_{4}\times\SU_{2})$, & $\E_{8} /(\F_{4} \times\G_{2})$. & 
\end{tabular}
\]
The same family exhausts  also all $\SU_{2q}$-invariant metric connections on the homogeneous space    $\SU_{2q}/\SU_{2}\times\SU_{q}$  $(q\geq 3)$,   but  not all the $\SU_{2q}$-invariant affine connections.

\noindent (ii) The     family $\{\nabla^{\al} : \al\in\bb{C}\}$ exhausts all $G$-invariant affine or metric connections on $M=G/K$,  if and only if $M=G/K$ is one of  the manifolds
\[
\begin{tabular}{llll}
$\SO_{8}/\SU_{3}$,& $\G_{2} /\SU_{3}$, &  $\F_{4} /(\SU_{3} \times\SU_{3})$, &  $\E_{6} /(\SU_{3}\times \SU_{3}\times \SU_{3})$, \\
$\E_{7} /(\SU_{3} \times \SU_{6})$, & $\E_{8} / \SU_{9}$, & $\E_{8} /(\E_{6} \times \SU_{3})$. &
\end{tabular}
    \]
 }
 \smallskip
 For  invariant metric connections  different from  $\nabla^{\al}$, we prove that
 
\medskip
 \noindent {\bf Theorem A.2.}
{\it   Let $(M=G/K, g=-B|_{\fr{m}})$ be an effective   non-symmetric  SII space   which admits  at least one  invariant metric connection, different from the Lie bracket family. Then, $M=G/K$ is isometric to a space given in Table \ref{table:two}. In this table we present the dimensions of the spaces $\Hom_{K}(\fr{m}, \Lambda^{2}\fr{m}) $ and $(\Lambda^{3}\fr{m})^{K} $, which respectively parametrize the space of invariant metric connections and the space of invariant metric connections with totally skew-symmetric torsion. In particular:\\ 
  (i) Any   homogeneous  space in Table \ref{table:two} whose isotropy representation  is of real type and which is not isometric to $\SO_{10}/\Sp_{2}$,  admits  a 2-dimensional space of  $G$-invariant metric connections with skew-torsion.
For $\SO_{10}/\Sp_{2}$, the unique family of $\SO_{10}$-invariant metric connections with skew-torsion is given by $\nabla^{\al}$ $(\al\in\bb{R})$. However,  the space of all $\SO_{10}$-invariant metric connections   is 2-dimensional. \\
   (ii) Any   homogeneous  space in Table \ref{table:two} whose isotropy representation  is of  complex type, admits  a 6-dimensional space of $G$-invariant metric connections and a  4-dimensional subspace of   $G$-invariant metric connections with skew-torsion.

}

{\small
\begin{table}
\caption{Non-symmetric  SII spaces carrying new  $G$-invariant metric connections and the dimension of the space of global $G$-invariant 3-forms}
 \label{table:two}
\begin{tabular}{llcc}
   \multicolumn{4}{c}{Real type}  \\
  \thickline
    &  &  Invariant metric connections & skew-torsion \\
 \ $M=G/K$ (families) & $\dim_{\bb{R}} M$ &  $\dim_{\bb{R}}\Hom_{K}(\fr{m}, \Lambda^{2}\fr{m})$ & $\dim_{\bb{R}}(\Lambda^{3}\fr{m})^{K}$ \\
   \thickline

     \multirow{10}*{}   $\SU_{pq}/\SU_{p}\times \SU_{q}$ $(p, q\geq 3)$ & $(p^2-1)(q^2-1)$ & 2 & 2 \\ \cmidrule(l){2-4}
                                         \    $\SO_{\frac{n(n-1)}{2}}/\SO_{n}$ $(n\geq 9)$ & $\frac{1}{8}(6 n - 5 n^2 - 2 n^3 + n^4)$ & 3 & 2 \\ \cmidrule(l){2-4}
                                          \     $\SO_{\frac{(n-1)(n+2)}{2}}/\SO_{n}$ $(n\geq 7)$ & $\frac{1}{8} (8 - 2 n - 9 n^2 + 2 n^3 + n^4)$ & 3 & 2 \\ \cmidrule(l){2-4}
                                           \     $\SO_{(n-1)(2n+1)}/\Sp_{n}$ $(n\geq 4)$ & $\frac{1}{2}(2 + n - 9 n^2 - 4 n^3 + 4 n^4)$ & 3 & 2 \\ \cmidrule(l){2-4}
                                            \   $\SO_{n(2n+1)}/\Sp_{n}$ $(n\geq 3)$ & $\frac{1}{2}(-3 n - 5 n^2 + 4 n^3 + 4 n^4)$ & 3 & 2  \\\midrule[0.08em]

\  (low-dim cases) & \\ 
\thickline
    \multirow{5}*{}    $\SO_{21}/\SO_{7}$ & $189$ & 3 & 2 \\ \cmidrule(l){2-4}
                                      \           $\SO_{28}/\SO_{8}$ & $350$ & 4 & 2 \\ \cmidrule(l){2-4}
                                       \             $\SO_{14}/\SO_{5}$ & $81$ & 3 & 2 \\ \cmidrule(l){2-4}
                                        \              $\SO_{20}/\SO_{6}$& $175$   & 3 & 2 \\ \cmidrule(l){2-4}
                                         \              $\SO_{10}/\Sp_{2}$ & $35$ & 2 & 1  \\\midrule[0.08em]

\ (exceptions) & \\ 
\thickline
 \multirow{5}*{}    $\SO_{14}/\G_2$& $70$  & 2 & 2 \\ \cmidrule(l){2-4}
                                          \          $\SO_{26}/\F_4$ & $273$ &  2 &  2 \\ \cmidrule(l){2-4}
                                           \      $\SO_{42}/\Sp_{4}$ & $825$ &  2 &   2 \\ \cmidrule(l){2-4}
                                        \ $\SO_{52}/\F_4$ & $1274$ &   2  & 2 \\ \cmidrule(l){2-4}
                                       \  $\SO_{70}/\SU_{8}$ & $2352$  &  2 &   2 \\ \cmidrule(l){2-4}
                                     \  $\SO_{248}/\E_8$  & $30380$ &   2 &   2 \\ \cmidrule(l){2-4}
                                    \  $\SO_{78}/\E_6$  & $2925$ &   2   & 2 \\ \cmidrule(l){2-4}
                                     \  $\SO_{128}/\Spin_{16}$ & $8008$ &   2 &   2 \\ \cmidrule(l){2-4}
                            \ $\SO_{133}/\E_7$ & $8645$ &   2 &   2   \\ \cmidrule(l){2-4}
                           \  $\E_7/\SU_3$ & $125$ & 2 & 2   \\\bottomrule

                               \multicolumn{4}{c}{Complex type}  \\
  \thickline
  &  &  Invariant metric connections & skew-torsion \\
 \ $M=G/K$  & $\dim_{\bb{R}} M$ &  $\dim_{\bb{R}}\Hom_{K}(\fr{m}, \Lambda^{2}(\fr{m}))$ & $\dim_{\bb{R}}(\Lambda^{3}\fr{m})^{K}$ \\
   \thickline
    \multirow{3}*{}   $\SO_{n^{2}-1}/\SU_{n}$ $(n\geq 4)$ &  $\frac{1}{2} (4 - 5 n^2 + n^4)$ & 6 & 4  \\ \cmidrule(l){2-4}
                            \  $\E_6/\SU_{3}$ & $70$ &  6 & 4 \\ \bottomrule
\end{tabular}
\end{table}}

  \smallskip
  For  invariant connections  induced by   some $0\neq\mu\in{\rm Hom}_{K}(\Sym^{2}\fr{m}, \fr{m})$, we prove the following:
  
  \medskip
\noindent {\bf Theorem B.} \label{MAINNEW2}
 {\it  Let $(M=G/K, g=-B|_{\fr{m}})$ be an effective   non-symmetric  SII space, which admits at least one invariant affine connection $\nabla^{\mu}$,  induced by some $0\neq \mu\in\Hom_{K}(\Sym^{2}\fr{m},\fr{m})$. Then:\\
 (i) If the associated isotropy representation  is of real type, then $M=G/K$  is isometric to a manifold in  Table \ref{table:three}. In this table,   $\bold s$ states for the dimension of the module $\Hom_{K}(\Sym^{2}\fr{m},\fr{m})$,  which parametrizes the space of such invariant connections. }

{\it \noindent (ii)  If the associated isotropy representation  is of complex  type, then $M=G/K$ is isometric to one  of the manifolds $\SO_{n^{2}-1}/\SU_{n}$ $(n\geq 4)$  or $ \E_6/\SU_{3}$,  where the   dimension of the space of such invariant connections is  2  and 4, respectively.
}

{\it \noindent  (iii) The $G$-invariant connection $\nabla^{\mu}$ does not preserve the Killing metric $g=-B|_{\fr{m}}$. Thus,  $\nabla^{\mu}$ is not metric with respect to any $G$-invariant metric.}

{\small  \begin{table}
\caption{Non-symmetric  SII spaces of real type carrying    $G$-invariant affine  connections  induced by     $0\neq\mu \in {\rm Hom}_{K}(\Sym^{2}\fr{m}, \fr{m})$}
\label{table:three}
\begin{tabular}{l | l | l  l}
 \multicolumn{4}{l}{Real type}  \\
  \thickline
${\bold s=1}$                                                       & ${\bold s=2}$               & ${\bold s=3}$                    &  \\
\thickline
$\SU_{10}/\SU_{5}$    & $\SU_{\frac{n(n-1)}{2}}/\SU_{n}$  $(n\geq 6)$ & $\SO_{28}/\SO_{8}$ &  \\
$\SO_{\frac{n(n-1)}{2}}/\SO_{n}$ $(n\geq 9)$  &  $\SU_{\frac{n(n+1)}{2}}/\SU_{n}$ $(n\geq 3)$ & $\E_7/\SU_3$ & \\
 $\SO_{\frac{(n-1)(n+2)}{2}}/\SO_{n}$ $(n\geq 7)$ &  $\SO_{20}/\SO_{6}$ & & \\
$\SO_{21}/\SO_{7}$  & $\SU_{27}/\E_6$   &  & \\
$\SO_{14}/\SO_{5}$  & $\SU_{pq}/\SU_{p}\times \SU_{q}$ $(p, q\geq 3)$   & & \\ 
$\SO_{(n-1)(2n+1)}/\Sp_{n}$ $(n\geq 4)$  &    &\\
$\SO_{n(2n+1)}/\Sp_{n}$ $(n\geq 3)$ & &\\
$\SU_{2q}/\SU_{2}\times\SU_{q}$ $(q\geq 3)$  & & \\
$\SO_{10}/\Sp_{2}$ & &\\
$\SU_{16}/\Spin_{10}$ && \\
$\SO_{70}/\SU_{8}$ & &\\
$\E_6/\G_2$ && \\
$\E_6/(\G_2\times\SU_3)$ & & \\
   \thickline
\end{tabular}
\end{table}}

\medskip
 Now, a small combination of Theorems A.1, A.2 and B yields the desired  dimension of the space of all {\it $G$-invariant affine connections}  for any   non-symmetric (compact) SII space $M=G/K$, 
\[
{\cal{N}}:=\dim_{\bb{R}}{\cal A}ff_{G}(F(G/K))=\dim_{\bb{R}}{\rm Hom}_{K}(\fr{m}\otimes\fr{m}, \fr{m}).
\]
 We refer to the Tables \ref{table:four} and \ref{table:five}, where the number ${\cal{N}}$ is explicitly  indicated. Note  that for SII homogeneous spaces $M=G/K$ of the   Lie group $G=\SU_{n}$,   we can describe explicitly  some of the  $\SU_{n}$-invariant affine  connections induced by a symmetric $K$-intertwining map  $0\neq\mu \in {\rm Hom}_{K}(\Sym^{2}\fr{m}, \fr{m})$ (and in a few cases all such connections, see Corollary \ref{symmsun}).  
We also conclude  that  the space of invariant {\it torsion-free connections} on   a   non-symmetric SII space $M=G/K$, denoted by  ${\cal A}ff^{0}_{G}(F(G/K))$, is parametrized by an affine subspace of ${\rm Hom}_{K}(\fr{m}\otimes\fr{m}, \fr{m})$, which is modelled on ${\rm Hom}_{K}(\Sym^2\fr{m}, \fr{m})$ and  contains the Levi-Civita connection, see Lemma \ref{torsionfree} and Remark \ref{henrik}.  In particular, for any $G$-invariant affine connection $\nabla^{\mu}$ induced by $\mu\in {\rm Hom}_{K}(\Sym^2\fr{m}, \fr{m})$, the invariant connection $\nabla:=\nabla^{\mu}-\frac{1}{2}T^{\mu}$ is torsion-free. Thus, the following is a direct consequence of Theorem B.
 
\medskip 
\noindent {\bf Corollary of Theorem B.}
{\it The classification of non-symmetric SII spaces which admit new invariant torsion-free connections, in addition to the Levi-Civita connection, is given by the manifolds of Theorem B. In particular, for a space in Table \ref{table:three} we have $\dim_{\bb{R}}{\cal A}ff^{0}_{G}(F(G/K))=\bold s$, and for the  almost complex homogeneous spaces in Theorem B it is  $\dim_{\bb{R}}{\cal A}ff^{0}_{G}(F(G/K))=2 $ or $4$, respectively. }

  \smallskip
\noindent {\bf Classification  of $\nabla$-Einstein structures with skew-torsion.}
 After obtaining Theorems A.1, A.2 and B, in the final Section  \ref{IV} we turn our attention  to more geometric problems. We use our classification results of Table 2  to examine   $\nabla$-Einstein structures with skew-torsion.  Roughly speaking, such a structure   consists  of  a $n$-dimensional connected Riemannian manifold $(M, g)$ endowed with a metric connection $\nabla$ which has  non-trivial skew-torsion $0\neq T\in\Lambda^{3}(T^{*}M)$  and whose    Ricci tensor has symmetric part a multiple of the metric tensor, i.e.  (see \cite{FrIv, 3Sak, AFer, Chrysk, Chrysk2, Draper})
\[
 \Ric^{\nabla}_{S}=\frac{\Sca^{\nabla}}{n}g. 
 \]
  For $T=0$ the whole notion reduces to the original Einstein metrics. 
    In fact, like Einstein metrics on compact Riemannian manifolds, in \cite{AFer} it was shown that  $\nabla$-Einstein structures can be characterized variationally. On the other hand,  the classification of $\nabla$-Einstein structures with skew-torsion on a fixed Riemannian manifold $(M, g)$, is initially based on the classification of all  metric connections on $M$  whose torsion is a non-trivial  3-form.  
  For example, for odd dimensional spheres  $\Ss^{2n+1}\cong \SU_{n+1}/\SU_{n}$ endowed with their Sasakian structure, a classification  of {\it $\SU_{n+1}$-invariant   $\nabla$-Einstein structures with skew-torsion}  has been  very recently given in \cite{Draper}, and it  follows only after the classification of $\SU_{n+1}$-invariant metric connections (with skew-torsion) and their description in terms of tensor fields related to the special structure (see also \cite{AFer}).
  
 
 As far as we know, most well-understood  examples of $\nabla$-Einstein manifolds appear  in the context  of non-integrable geometries, where a metric connection with skew-torsion $0\neq T$ is adapted to the geometry under consideration, the so-called {\it characteristic connection} $\nabla^{c}$ (see \cite{FrIv}).  
 This connection, which in the homogeneous case  coincides with the canonical connection,  plays a crucial role in the theory of special geometries and  nowadays  is a traditional  approach to describing the associated non-integrable structure in terms of $\nabla^{c}$ (or the very closely related  {\it intrinsic torsion}).  Moreover, the articles   \cite{FrIv, 3Sak, AFer}  provide  some nice classes of   $\nabla^{c}$-Einstein structures, e.g. nearly K\"ahler  manifolds in dimension 6, nearly-parallel $\G_2$-manifolds in dimension 7, or  7-dimensional 3-Sasakian manifolds.   Notice that  these special structures admit  ($\nabla^{c}$-parallel) real Killing spinors and hence,  in some cases  one can describe a  deeper relation  between  the  $\nabla$-Einstein  condition  and a class of  spinor fields,  known as {\it Killing spinors with torsion}. These  are natural  generalizations of the  original  Killing spinor fields, satisfying the Killing spinor equation with respect to  a metric connection with skew-torsion. Their  existence is known for several   types of special geometries (see \cite{ABK, Julia, Chrysk2}).   For example, on 6-dimensional nearly K\"ahler manifolds, 7-dimensional nearly parallel $\G_2$-manifolds, or even on $\Ss^{3}\cong\SU_{2}$, such spinors  are induced  by the associated $\nabla^{c}$-parallel spinors and their description is given in terms of whole 1-parameter families  $\{\nabla^{s} : s\in\bb{R}\}$ of metric connections with skew-torsion. Moreover, their  existence   imposes the following strong geometric constraint: $\Ric^{s}=\frac{1}{n}\Sca^{s} g$       for any $s\in\bb{R}$ \cite{Chrysk2} (although in general  this is not the case, see \cite{Julia}). 
  The special value $s=1/4$ corresponds to the characteristic connection (which has parallel torsion $T$), while the parameter $s=0$  induces the original Einstein metric   related with the existent real Killing spinor. 
 
   Beside these  classes of $\nabla$-Einstein manifolds, the first author  in \cite{Chrysk} studies homogeneous $\nabla$-Einstein structures for more general manifolds, e.g.  on  compact   isotropy irreducible spaces and  a class of normal homogeneous manifolds with two isotropy summands. An important result for us  from \cite{Chrysk}, is that any  effective compact   isotropy irreducible homogeneous space $M=G/K$ which is not a symmetric space of Type I, is a $\nabla^{\al}$-Einstein manifold for any parameter $\al\neq 0$, where $\nabla^{\al}$ is the Lie bracket family. 
 As a consequence of the results in Section \ref{II}, we conclude  that any (effective) non-symmetric   SII homogeneous space $(M=G/K, g=-B|_{\fr{m}})$
   is a  $\nabla^{\al}$-Einstein manifold for any parameter $\al\neq 0$. 
    Moreover, our Lemma \ref{dimskew} in combination with Schur's lemma, yield a  natural parameterization of the set of $G$-invariant $\nabla$-Einstein structures with   skew-torsion, by the space of invariant metric connections with non-trivial skew-torsion, or equivalently of the vector space of (global) invariant 3-forms. 
   Hence,  in this case  the space of all homogeneous $\nabla$-Einstein structures with skew-torsion on $(M=G/K, g=-B|_{\fr{m}})$, denoted by ${\cal{E}}^{sk}_{G}(\SO(G/K, -B|_{\fr{m}}))$,   can be viewed as an affine subspace of the space of all $G$-invariant metric connections.     Combining with our classification results on $G$-invariant metric connections  with skew-torsion (see Theorems A.1, A.2, Table 2), we finally deduce that 

  \medskip
   \noindent {\bf Theorem C.} {\it  Let $(M=G/K, g=-B|_{\fr{m}})$ be an effective non-symmetric  SII space and assume that the family $\nabla^{\al}$ exhausts all $G$-invariant metric connections. Then,  the  associated space of  $G$-invariant $\nabla$-Einstein structures with skew-torsion   has dimension either
  \[
  \dim_{\bb{R}} {\cal{E}}^{sk}_{G}(\SO(G/K, -B|_{\fr{m}}))=1, \quad \text{or} \quad  \dim_{\bb{R}} {\cal{E}}^{sk}_{G}(\SO(G/K, -B|_{\fr{m}}))=2, 
  \]
    for spaces with isotropy representation of real or complex type, respectively, and the manifold is one of the manifolds of Theorem A.1 or $\SO_{10}/\Sp_{2}$.}

\medskip
 For the new   invariant metric connections with skew-torsion, different from  the family $\nabla^{\al}$, an explicit description seems difficult (for dimensional reasons, see Table \ref{table:two}). 
 However, we   prove that 

  \medskip
 \noindent {\bf Theorem D.} 
{\it Let $(M=G/K, g=-B|_{\fr{m}})$  be an  effective non-symmetric SII space of Table \ref{table:two}, whose isotropy representation $\chi$ is of real type and assume that $M$ is not isometric to $\SO_{10}/\Sp_{2}$. Then, the Ricci tensor associated to   the 1-parameter family of invariant metric connections with skew-torsion, orthogonal to the Lie bracket family $\nabla^{\al}$, is also symmetric. Moreover,
  \[
  \dim_{\bb{R}} {\cal{E}}^{sk}_{G}(\SO(G/K, -B|_{\fr{m}}))=2.
  \]}
  This result is based on    Theorem A.2  (Table \ref{table:two}) and the fact  $(\Lambda^{2}\fr{m})^{K}=0$ for real representations of real type.  This  means that the space of skew-symmetric 2-forms $\Lambda^{2}\fr{m}$ associated to a space $M=G/K$ of Theorem D (or even for the space $\SO_{10}/\Sp_{2}$), does not  contain the trivial representation, hence there do not exist  $G$-invariant 2-forms.  Consequently,  the co-differential of the torsion  associated to any existent $G$-invariant affine metric connection  on $M$ must vanish and our assertion follows by Schur's lemma in combination with the  expression of the Ricci tensor for a metric connection with skew-torsion.
 
Theorems C and D give the  complete  classification of all existent $G$-homogeneous $\nabla$-Einstein structures on any effective,   non-symmetric,  SII space $M=G/K$,  except of  the quotients $\SO_{n^{2}-1}/\SU_{n}$ $(n\geq 4)$ and  $\E_6/\SU_{3}$.   These are privileged manifolds with respect to  Theorem A.2; the associated  space of  $G$-invariant metric connections with skew-torsion is  4-dimensional.  Moreover, they both admit  an invariant almost complex structure and hence   $\Lambda^{2}(\fr{m})$  contains a copy of the trivial representation $\bb{R}$ (Lemma \ref{complex}), i.e. there exist  $G$-invariant (global) 2-forms.   However,  since we  are interested only on the symmetric part of $\Ric^{\nabla}$ and  the isotropy representation is (strongly)  irreducible, again a combination of  the results of Theorem A.2 with 
  Schur's lemma, yields that

\medskip
 \noindent {\bf Theorem E.}
 Let $(M=G/K, g=-B|_{\fr{m}})$ be one of the manifolds $\SO_{n^{2}-1}/\SU_{n}$ $(n\geq 4)$ or  $\E_6/\SU_{3}$. Then, the  space of $G$-homogeneous $\nabla$-Einstein structures with skew-torsion has dimension 
 \[
  \dim_{\bb{R}} {\cal{E}}^{sk}_{G}(\SO(G/K, -B|_{\fr{m}}))=4.
  \]
 

    \smallskip
 \noindent {\bf Acknowledgements:} 
The authors  thank Faculty of Science,  Department of Mathematics and Statistics  in Masaryk University  and FB12 at Philipps-Universit\"at Marburg for hospitality. The first author acknowledges the  institutional support of  University of Hradec Kr\'alov\'e. 

\section{Preliminaries}\label{I}
 \subsection{Metric connections and their  types}
  Consider a connected, oriented Riemannian  manifold $(M^{n}, g)$ and  identify  the tangent and cotangent bundle  $TM\cong T^{*}M$ via the  bundle isomorphism provided by the metric tensor $g$.  Any metric connection  $\nabla : \Gamma(TM)\to\Gamma(T^{*}M\otimes TM)\cong\Gamma(TM\otimes TM)$ on $M$ can be written as  $\nabla_{X}Y=\nabla^{g}_{X}Y+A(X, Y)$ for any $X, Y\in\Gamma(TM)$,
 for some tensor $A\in TM\otimes\Lambda^{2}(TM)$, where  $\nabla^{g}$ is the  Levi-Civita connection. Let us denote by   $A(X, Y, Z):=g(A(X, Y), Z)$ the induced tensor obtained by contraction with $g$.  The affine connections on $M$ which are compatible with $g$, form an affine space modelled on the sections of the tensor bundle
 \[
 {\cal A}:=\{A\in\otimes^{3}TM :   A(X, Y, Z)+A(X, Z, Y)=0\}\cong TM\otimes\Lambda^{2}(TM),
 \]
 which has fibre dimension $n^{2}(n-1)/2$. 
          It is   well-known  that ${\cal{A}}$ coincides with the space of torsion tensors 
          \[
          \cal{T}=\{A\in\otimes^{3}TM :   A(X, Y, Z)+A(X, Y, Z)=0\}\cong \Lambda^{2}(TM)\otimes TM.
          \]
       Moreover,    under the action of the structure group ${\rm SO}_{n}$  it decomposes  into three irreducible  representations  ${\cal{A}}={\cal{A}}_{1}\oplus{\cal{A}}_{2}\oplus{\cal{A}}_{3}$, defined by                  \begin{eqnarray*}
         {\cal{A}}_{1}&:=&\{A\in{\cal{A}} : A(X, Y, Z)=g(X, Y)\varphi(Z)-g(X, Z)\varphi(Y), \ \varphi\in \Gamma(T^{*}M)\}\cong TM,\\
                {\cal{A}}_{2}&:=&\{A\in{\cal{A}} :  \fr{S}^{X, Y, Z}A(X, Y, Z)=0, \Phi(A)=0\}, \\
                  {\cal{A}}_{3}&:=&\{A\in{\cal{A}} : A(X, Y, Z)+A(Y, X, Z)=0\}\cong \Lambda^{3}TM.
         \end{eqnarray*}
Here, the map    $\Phi : {\cal{A}}\to T^{*}M$  is given by $\Phi(A)(Z):={\rm tr}A_{Z}:=\sum_{i}A(e_{i}, e_{i}, Z)$, for a vector field $Z\in\Gamma(TM)$ and a  (local) orthonormal frame  $\{e_{i}\}$ of $M$. 
    The torsion $T(X, Y)=\nabla_{X}Y-\nabla_{Y}X-[X, Y]$ of $\nabla$ satisfies the relation $T(X, Y)=A(X, Y)-A(Y, X)$     and conversely, $A$ is expressed in terms of $T$ by the condition
    \begin{equation}\label{relay}
    2A(X, Y, Z)=T(X, Y, Z)-T(Y, Z, X)+T(Z, X, Y), \quad \forall X, Y, Z\in\Gamma(TM).
    \end{equation}   
        We say that $\nabla$   is of {\it vectorial type} (and the same   for  its torsion)  if  $A\in{\cal{A}}_{1}\cong TM$,  of {\it Cartan type, or traceless cyclic}   if  $A\in{\cal{A}}_{2}$ and finally  (totally) {\it skew-symmetric} (or, of skew-torsion) if $A\in{\cal{A}}_{3}\cong \Lambda^{3}TM$.  Notice that for $n=2$, ${\cal{A}}\cong\bb{R}^{2}$ is irreducible.  For $n\geq 3$, the mixed types occur by taking the direct sums of $\cal{A}_{1}, \cal{A}_{2}, \cal{A}_{3}$:
            \begin{eqnarray*}
            \cal{A}_{1}\oplus \cal{A}_{2}&=&\{A\in{\cal{A}} : \fr{S}^{X, Y, Z}A(X, Y, Z)=0\},\\
            \cal{A}_{2}\oplus \cal{A}_{3}&=&\{A\in{\cal{A}} : \Phi(A)=0\},\\
            \cal{A}_{1}\oplus \cal{A}_{3}&=&\{A\in{\cal{A}} : A(X, Y, Z)+A(Y, X, Z)=2g(X, Y)\varphi(Z)-g(X, Z)\varphi(Y)\\
            &&\quad\quad\quad\quad\quad\quad\quad\quad\quad\quad\quad\quad\quad\quad\quad\quad\quad-g(Y, Z)\varphi(X), \  \varphi\in \Gamma(T^{*}M)\}.
            \end{eqnarray*}
Usually,  connections of type $\cal{A}_{1}\oplus\cal{A}_{2}$ are called {\it cyclic} and connections of type $\cal{A}_{2}\oplus\cal{A}_{3}$ are known as {\it traceless} connections.  

Let us finally recall that a tensor field $A\in\cal{A}$ satisfying $\nabla A=0=\nabla R$,
where $R$ denotes the curvature of the metric connection $\nabla=\nabla^{g}+A$  is called a  {\it homogeneous structure}.   The existence of a metric connection with these properties implies that $(M, g)$ is locally homogeneous and  if in addition $(M,g)$ is complete, then it is locally isometric to a homogeneous Riemannian manifold. In particular, a complete,  connected and simply-connected Riemannian manifold $(M, g)$ endowed with a metric connection $\nabla$ solving the  equations $\nabla A=0=\nabla R$ is a homogeneous Riemannian manifold, see   \cite{Tric1} for more details and proofs. 
 \subsection{Connections with skew-torsion and $\nabla$-Einstein manifolds} 
   Let $(M^{n}, g)$ be a connected Riemannian manifold carrying a metric connection $\nabla$  with skew-torsion $0\neq T\in\Lambda^{3}(TM)$, i.e.
      \[
   g(\nabla_{X}Y, Z)=g(\nabla^{g}_{X}Y, Z)+\frac{1}{2}T(X, Y, Z).
   \]
 We normalize the length  of $T$ such that $\|T\|^{2}:=(1/6)\sum_{i, j}g(T(e_{i}, e_{j}), T(e_{i}, e_{j}))$  and we denote by $\delta^{\nabla}T=-\sum_{i=1}^{n}e_{i}\lrcorner \nabla_{e_{i}}T$ the co-differential of $T$. It is easy to check that $\delta^{g} T=\delta^{\nabla}T$. It is also known that (see for example \cite{IvPap, Dalakov, FrIv})
 \begin{lemma}
 The Ricci tensor   associated to $\nabla$ is given by
 \[
          \Ric^{\nabla}(X, Y)\equiv \Ric(X, Y)=\Ric^{g}(X, Y)-\frac{1}{4}\sum_{i=1}^{n}g(T(e_{i}, X), T(e_{i}, Y))-\frac{1}{2}(\delta^{g}T)(X, Y).
 \]
   \end{lemma}
          Thus, in contrast to the Riemannian Ricci tensor $\Ric^{g}$, the Ricci tensor   of $\nabla$ is not   symmetric;  it decomposes into a  symmetric  and antisymmetric part $\Ric=\Ric_{S}+\Ric_{A}$, given   by 
 \[
   \Ric_{S}(X, Y):=\Ric^{g}(X, Y)-\frac{1}{4}S(X, Y), \quad   \Ric_{A}(X, Y):=-\frac{1}{2}(\delta^{g}T)(X, Y),
\]
  respectively, where $S$ is the symmetric  tensor defined by $S(X, Y)=\sum_{i=1}^{n}g(T(e_{i}, X), T(e_{i}, Y))$.  
        \begin{definition}\textnormal{(\cite{AFer})}
A  triple $(M, g, T)$ is called  a {\it $\nabla$-Einstein manifold with non-trivial skew-torsion} $0\neq T\in\Lambda^{3}(TM)$, or for short, a {\it $\nabla$-Einstein manifold},  if the symmetric part $\Ric_{S}$ of the Ricci tensor associated to the metric connection $\nabla=\nabla^{g}+\frac{1}{2}T$ satisfies the equation 
    \begin{equation}\label{skein}
    \Ric_{S}=\frac{\Sca}{n}g,
    \end{equation}
   where  $\Sca\equiv\Sca^{\nabla}$ is the scalar curvature associated to $\nabla$ and $n=\dim_{\bb{R}}M$. If $\nabla T=0$, then  $(M, g, T)$ is called  a {\it $\nabla$-Einstein manifold with parallel skew-torsion}. 
    \end{definition}
     Notice   that in contrast to the Riemannian case, for  a $\nabla$-Einstein manifold  the scalar curvature   $\Sca^{\nabla}\equiv \Sca=\Sca^{g}-\frac{3}{2}\|T\|^{2}$  is not necessarily constant (for details see \cite{AFer}).  For parallel torsion, i.e. $\nabla T=0$,  one has $\delta^{\nabla}T=0$  and  the Ricci tensor  becomes symmetric  $\Ric=\Ric_{S}$. If in addition $\delta\Ric^{g}=0$, then the scalar curvature is  constant, similarly to  an Einstein manifold. This is the case for any $\nabla$-Einstein manifold $(M, g, \nabla, T)$ with {\it parallel skew-torsion} \cite[Prop.~2.7]{AFer}.




   \subsection{Invariant connections}
Consider a Lie group $G$  acting transitively on a smooth manifold  $M$  and   let us denote by $\pi : P\to M$    a   $G$-homogeneous  principal bundle  over $M$ with structure group $U$.  
 Let $K$ be the  isotropy subgroup at the point $o=\pi(p_{0})\in M$ with $p_{0}\in P$ (this is a closed subgroup $K\subset G$).  Then, there is a Lie group homomorphism  $\lambda : K\to U$ and hence an action  of $K$ on $U$, given by $ku=\lambda(k)u$. This induces a $G$-homogeneous  principal $U$-bundle  $P_{\lambda}\to M=G/K$, defined by $P_{\lambda}:=G\times_{K}U=G\times_{\lambda}U=G\times U/\sim$, where $(g, u)\sim(gk, \lambda(k^{-1})u)$ for any  $g\in G, u\in U, k\in K$.
Because the left action of $G$ on $P$ restricts to a left action of $K$ on the fiber $P_{o}$ of $P$ over a base point $o=eK\in G/K$, for the original bundle $P$  we have   $P\cong G\times_{K}P_{o}$.  But  fixing a point $u_{0}\in P_{o}$ we see that the map $U\to P_{o}$, $u\mapsto u_{0}u$ is a diffeomorphism and hence we identify $P\cong G\times_{K}P_{o}=G\times_{K}U=P_{\lambda}$, see also \cite{Cap, Laq1}.
    
 For $G$-homogeneous principal $U$-bundles $P\cong P_{\lambda}\to G/K$, it makes sense to speak about $G$-invariant connections, i.e.  connections for which the horizontal subspaces $\cal{H}_{p}$ are also  invariant by the left $G$-action,  $(L_{g})_{*}\cal{H}_{p}=\cal{H}_{gp}$ for any $g\in G$ and $p\in P$.  In other words,  a connection in $P_{\lambda}$ is $G$-invariant if and only if the  associated connection form $Z\in \Omega^{1} (P, \fr{u})$ is such that  $(\tau_{g}')^{*}Z=Z$, for all $g\in G$, where $\tau'_{g} : P\to P$  is the (right) $U$-equivariant bundle map. 
\begin{theorem}\textnormal{(\cite{Wang})}\label{W}
Let $P\cong P_{\lambda}\to G/K$ be a $G$-homogeneous principal $U$-bundle associated to a homomorphism $\lambda : K\to U$, as above.  Then, $G$-invariant connections on $P_{\lambda}$ are in a bijective correspondence with linear mappings $\Lambda : \fr{g}\to\fr{u}$ satisfying the following conditions:

(a) \ $\Lambda(X)=\lambda_{*}(X)$, for all $X\in\fr{k}=T_{e}K$, where $\lambda_{*} :\fr{k}\to\fr{u}$ is the differential of $\lambda$,

(b) \ $\Lambda(\Ad(k)X)=\Ad(\lambda(k))\Lambda(X)$, for all  $X\in\fr{g}=T_{e}G$, $k\in K$.  
\end{theorem}

 \subsection{Reductive homogeneous spaces}  Consider now a  reductive homogeneous  space $M=G/K$, i.e. we assume that there is an orthogonal decomposition $\fr{g}=\fr{k}\oplus\fr{m}$  of $\fr{g}=T_{e}G$  with $\Ad(K)\fr{m}\subset\fr{m}$.  Then we may identify $\fr{m}=T_{o}M$ at $o=eK\in M$ and  the  isotropy  representation $\chi : K\to \Aut(\fr{m})$ of $K$   with the restriction of the adjoint representation $\Ad|_{K}$ on $\fr{m}$.   Therefore, there is a direct sum  decomposition $ \Ad\big|_{K}=\Ad_{K}\oplus\chi$ 
  where  $\Ad_{K}$ is the adjoint representation of $K$.  As a  further consequence, we identify  the tangent bundle $TM$ and the frame bundle $F(M)$ of $M=G/K$   with the homogeneous vector bundle $G\times_{K}\fr{m}$ and     the  homogeneous principal bundle  $G\times_{K}\Gl(\fr{m})$, respectively, the latter with structure group  $\Gl(\fr{m})=\Gl_{n}\bb{R}$ $(n=\dim_{\bb{R}}\fr{m}=\dim_{\bb{R}}M)$.

   An invariant affine connection on $M=G/K$ is a principal connection on  $F(G/K)$ that is  $G$-invariant.
    By Theorem \ref{W} such an affine connection is described by a $\bb{R}$-linear map 
    $\Lambda : \fr{m}\to\fr{gl}(\fr{m})$ which  is equivariant under the isotropy representation,  i.e. $ \Lambda(\Ad(k)X)=\Ad(k)\Lambda(X)\Ad(k)^{-1}$ for any $X\in\fr{m}$ and $k\in K$. Let us denote by ${\rm Hom}_{K}(\fr{m}, \fr{gl}(\fr{m}))$   the set of such  linear maps. 
   The assignment  $\Lambda(X)Y=\eta(X, Y)$  provides an identification of ${\rm Hom}_{K}(\fr{m}, \fr{gl}(\fr{m}))$ (and hence of the space of $G$-invariant affine connections on $M=G/K$) with the set of all $\Ad(K)$-equivariant bilinear maps  $\eta : \fr{m}\times\fr{m}\to\fr{m}$, i.e. 
    \begin{equation}\label{equiv1}
    \eta(\Ad(k)X, \Ad(k)Y)=\Ad(k)\eta(X, Y),
    \end{equation}
     for any $X, Y\in\fr{m}$ and $k\in K$. Moreover,  since any such map $\eta : \fr{m}\times\fr{m}\to\fr{m}$ induces a unique linear map $\tilde{\eta} : \fr{m}\otimes\fr{m}\to\fr{m}$  with $\tilde{\eta}(X\otimes Y)=\eta(X, Y)$, one may further  identify (see \cite[Thm.~5.1]{Laq1})
     \[
 \cal{A}ff_{G}\big(F(G/K)\big)\cong {\rm Hom}_{K}(\fr{m}, \fr{gl}(\fr{m})) \cong {\rm Hom}_{K}(\fr{m}\otimes\fr{m}, \fr{m}),
 \]
 where  in general $\cal{A}ff_{G}(P)$  denotes the affine space of $G$-invariant affine connections on a homogeneous principal bundle $P\to G/K$ over $M=G/K$ and  ${\rm Hom}_{K}(\fr{m}\otimes\fr{m}, \fr{m})$  is the space of $K$-intertwining maps $\fr{m}\otimes\fr{m}\to\fr{m}$. Usually we shall work with $K$  connected and in this case we may  identify   ${\rm Hom}_{K}(\fr{m}\otimes\fr{m}, \fr{m})={\rm Hom}_{\fr{k}}(\fr{m}\otimes\fr{m}, \fr{m})$. Due to the orthogonal splitting $ \fr{m}\otimes\fr{m}=\Lambda^{2}\fr{m}\oplus \Sym^{2}\fr{m}$ we also remark that
\begin{equation}\label{hom1}
 {\rm Hom}_{K}(\fr{m}\otimes\fr{m}, \fr{m})={\rm Hom}_{K}(\Lambda^{2}\fr{m}, \fr{m})\oplus {\rm Hom}_{K}(\Sym^{2}, \fr{m}).
  \end{equation}
 
 The linear map $\Lambda : \fr{m}\to\fr{gl}(\fr{m})$ is usually called  {\it Nomizu map} or {\it connection map} (for details see \cite{AVL, Kob2}) and it satisfies the relation $\Lambda(X)=-(\nabla_{X}-L_{X})_{o}$, where   $L_{X}$ is the Lie derivative with respect to $X$.  Hence it encodes most of the properties of $\nabla$; for example, the  torsion $T\in\Lambda^{2}(\fr{m})\otimes\fr{m}$  and   curvature  $R\in \Lambda^{2}(\fr{m})\otimes\fr{k}$ of $\nabla$ are given by:
 \[T(X, Y)_{o}=\Lambda(X)Y-\Lambda(Y)X-[X, Y]_{\fr{m}}, \quad R(X, Y)_{o}= [\Lambda(X), \Lambda(Y)]-\Lambda([X, Y]_{\fr{m}})-\ad([X, Y]_{\fr{k}}).\]
 \begin{lemma}\label{torsionfree}
 Let $M=G/K$ be a homogeneous space with a reductive decomposition $\fr{g}=\fr{k}\oplus\fr{m}$. 
 Let $\Lambda, \Lambda'\in{\rm Hom}_{\fr{k}}(\fr{m}, \fr{gl}(\fr{m}))$ be two connection maps and let $\nabla, \nabla'\in  \cal{A}ff_{G}\big(F(G/K)\big)$ be the associated $G$-invariant affine connections. Set $\eta:=\Lambda-\Lambda'$. Then\\
 (i) $\nabla$ and $\nabla'$ have  the same geodesics if and only if $\eta\in {\rm Hom}_{K}(\Lambda^{2}\fr{m}, \fr{m})$.\\
 (ii) $\nabla$ and $\nabla'$ have the same torsion if and only if $\eta\in  {\rm Hom}_{K}(\Sym^{2}\fr{m}, \fr{m})$.
 \end{lemma}

Consider now   a   homogeneous Riemannian manifold $(M=G/K, g)$. In this case $G$ can be considered as a closed subgroup of the full isometry group $\Iso(M, g)$, which  implies  that $K$ and the Lie subgroup $\Ad(K)\subset\Ad(G)$  are compact subgroups.  Hence, there is always a reductive decomposition  $\fr{g}=\fr{k}\oplus\fr{m}$   with respect to some $\Ad(K)$-invariant inner product in the Lie algebra $\fr{g}$.  We shall denote by  $\langle \ , \ \rangle$   the  $\Ad(K)$-invariant inner product on $\fr{m}$ induced by    $g$. 
We  equivariantly  identify the $K$-modules   $\fr{so}(\fr{m}, g)\equiv \fr{so}(\fr{m})=\Lambda^{2}(\fr{m})$ via the isomorphism $X\wedge Y\mapsto \langle X, \cdot \rangle Y-\langle Y, \cdot \rangle X$, for any $ X, Y\in\fr{m}$. Consider the $\SO(\fr{m})$-principal bundle $\SO(G/K)\to G/K$   of $\langle \ , \ \rangle$-orthonormal frames.  This is a homogeneous principal bundle  and  an invariant metric connection on $M=G/K$ is  a principal connection on  $\SO(G/K)$ that is  $G$-invariant. 
It follows that
\begin{lemma}
A $G$-invariant affine connection  $\nabla$ on $(M=G/K, g)$ preserves  the $G$-invariant Riemannian metric $g$  if and only if the associated  Nomizu map satisfies $\Lambda(X)\in\fr{so}(\fr{m}, g)$ for any $X\in\fr{m}$. 
\end{lemma}
 Notice that  the existence of an invariant metric means that the isotropy representation of $M=G/K$ is self-dual, $\fr{m}\simeq \fr{m}^\ast$. Thus we may equivariantly identify 
\[
\fr{gl}(\fr{m})\simeq \Ed(\fr{m})\simeq \fr{m} \otimes \fr{m}, \quad 
{\rm Hom}_{K}(\fr{m},\Ed(\fr{m}))=(\fr{m}^\ast \otimes \fr{m}^\ast \otimes \fr{m})^{K}\simeq (\otimes^3\fr{m})^{K}.
\]
 In the last case, a $K$-equivariant map $\Lambda$ on the left hand side is equivalent to a $K$-invariant tensor on the right hand side: ${\rm Hom}_{K}(\fr{m},\Ed(\fr{m}))= (\otimes^3\fr{m})^{K}$. The latter space has the following obvious $K$-submodules: $\Lambda^2 \fr{m} \otimes \fr{m}$, $\Sym^2\fr{m} \otimes \fr{m}$, $\fr{m} \otimes \Sym^2\fr{m}$ and $\fr{m} \otimes \Lambda^2\fr{m}$.
Of these, the last space corresponds to the $\fr{so}(\fr{m})$-valued Nomizu maps, i.e. the space of homogeneous metric connections which we denote by ${\cal {M}}_{G}(\SO(G/K))$. In particular, there is an equivariant isomorphism \[{\cal {M}}_{G}(\SO(G/K, g))\cong \Hom_{K}(\fr{m}, \Lambda^{2}\fr{m}).\]

\begin{remark}\label{henrik}
\textnormal{The other submodules  have different interpretations. For example, $\Sym^2\fr{m} \otimes \fr{m}$ is the vector space on which the affine space of invariant torsion-free connections $\cal{A}ff_{G}^{0}(F(G/K))$ is modelled, and $\Lambda^2\fr{m} \otimes \fr{m}$ is the vector space on which the affine space of possible invariant torsion tensors is modelled.  In fact, since the   rearrangement of indices is equivariant (even with respect to the bigger algebra $\fr{gl}(\fr{m})$),  one has the following isomorphisms: $\Lambda^2 \fr{m} \otimes \fr{m} \simeq \fr{m} \otimes \Lambda^2\fr{m}$ and $\Sym^2 \fr{m} \otimes \fr{m} \simeq \fr{m} \otimes \Sym^2\fr{m}$. Let us now relate this to the question of multiplicities of $\fr{m}$ inside $\otimes^2\fr{m}=\Ed(\fr{m})$. Suppose we have a copy of $\fr{m}$   inside  the invariant decomposition of $\Lambda^2\fr{m}$ (or respectively, in $\Sym^2\fr{m}$). This is equivalent to a map $\theta:\fr{m}\to  \Lambda^2\fr{m}$ (respectively  $\fr{m}\to\Sym^2\fr{m}$). We may then raise all indices of $\theta$ to produce a $K$-invariant element of $\otimes^3\fr{m}$. However through our freedom to rearrange indices, we may change to which of our four submodules this tensor belongs. For example, one may interpret the tensor corresponding to the instance of $\fr{m}$ in $\Lambda^2\fr{m}$ either as a metric connection in $\fr{m} \otimes \Lambda^2\fr{m}$, or a potentially non-metric connection in $\Lambda^2 \fr{m} \otimes \fr{m}$. These coincide up to a scalar when $\theta\in\Lambda^3\fr{m}$.}
\end{remark}

On a homogeneous Riemannian manifold $(M=G/K, g)$ the Levi-Civita connection $\nabla^{g}$ is the unique   G-invariant metric connection determined  by (cf. \cite{Bes, Nik})
\[
 \langle \nabla^{g}_{X}Y , Z\rangle = -\frac{1}{2} \big[ \langle [ X , Y ]_{\fr{m}} , Z \rangle + \langle  [ Y , Z ]_{\fr{m}} , X \rangle - \langle [ Z , X ]_{\fr{m}} , Y \rangle\big], \quad \forall \  X, Y, Z\in\fr{m}.\quad (\star)
 \] 
 On the other hand, the {\it canonical connection} on $M=G/K$  is induced by the   principal $K$-bundle $G\to G/K$ and     depends on  the choice of the reductive complement $\fr{m}$. It is defined by the horizontal distribution   $\{\cal{H}_{g}:=d\ell_{g}(\fr{m}) : g\in G\}$, where $\ell_{g}$ denotes the left translation on $G$ and its Nomizu map is given by $ \Lambda^{c} : \fr{g}=\fr{k}\oplus\fr{m}\xrightarrow{{\rm pr}_{\fr{k}}}\fr{k}\xrightarrow{\chi_{*}}\fr{so}(\fr{m})$,  i.e.  $\Lambda^{c}=\chi_{*}\circ{\rm pr}_{\fr{k}}$. Thus,    $\Lambda^{c}(X)=0$  for any $X\in\fr{m}$ (cf. \cite{Kob2, AVL}).  Both the torsion $T^{c}(X, Y)=-[X, Y]_{\fr{m}}$  
 and the  curvature   $R^{c}(X, Y)=-\ad([X, Y]_{\fr{k}})$ of $\nabla^{c}$ are parallel objects,  in particular  any $G$-invariant tensor field on $M=G/K$ is $\nabla^{c}$-parallel  (cf. \cite{N, Kob2}). Hence, any homogeneous Riemannian manifold $(M=G/K, g)$ admits a homogeneous structure $A^{c}\in\fr{m}\otimes\Lambda^{2}\fr{m}\cong \cal{A}$ induced by the canonical connection $\nabla^{c}$ associated  to the reductive  decomposition $\fr{g}=\fr{k}\oplus\fr{m}$. In the following, we shall refer to this homogeneous structure   as the {\it canonical homogeneous structure}, adapted to $\fr{m}$ and $G$.  Using $(\star)$ it is easy to see that    $A^{c}:=\nabla^{c}-\nabla^{g}$ satisfies the relation
 \begin{equation}\label{can} 
 A^{c}(X, Y, Z)=\frac{1}{2}T^{c}(X, Y, Z)-\langle U(X, Y ), Z\rangle, \quad \forall \ X, Y, Z\in\fr{m},
 \end{equation}
where   $U : \fr{m}\times\fr{m}\to\fr{m}$ is the symmetric bilinear mapping defined by 
 \begin{equation}\label{uu}
 2\langle U(X, Y ), Z\rangle = \langle [Z, X]_{\fr{m}}, Y \rangle + \langle X, [Z, Y ]_{\fr{m}}\rangle.
 \end{equation}

  \section{Invariant connections and derivations}\label{derivat} 
 Given a reductive homogeneous space $M=G/K$ endowed with a $G$-invariant affine connection $\nabla$,    in the following we examine  $\Ad(K)$-equivariant derivations on $\fr{m}$ induced by $\nabla$ in terms of Nomizu maps. For the case of a compact Lie group $G$, this problem has been analyzed in \cite{Chrysk}. 
 \subsection{Derivations and   generalized derivations}  For the following of this section let us fix a (connected) homogeneous manifold $M=G/K$ with a reductive decomposition $\fr{g}=\fr{k}\oplus\fr{m}$. For simplicity we  assume that the transitive $G$-action is effective. We   consider a   bilinear mapping $\mu : \fr{m}\otimes\fr{m}\to\fr{m}$ and denote by   $\Lambda : \fr{m}\to\fr{gl}(\fr{m})$     the adjoint map, defined by $\Lambda(X)Y=\mu(X, Y)$.  
  \begin{definition}
  \textnormal{The endomorphism $\Lambda(Z) : \fr{m}\to\fr{m}$ $(Z\in\fr{m})$  is called a {\it derivation} of $\fr{m}$, with respect to the Lie bracket operation $\ad_{\fr{m}}:=[ \ , \ ]_{\fr{m}} : \fr{m}\times\fr{m}\to\fr{m}$, $\ad_{\fr{m}}(X, Y):=[X, Y]_{\fr{m}}$, if and only if $\fr{der}^{\mu}(X, Y; Z)=0$ identically, where for any $X, Y, Z\in\fr{m}$ we set
 \begin{eqnarray*}
\fr{der}^{\mu}(X, Y; Z)&:=&\Lambda(Z)[X, Y]_{\fr{m}}-[\Lambda(Z)X, Y]_{\fr{m}}-[X, \Lambda(Z)Y]_{\fr{m}}\\
&=&\mu(Z, [X, Y]_{\fr{m}})-[\mu(Z, X), Y]_{\fr{m}}-[X, \mu(Z, Y)]_{\fr{m}}.
 \end{eqnarray*}}
 \end{definition}
\noindent From now on, let us denote by ${\rm Der}(\ad_{\fr{m}}; \fr{m})\equiv{\rm Der}(\fr{m})$ the  vector space of all derivations on $\fr{m}$.   We mention that given a bilinear map $\mu : \fr{m}\otimes\fr{m}\to\fr{m}$, the condition  $\mu\in{\rm Der}(\fr{m})$ is equivalent to  say that the associated connection map $\Lambda$ is valued in ${\rm Der}(\fr{m})$, i.e. $\Lambda\in{\rm Hom}(\fr{m}, {\rm Der}(\fr{m}))$.   Restricting on $K$-intertwining maps $\mu\in{\rm Hom}_{K}(\fr{m}\otimes\fr{m}, \fr{m})$ the vector space ${\rm Der}(\fr{m})$ becomes a $K$-module,    denoted  by ${\rm Der}_{K}(\fr{m})$. In fact, in this case we shall  speak about $\Ad(K)$-equivariant derivations on $\fr{m}$. So, let us focus on $\Ad(K)$-equivariant derivations induced by invariant connections on $M=G/K$.
\begin{prop}\label{derin}
Let $\nabla\equiv\nabla^{\mu}$ be a $G$-invariant connection on $M=G/K$ corresponding to    $\mu\in {\rm Hom}_{K}(\fr{m}\otimes\fr{m}, \fr{m})$.  Then, $\nabla^{\mu}$ induces a  $\Ad(K)$-equivariant derivation $\mu\in {\rm Der}_{K}(\fr{m})$,  if and only if $\ad_{\fr{m}}:=[ \ , \ ]_{\fr{m}}$  is $\nabla^{\mu}$-parallel, i.e. $\nabla^{\mu}\ad_{\fr{m}}=0$ (which is equivalent to say that the torsion $T^{c}$ of the canonical connection $\nabla^{c}$ associated to the reductive complement $\fr{m}$ is $\nabla^{\mu}$-parallel, i.e. $\nabla^{\mu} T^{c}=0$). 
\end{prop}
\begin{proof}
The equivalence   $\mu\in {\rm Der}(\ad_{\fr{m}}; \fr{m})\equiv{\rm Der}(\fr{m}) \ \Leftrightarrow \ \nabla^{\mu}\ad_{\fr{m}}\equiv 0$ is an immediate consequence of the identity
 \begin{equation}\label{cantor}
\fr{der}^{\mu}(X, Y; Z)=(\nabla^{\mu}_{Z}\ad_{\fr{m}})(X, Y)=-(\nabla^{\mu}_{Z}T^{c})(X, Y), \quad \forall \ X, Y, Z\in\fr{m}. 
\end{equation}
The proof of (\ref{cantor}) relies  on the fact that $G$-invariant tensor fields are $\nabla^{c}$-parallel, where $\nabla^{c}$ is the canonical connection associated to $\fr{m}$. In particular, since $\nabla$ is a $G$-invariant connection we write $\nabla^{\mu}_{Z}=\nabla^{c}_{Z}+\Lambda(Z)$, for any $Z\in\fr{m}$, where $\Lambda : \fr{m}\to\fr{gl}(\fr{m})$ is the associated Nomizu map.  Then,  for any $X, Y, Z\in\fr{m}$ we obtain that
\begin{eqnarray*}
(\nabla_{Z}^{\mu}\ad_{\fr{m}})(X, Y)&=&\nabla^{\mu}_{Z}\ad_{\fr{m}}(X, Y)-\ad_{\fr{m}}(\nabla^{\mu}_{Z}X, Y)-\ad_{\fr{m}}(X, \nabla_{Z}^{\mu}Y) \\
&=&\big[\nabla^{c}_{Z}\ad_{\fr{m}}(X, Y)-\ad_{\fr{m}}(\nabla^{c}_{Z}X, Y)-\ad_{\fr{m}}(X, \nabla_{Z}^{c}Y)\big] \\
&&+\big[\La(Z)\ad_{\fr{m}}(X, Y)-\ad_{\fr{m}}(\La(Z)X, Y)-\ad_{\fr{m}}(X, \La(Z)Y)\big]\\
&=&(\nabla_{Z}^{c}\ad_{\fr{m}})(X, Y)+\fr{der}^{\mu}(X, Y; Z)=\fr{der}^{\mu}(X, Y; Z),
\end{eqnarray*}
 where the last equality follows since  $\nabla^{c}\ad_{\fr{m}}\equiv 0$ .  Similarly for the second equality in (\ref{cantor}).
\end{proof}
\begin{example}
\textnormal{The canonical connection $\nabla^{c}$ associated to the reductive complement $\fr{m}$ induces a derivation on $\fr{m}$ (the zero one,   corresponding  to $0\in {\rm Der}_{K}(\fr{m})$), since $\nabla^{c}T^{c}=0$, or in other words since  $T^{c}$ is $\nabla^{\mu}$-parallel, where  $\mu=0\in {\rm Hom}_{K}(\fr{m}\otimes\fr{m}, \fr{m})$. }
\end{example}

Let us now generalize the notion of derivations on $\fr{m}$, as follows:
 \begin{definition}\label{generalderiv} 
 Consider a tensor  $F  : \otimes^{p}\fr{m} \to\fr{m}$.  Then, a bilinear mapping    $\mu : \fr{m} \otimes\fr{m}\to\fr{m}$  is said to be a {\it generalized derivation} of $F$ on $\fr{m}$,    if and only if    $\mu$  satisfies the relation 
 \begin{eqnarray*}
 \mu(Z, F(X_{1}, \ldots, X_{p}))&=&F(\mu(Z, X_{1}), X_{2}, \ldots, X_{p})+\cdots+F(X_{1}, \ldots, X_{p-1}, \mu(Z, X_{p})) \ \Leftrightarrow \\
 \Lambda(Z)F(X_{1}, \dots, X_{p})&=&F(\Lambda(Z)X_{1}, X_{2}, \ldots, X_{p})+\cdots+F(X_{1}, \ldots, X_{p-1}, \Lambda(Z)X_{p}),
 \end{eqnarray*}
     for any $Z, X_{1}, \ldots, X_{p}\in\fr{m}$, where $\Lambda\in{\rm Hom}(\fr{m}, \fr{gl}(\fr{m}))$ is the adjoint map induced by $\mu$.
       \end{definition}
  For a  tensor $F  : \otimes^{p}\fr{m} \to\fr{m}$, the definition of a generalized derivation implies that if $\mu_{1}, \mu_{2} : \fr{m}\otimes\fr{m}\to\fr{m}$ are two such bilinear mappings, then the linear combination $a\mu_{1}+b\mu_{2}$ is also  a generalized derivation of $F$ on $\fr{m}$.  Hence, the set  ${\rm Der}(F; \fr{m})$ of all generalized derivations of $F$ on $\fr{m}$  is a vector space.  Obviously, for  $F=\ad_{\fr{m}}$, a  generalized derivation is just a classical   derivation  on $\fr{m}$.  Notice however that $F$ can be much more general than the Lie bracket restriction, e.g.  the torsion, or the curvature of a $G$-invariant connection $\nabla$ on $M=G/K$ induced by  some $\mu\in{\rm Hom}_{K}(\fr{m}\otimes\fr{m}, \fr{m})$, or  even $\mu$ itself.  In particular, one may restrict   Definition \ref{generalderiv} on $K$-intertwining maps $\mu\in{\rm Hom}_{K}(\fr{m}\otimes\fr{m}, \fr{m})$; then, the space  ${\rm Der}(F; \fr{m})$ becomes a $K$-module, which we shall denote by ${\rm Der}_{K}(F; \fr{m})$. If moreover we focus on   $G$-invariant tensor fields, then similarly to Proposition \ref{derin} we conclude that

\begin{theorem}\label{mdiffer} 
Let $(M=G/K, \fr{g}=\fr{k}\oplus\fr{m})$ be a reductive homogeneous space  endowed with an $\Ad(K)$-invariant tensor $F  : \otimes^{p}\fr{m} \to\fr{m}$.  Consider a $K$-intertwining map $\mu\in {\rm Hom}_{K}(\fr{m}\otimes\fr{m}, \fr{m})$ and let us denote by $\nabla^{\mu}$ the associated $G$-invariant affine connection. Then, $\mu$ is an $\Ad(K)$-equivariant generalized derivation of $F$  if and only if $F$ is $\nabla^{\mu}$-parallel, i.e. $\mu\in {\rm Der}_{K}(F; \fr{m}) \  \Leftrightarrow  \ \nabla^{\mu}F\equiv 0$.  \end{theorem}
 \begin{proof}
 A direct computation shows that  the  evaluation of  the covariant differentiation $\nabla F$  at the point  $o=eK\in G/K$ gives rise to the following $\Ad(K)$-invariant tensor on $\fr{m}$:
 \begin{eqnarray*}
(\nabla_{Z}F)(X_{1}, \ldots, X_{p})&=&\nabla_{Z}F(X_{1}, \ldots, X_{p})-\sum_{i=1}^{p}F(X_{1}, \ldots, \nabla_{Z}X_{i}, \ldots, X_{p})\\
&=&\nabla^{c}_{Z}F(X_{1}, \ldots, X_{p})+\Lambda(Z)F(X_{1}, \ldots, X_{p})-\sum_{i=1}^{p}F(X_{1}, \ldots, \nabla_{Z}^{c}X_{i}, \ldots, X_{p})\\
&&-\sum_{i=1}^{p}F(X_{1}, \ldots, \Lambda(Z)X_{i}, \ldots, X_{p})\\
&=&\nabla^{c}_{Z}F(X_{1}, \ldots, X_{p})-\sum_{i=1}^{p}F(X_{1}, \ldots, \nabla_{Z}^{c}X_{i}, \ldots, X_{p})\\
&&+\Lambda(Z)F(X_{1}, \ldots, X_{p})-\sum_{i=1}^{p}F(X_{1}, \ldots, \Lambda(Z)X_{i}, \ldots, X_{p})\\
&=&(\nabla^{c}_{Z}F)(X_{1}, \ldots, X_{p})+(\cal{D}^{\mu}_{Z}F)(X_{1}, \ldots, X_{p}), 
\end{eqnarray*}
where we set $(\cal{D}^{\mu}_{Z}F)(X_{1}, \ldots, X_{p}):=\Lambda(Z)F(X_{1}, \ldots, X_{p})-\sum_{i=1}^{p}F(X_{1}, \ldots, \Lambda(Z)X_{i}, \ldots, X_{p})$. However,  $F$ is by assumption $G$-invariant, hence $\nabla^{c}F=0$ and our claim immediatelly follows. 
\end{proof}
Moreover, we conclude that

\begin{corol}\label{corder}
On a reductive homogeneous space $(M=G/K, \fr{g}=\fr{k}\oplus\fr{m})$, given an $\Ad(K)$-invariant tensor $F  : \otimes^{p}\fr{m} \to\fr{m}$ and some  $K$-intertwining map $\mu\in {\rm Hom}_{K}(\fr{m}\otimes\fr{m}, \fr{m})$,   the operation
\[
(\cal{D}^{\mu}_{Z}F)(X_{1}, \ldots, X_{p}):=\Lambda(Z)F(X_{1}, \ldots, X_{p})-\sum_{i=1}^{p}F(X_{1}, \ldots, \Lambda(Z)X_{i}, \ldots, X_{p})
\]
 coincides with the covariant differentiation of $F$ with respect to the    connection $\nabla=\nabla^{\mu}$ induced on $M=G/K$ by $\mu$, i.e. $(\nabla^{\mu}_{Z}F)(X_{1}, \ldots, X_{p})=(\cal{D}^{\mu}_{Z}F)(X_{1}, \ldots, X_{p})$ for any $X_{1}, \ldots, X_{p}, Z\in\fr{m}$.
\end{corol}

 For a  bilinear mapping $\mu : \fr{m}\otimes\fr{m}\to\fr{m}$ let us now introduce   the tensor ${\cal{C}}^{\mu}$, defined by 
\[
{\cal{C}}^{\mu}(X, Y; Z):=(\nabla_{Z}^{\mu}\mu)(X, Y)-(\nabla^{\mu}_{Z}\mu)(Y, X),
\]
 for any $X, Y, Z\in\fr{m}$.  If $\mu\in{\rm Hom}_{K}(\fr{m}\otimes\fr{m}, \fr{m})$, then we  get the further identification ${\cal{C}}^{\mu}(X, Y; Z):=({\cal{D}}_{Z}^{\mu}\mu)(X, Y)-(\cal{D}^{\mu}_{Z}\mu)(Y, X)$. 
  In terms of ${\cal{C}}^{\mu}$ we obtain that
\begin{prop}\label{input1}
Let  $\nabla=\nabla^{\mu}$ be a $G$-invariant affine connection  on a reductive homogeneous space $(M=G/K, \fr{g}=\fr{k}\oplus\fr{m})$, corresponding to some  $\mu\in {\rm Hom}_{K}(\fr{m}\otimes\fr{m}, \fr{m})$. Then,   $\mu\in{\rm Der}_{K}(\fr{m})$, if and only if 
\[
(\nabla_{Z}T)(X, Y)\equiv (\cal{D}^{\mu}_{Z}T)(X, Y)={\cal{C}}^{\mu}(X, Y; Z),\quad \forall \ X, Y, Z\in\fr{m},
\]
where  $T=T^{\mu}$ is the torsion  associated to $\nabla^{\mu}$.
\end{prop}
 
\begin{proof}
 As in the proof of Theorem  \ref{mdiffer},  we easily get that
 \begin{eqnarray}
(\nabla_{Z}T)(X, Y)&=&({\cal{D}}^{\mu}_{Z}T)(X, Y)=\mu(Z, T(X, Y))-T(\mu(Z, X), Y)-T(X, \mu(Z, Y)) \label{ntors}
\end{eqnarray}
 for any $X, Y, Z\in\fr{m}$. We will show now  that the left hand side reduces to  $({\cal{D}}^{\mu}_{Z}T^{c})(X, Y)+{\cal{C}}^{\mu}(X, Y; Z)$. For this,  notice first that 
\begin{eqnarray*}
(\nabla_{Z}T)(X, Y)&=&\mu(Z, \mu(X, Y))-\mu(Z, \mu(Y, X))-\mu(\mu(Z, X), Y)+\mu(Y, \mu(Z, X))\\
&&-\mu(X, \mu(Z, Y))+\mu(\mu(Z, Y), X)-\fr{der}_{\fr{m}}(X, Y; Z).
\end{eqnarray*}
An easy computation also gives that
\begin{eqnarray*}
({\cal{D}}_{Z}^{\mu}\mu)(X, Y)-(\cal{D}^{\mu}_{Z}\mu)(Y, X)&=&\mu(Z, \mu(X, Y))-\mu(Z, \mu(Y, X))-\mu(\mu(Z, X), Y)+\mu(Y, \mu(Z, X))\\
&&-\mu(X, \mu(Z, Y))+\mu(\mu(Z, Y), X).
\end{eqnarray*}
Hence $(\cal{D}^{\mu}_{Z}T)(X, Y)=({\cal{D}}^{\mu}_{Z}T^{c})(X, Y)+{\cal{C}}^{\mu}(X, Y; Z)$ and  in combination with (\ref{cantor}) one   can easily finish the proof.
\end{proof}
Consequently, for some $\mu\in{\rm Hom}_{K}(\fr{m}\otimes\fr{m}, \fr{m})$ the  condition $\mu\in {\rm Der}_{K}(\fr{m})$  can also be read in terms of the  $\Ad(K)$-invariant  tensor ${\cal{C}^{\mu}}$,   which  geometrically,     represents the difference 
\[
(\nabla_{Z}T)(X, Y)-(\nabla_{Z}T^{c})(X, Y)\equiv ({\cal{D}}^{\mu}_{Z}T)(X, Y)-({\cal{D}}^{\mu}_{Z}T^{c})(X, Y),
\]
 for any $X, Y, Z\in\fr{m}$. 
 In particular, a combination of   Proposition \ref{input1} and identity (\ref{hom1}),  yields that

\begin{theorem}\label{deriv2} 
Let $M=G/K$ an effective homogeneous space with a reductive decomposition $\fr{g}=\fr{k}\oplus\fr{m}$. Then the following hold:\\
(1) A $G$-invariant affine connection $\nabla=\nabla^{\mu}$ on $M=G/K$ corresponding to   $\mu\in{\rm Hom}_{K}(\Lambda^{2}\fr{m}, \fr{m})$, induces an $\Ad(K)$-equivariant derivation   $\mu\in{\rm Der}_{K}(\fr{m})$, if and only if 
\begin{equation}\label{formula1}
(\nabla^{\mu}_{Z}T)(X, Y)\equiv (\cal{D}^{\mu}_{Z}T)(X, Y)=2\fr{S}_{X, Y, Z}\mu(X, \mu(Y, Z)),
\end{equation}
for any $X, Y, Z\in\fr{m}$. This is equivalent to say that
\begin{equation}\label{formula2}
(\nabla^{\mu}_{Z}T)(X, Y)\equiv ({\cal{D}}^{\mu}_{Z}T)(X, Y)=2\big\{R(Z, X)Y+\Lambda(Y)(\Lambda(Z)X-[Z, X]_{\fr{m}})+\ad([Z, X]_{\fr{k}})Y\},
\end{equation}
where $R$ is the curvature tensor associated to $\nabla$. \\
(2) A $G$-invariant affine connection $\nabla=\nabla^{\mu}$ on $M=G/K$ corresponding to  $\mu\in{\rm Hom}_{K}(\Sym^{2}\fr{m}, \fr{m})$, induces an $\Ad(K)$-equivariant derivation   $\mu\in{\rm Der}_{K}(\fr{m})$ on $\fr{m}$ if and only if the torsion $T^{\mu}$ associated to $\nabla^{\mu}$ is   $\nabla^{\mu}$-parallel.  \\
(3)   Let $\mu\in{\rm Hom}_{K}(\Lambda^{2}\fr{m}, \fr{m})$. Then $\mu$ is an $\Ad(K)$-equivariant  generalized derivation of itself, i.e. $\mu\in{\rm Der}_{K}(\mu; \fr{m})$ if and only if $\cal{C}^{\mu}=0$ identically.
\end{theorem}
\begin{proof}
  For a skew-symmetric mapping $\mu\in {\rm Hom}_{K}(\Lambda^{2}\fr{m}, \fr{m})$ a simple computation gives  that
  \[
  \cal{C}^{\mu}(X, Y; Z)=2\fr{S}_{X, Y, Z}\mu(X, \mu(Y, Z)).
  \]
   Hence,  (\ref{formula1}) is an immediate consequence of  Proposition  \ref{input1}. For the second relation (\ref{formula2}), using the definition of the  curvature tensor $R$ and (\ref{ntors}), for some $\mu\in{\rm Hom}_{K}(\Lambda^{2}\fr{m}, \fr{m})$ we get that  
\begin{eqnarray}
(\nabla_{Z}T)(X, Y)=({\cal{D}}^{\mu}_{Z}T)(X, Y)&=&2R(Z, X)Y+2\Lambda(Y)(\Lambda(Z)X-[Z, X]_{\fr{m}})\nonumber\\
&&+2\ad([Z, X]_{\fr{k}})Y-\fr{der}^{\mu}(X, Y; Z),   \label{chr1} 
\end{eqnarray}
for any $X, Y, Z\in\fr{m}$ and our claim immediately follows. \\
  For the second statement and for a symmetric map $\mu\in{\rm Hom}_{K}(\Sym^{2}\fr{m}, \fr{m})$ it is easy to see that  $\cal{C}^{\mu}=0$. Therefore,  our assertion is a direct consequence of  Proposition \ref{input1}.  \\
Let us finally  prove (3).  By definition, it is $\mu\in{\rm Der}_{K}(\mu; \fr{m})$ or equivalent  $\Lambda\in{\rm Hom}_{K}(\fr{m}, {\rm Der}_{K}(\mu; \fr{m}))$,  if and only if
\[
\mu(Z, \mu(X, Y))=\mu(\mu(Z, X), Y)+\mu(X, \mu(Z, Y))
\]
for any $X, Y, Z\in\fr{m}$, which is equivalent to say that $\fr{S}_{X, Y, Z}\mu(X, \mu(Y, Z))=0$ identically. But since $\cal{C}^{\mu}(X, Y; Z)=2\fr{S}_{X, Y, Z}\mu(X, \mu(Y, Z))$, we conclude.     
\end{proof}

\begin{remark}\textnormal{For   a compact connected Lie group $G\cong (G\times G)/\Delta G$ endowed with a bi-invariant affine connection $\nabla$ corresponding to a skew-symmetric  mapping $\mu\in{\rm Hom}_{G}(\La^{2}\fr{g}, \fr{g})$, formula (\ref{formula2}) has been described  in  \cite[Prop.~2.4]{Chrysk}. In particular, in this case  relation (\ref{chr1})  reduces to
\[
(\nabla_{Z}T)(X, Y)=2R(Z, X)Y+2\Lambda(Y)(\Lambda(Z)X-[Z, X]_{\fr{m}})-\fr{der}_{\fr{g}}(X, Y; Z),
\]
for any $X, Y, Z\in\fr{g}=T_{e}G$, see also \cite[Prop.~2.4]{Chrysk}.}
\end{remark}

\begin{example}\label{deram} 
\textnormal{Let  $M=G/K$ an effective homogeneous space with a reductive decomposition $\fr{g}=\fr{k}\oplus\fr{m}$.  We consider the restricted Lie bracket $\ad_{\fr{m}}:=[ - , - ]_{\fr{m}} : \fr{m}\times\fr{m}\to\fr{m}$ and   denote the associated Nomizu map just by $\Lambda_{\fr{m}}$.  Obviously,  {\it $\ad_{\fr{m}}$ induces a derivation on $\fr{m}$ if and only if $\Jac_{\fr{m}}\equiv 0$}, where $\Jac_{\fr{m}} : \fr{m}\times\fr{m}\times\fr{m}\to\fr{m}$ is the trilinear map defined by
\[
\Jac_{\fr{m}}(X, Y, Z):=\fr{S}_{X, Y, Z}[X, [Y, Z]_{\fr{m}}]_{\fr{m}}=[X, [Y, Z]_{\fr{m}}]_{\fr{m}}+[Y, [Z, X]_{\fr{m}}]_{\fr{m}}+[Z, [X, Y]_{\fr{m}}]_{\fr{m}},
\] 
for any $ X, Y, Z \in\fr{m}$.
The  same conclusion follows from Theorem \ref{deriv2}.   Indeed, let us denote by $\nabla^{\fr{m}}$ the $G$-invariant connection associated to $\ad_{\fr{m}}$ and by $T^{\fr{m}}$ and $R^{\fr{m}}$ its torsion and curvature, respectively.   It is  $T^{\fr{m}}(X, Y)=[X, Y]_{\fr{m}}$ and  
\[
(\nabla^{\fr{m}}_{Z}T^{\fr{m}})(X, Y)=({\cal{D}}^{\al_{\fr{m}}}_{Z}T^{\fr{m}})(X, Y)=\Jac_{\fr{m}}(X, Y, Z),
\]
 for any $X, Y, Z\in\fr{m}$.  Moreover, $R^{\fr{m}}(Z, X)Y=\Jac_{\fr{m}}(X, Y, Z)-[[Z, X]_{\fr{k}}, Y]$ and since $\Lambda_{\fr{m}}(Z)X=[Z, X]_{\fr{m}}$,  an application of  Theorem \ref{deriv2}, (1), shows  that $\Lambda_{\fr{m}}\in {\rm Hom}_{K}(\fr{m}, {\rm Der}(\fr{m}))$ if and only if $\Jac_{\fr{m}}(X, Y, Z)=0$ for any $X, Y, Z\in\fr{m}$.  In fact, for $\mu=\ad_{\fr{m}}$  it is  $\cal{C}^{\ad_{\fr{m}}}(X, Y; Z)=2\Jac_{\fr{m}}(X, Y, Z)$, hence the same results follows by  relation (\ref{formula1}). Finally, for the same assertion one can even apply Theorem \ref{deriv2}, (3) for $\mu=\ad_{\fr{m}}$.}
   
\textnormal{Note that if $M=G/K$ is an effective symmetric space, then $\Jac_{\fr{m}}$ is identically zero and $\ad_{\fr{m}}$ is a derivation trivially. For example, any  compact connected Lie group  $M=G$   with a bi-invariant metric  can be viewed as a symmetric space of the form $(G\times G)/\Delta G$. The  Cartan decomposition is given by  $\fr{g}\oplus\fr{g}=\Delta\fr{g}\oplus\fr{m}$, where  $\Delta\fr{g}:=\{(X, X)\in\fr{g}\oplus\fr{g} : X\in\fr{g}\}\cong \fr{g}$ and $\fr{m}:=\{(X, -X)\in\fr{g}\oplus\fr{g}  : X\in\fr{g}\}\cong\fr{g}$, respectively. 
  Obviously, in this case the condition $\Jac_{\fr{m}}\equiv 0$ is the Jacobi identity which leads to the well-known result that the adjoint representation $\Lambda_{\fr{g}}=\ad_{\fr{g}}$ is a derivation of $\fr{g}$. In the following section we  examine   the condition  $\Jac_{\fr{m}}\equiv 0$  also  on non-symmetric, compact, effective and simply-connected  naturally reductive manifolds, see Corollary  \ref{nrder}.  }
\end{example}

\section{Invariant connections on naturally reductive manifolds}\label{natural} 
Next  we describe a series of new  results  related to  invariant connections (and their torsion type) on  effective naturally reductive spaces.  Note  that all these results can be  applied  on an effective, non-symmetric (compact) SII homogeneous Riemannian manifold.  

\subsection{Naturally reductive spaces}

 A   Riemannian manifold  $(M, g)$ is called  {\it naturally reductive}  if there exists a  closed subgroup $G$ of the isometry group ${\rm Iso}(M, g)$ which acts transitively on $M$ with isotropy group $K$ and  which induces a reductive decomposition $\fr{g}=\fr{k}\oplus\fr{m}$ such that   the torsion of the canonical connection $\nabla^{c}$  associated to $\fr{m}$, is a 3-form $T^{c}\in\Lambda^{3}(\fr{m})$.  This is equivalent to say that $U\equiv 0$ identically, where $U : \fr{m}\times\fr{m}\to\fr{m}$ is the bilinear map defined by (\ref{uu}). %
  Thus, an alternative way to define naturally reductive manifolds is as follows:
    \begin{definition}\label{ourdef} 
     A naturally reductive manifold is a homogeneous Riemannian manifold $(M=G/K, g)$ with a reductive decomposition $\fr{g}=\fr{k}\oplus\fr{m}$ such that canonical  homogeneous structure  $A^{c}\in\fr{m}\otimes\Lambda^{2}\fr{m}$   adapted  to $\fr{m}$ and $G$, is totally skew-symmetric, i.e. $2A^{c}(X, Y, Z)=T^{c}(X, Y, Z)$ for any $X, Y, Z\in\fr{m}$.
     \end{definition}

  A special subclass of naturally reductive manifolds $M=G/K$ consists of  the so-called     {\it normal} homogeneous Riemannian manifolds. In this case  there is   an $\Ad(G)$-invariant inner product $Q$ on $\fr{g}$ such that $Q(\fr{k}, \fr{m})=0$, i.e. $\fr{m}=\fr{k}^{\perp}$ and $Q|_{\fr{m}}=\langle \ , \ \rangle$. Thus,   a normal metric is defined by  a positive definite bilinear form $Q$. 
  Notice however that $Q$ can be more general, see \cite{K, Bes}.
     If $Q=-B$,  where $B$ denotes the   Killing form of $\fr{g}$, then the normal metric is  called the {\it Killing} (or {\it standard}) metric; this is the case if the   Lie group    $G$ is compact and  semi-simple. We mention that in this paper   whenever we refer to a naturally reductive space $(M=G/K, g, \fr{g}=\fr{k}\oplus\fr{m})$  we shall always assume that $G$ acts effectively  on $M$ and that $\fr{g}$ coincides with with the ideal $\tilde{\fr{g}}:=\fr{m}+[\fr{m}, \fr{m}]$. On the level of Lie groups this condition means that $G$ coincides with the transvection group of the associated canonical connection $\nabla^{c}$. Note that any compact normal homogeneous space satisfies this condition, see \cite{Reg1}.   


 \subsection{Properties of invariant connections in the naturally reductive setting} 
  We start with the following simple observation.
   \begin{lemma}\label{criterion}
Let $(M=G/K, g)$ be an effective compact homogeneous Riemannian manifold with reductive decomposition  $\fr{g}=\fr{k}\oplus\fr{m}$ which is not a symmetric space of Type I.   Then, there  is always an instance of  $\fr{m}$ inside $\Lambda^{2}(\fr{m})$, associated to the restriction of the Lie bracket operation  on the reductive complement $\fr{m}$. In particular,  this specific  copy gives rise to a $G$-invariant metric connection on $M$ if and only if  $g$ is naturally reductive with respect to $G$ and $\fr{m}$. 
 \end{lemma}
 \begin{proof}
 Since $M=G/K$ is not isometric to a symmetric space of Type I,  the canonical connection $\nabla^{c}$ has  non-trivial torsion $T^{c}(X, Y)=-[X, Y]_{\fr{m}}$, which   gives rise to a non-trivial $\Ad(K)$-equivariant skew-symmetric bilinear mapping $\ad_{\fr{m}} : \Lambda^{2}\fr{m}\to \fr{m}$. The second statement is apparent due to  the naturally reductive property. 
 \end{proof}
\begin{remark}
\textnormal{
If  $(M=G/K, g)$ is  a Riemannian symmetric space of Type I, then  $G$ is a compact   simple Lie group   and its Killing form $B$ gives rise to a reductive decomposition $\fr{g}=\fr{k}\oplus\fr{m}$ such that $[\fr{m}, \fr{m}]\subset\fr{k}$. Moreover,  the restriction $\langle \ , \ \rangle=-B|_{\fr{m}}$  induces a $G$-invariant metric which is naturally reductive with respect to $\fr{m}$.  However, the $K$-module $\Lambda^{2}(\fr{m})$ never contains a copy of $\fr{m}$, see also \cite{Laq2}. This is  in contrast to a Riemannian symmetric space $(M=G, g=\rho)$ of Type II endowed with a bi-invariant metric $\rho$, where one can always  find a copy of $\fr{g}$ inside $\Lambda^{2}(\fr{g})$, see also   Remark \ref{skewremark}.  Geometrically, this copy   represents the existence of 1-parameter family of canonical connections on any compact simple Lie group  $G$  (cf.  \cite{Kob2, AFH, Chrysk}). The same is true in the more general compact case (cf. \cite{Olmos}). }
\end{remark}

    \begin{lemma}\label{mainuse1}
  \textnormal{(\cite{Agr03, Chrysk})}  Let $(M=G/K, g)$ be an effective naturally reductive manifold with reductive decomposition $\fr{g}=\fr{k}\oplus\fr{m}$ such that $\fr{g}=\tilde{\fr{g}}$.  Then, \\
   (i) A $G$-invariant metric connection $\nabla$ on  $(M=G/K, g)$ has totally skew-symmetric torsion $T\in\Lambda^{3}(\fr{m})$  if and only if    $\Lambda(Z)Z=0$, for any $Z\in\fr{m}$, where $\Lambda$ is the associated Nomizu map.\\
   (ii)There is a bijective correspondence between  $\Ad(K)$-equivariant maps  $\Lambda^{\al}: \fr{m}\to\fr{so}(\fr{m})$, defined by $ \Lambda^{\al}(X)Y=\frac{1-\al}{2}[X, Y]_{\fr{m}}=(1-\al)\Lambda^{g}(X)Y$ for any $X, Y\in\fr{m}$,  and   $G$-invariant metric connections $\nabla^{\al}$ with totally skew-symmetric torsion $T^{\al}\in\Lambda^{3}(\fr{m})$, such that $T^{\al}=\al \cdot T^{c}$  for some parameter  $\al$, where $T^{c}$ is the torsion of the canonical connection $\nabla^{c}$ associated to $\fr{m}$ and $\Lambda^{g} : \fr{m}\to\fr{so}(\fr{m})$ the Nomizu map associated to the Levi-Civita connection $\nabla^{g}$.

              \end{lemma}
  
 Let us finally recall the following fundamental  result  by  Olmos and Reggiani.
  \begin{theorem}\label{or}\textnormal{(\cite[Thm.~1.2]{Olmos}, \cite[Thm.~2.1]{OReg})}
    Let $(M^{n},g)$ be a simply-connected and irreducible Riemannian manifold that is not isometric to a sphere, nor to a projective space, nor to a compact simple Lie group with a bi-invariant metric or its symmetric dual.  Then the canonical connection is unique, i.e. $(M^{n},g)$ admits at most one naturally reductive homogeneous structure.
  \end{theorem}
  
Combining the observations  in Example \ref{deram} and the results of Lemma \ref{mainuse1} and Theorem \ref{or},   in the compact and simply-connected case we obtain the following conclusion about derivations on $\fr{m}$:
\begin{corol}\label{nrder}
Let $(M=G/K, g)$ be an effective, compact and simply-connected  naturally reductive manifold, irreducible as Riemannian manifold, endowed with a reductive decomposition  $\fr{g}=\fr{k}\oplus\fr{m}$ such that $\fr{g}=\tilde{\fr{g}}$. Assume that $M=G/K$ is not isometric to a symmetric space of Type I, neither to a sphere or to a real projective space.  Then, the bilinear mapping $\ad_{\fr{m}}:=[ \ , \ ]_{\fr{m}}$ gives rise to $\Ad(K)$-equivariant derivation on $\fr{m}$, if and only if $M$ is isometric to a compact simple Lie group $G$ endowed with a bi-invariant metric.
\end{corol}
 \begin{proof}
A special version  of Corollary \ref{nrder} has been proved in \cite[Lem.~4.5]{Chrysk}. Here we improve this result. We  omit some  details and only present  the main idea. Assume that $\ad_{\fr{m}}\in{\rm Der}_{K}(\fr{m})$, i.e. $\Jac_{\fr{m}}(X, Y, Z)=0$ for any $X, Y, Z\in\fr{m}$ (see Example \ref{deram}). Then, for the family $\nabla^{\al}$ of Lemma \ref{mainuse1}, a small computation shows that $\nabla^{\al}T^{\al}=0$ for any $\al$, see for example \cite{Agr03}.  On the other hand, one can easily see that   $[X, [Y, Z]_{\fr{m}}]_{\fr{k}}+[Y, [Z, X]_{\fr{m}}]_{\fr{k}}+[Z, [X, Y]_{\fr{m}}]]_{\fr{k}}=0$, for any    $X, Y, Z\in\fr{m}$.   Combining this  identity with   $\Jac_{\fr{m}}(X, Y, Z)=0$, a long computation certifies that   $\nabla^{\al}R^{\al}=0$ for any  $\al$, as well.   But then,  $\nabla^{\al}$ is a 1-parameter family of canonical connections on $M=G/K$ (in the sense of the Ambrose-Singer Theorem) and Theorem \ref{or} yields the result. 
\end{proof}

\noindent {\bf Notation:} Let $(M=G/K, g)$ be  an effective compact naturally reductive Riemannian manifold with   a reductive decomposition $\fr{g}=\fr{k}\oplus\fr{m}=\tilde{\fr{g}}$.    If $\chi : K\to\Aut(\fr{m})$ is of real type,  we shall denote by ${\bold s}$ and ${\bold a}$  the multiplicity of $\fr{m}$ inside $\Sym^{2}(\fr{m})$ and $\Lambda^{2}(\fr{m})$, respectively (or twice the multiplicity of $\fr{m}$ inside  $\Sym^{2}(\fr{m})$ and $\Lambda^{2}(\fr{m})$, respectively,  if $\chi : K\to\Aut(\fr{m})$ is of complex type).
 We also set
    \begin{eqnarray*}
  \cal{N}={\bold s}+{\bold a}&:=&\dim_{\bb{R}}{\cal A}ff_{G}(F(G/K))=\dim_{\bb{R}}{\rm Hom}_{K}(\fr{m}\otimes\fr{m}, \fr{m})\\
  \cal{N}_{\rm mtr}&:=&\dim_{\bb{R}}{\cal {M}}_{G}(\SO(G/K))=\dim_{\bb{R}}\Hom_{K}(\fr{m}, \Lambda^{2}(\fr{m}))\leq \cal{N}.
\end{eqnarray*}
Since $K$ is compact, and we treat finite dimensional $K$-representations, we conclude  that
\begin{lemma}\label{ourr}  
The dimensions of modules $ {\rm Hom}_{K}(\La^{2}\fr{m}, \fr{m})$ and $ {\rm Hom}_{K}(\fr{m}, \La^{2}\fr{m})$ coincide,
\[
\dim_{\bb{R}} {\rm Hom}_{K}(\La^{2}\fr{m}, \fr{m})=\dim_{\bb{R}} {\rm Hom}_{K}(\fr{m}, \La^{2}\fr{m}),
\]
or in other words ${\bold a}=\cal{N}_{\rm mtr}$.
\end{lemma}
\begin{remark}\label{revised}
\textnormal{
Note that there exists compact Lie groups, e.g. $G=\U_{n}$, admitting  skew-symmetric $\Ad(G)$-equivariant maps $\Lambda^{2}(\fr{g})\to\fr{g}$ which do not induce bi-invariant  metric  connections with respect to the bi-invariant inner product $\langle X, Y\rangle=-\tr XY$ (see also the proof of Theorem \ref{mtr1}).   In fact,   below will show that Lemma \ref{ourr} implies  that   \cite[Lem.~3.1]{AFH} or \cite[Corol.~2.3, Thm.~2.9]{Chrysk} are in general false.  In  particular, the corresponding  statements  hold only for  compact simple Lie groups, but fail  for  general compact Lie groups.}
\end{remark}

Next, our  aim is to clarify Remark  \ref{revised}. 
  For simplicity, given an $\Ad(K)$-equivariant bilinear mapping $\mu : \fr{m}\times\fr{m}\to\fr{m}$ associated to a $G$-invariant connection $\nabla$ on  $(M=G/K, g)$ we shall use the   same notation for the corresponding    $K$-intertwining map $\mu : \fr{m}\otimes\fr{m}\to\fr{m}$ (and we shall identify them) and denote by $\hat{\mu}$ the contraction of  $\mu$ with the  $\Ad(K)$-invariant inner product $\langle \ , \ \rangle : \fr{m}\times\fr{m}\to \bb{R}$ associated to   $g$, i.e. $\hat{\mu}(X, Y, Z):=\langle \mu(X, Y), Z\rangle$, for any $X, Y, Z\in\fr{m}$. Notice that $\hat{\mu}$ is an $\Ad(K)$-invariant  tensor on $\fr{m}$.
Initially, it is useful to examine  invariant connections related to some  $\mu\in  {\rm Hom}_{K}(\Sym^{2}\fr{m}, \fr{m})$. Then,  the induced tensor $\hat{\mu}=\hat{\mu}^{\bold s}$ is such that $\hat{\mu}(X, Y, Z)=\hat{\mu}(Y, X, Z)$ for any $ X, Y, Z\in\fr{m}\cong T_{o}G/K$ and the corresponding Nomizu map $\Lambda:=\Lambda^{\bold s}: \fr{m}\to\Sym^{2}(\fr{m})$ is also symmetric in the sense that $\Lambda(X)Y=\Lambda(Y)X$ (since $\mu(X, Y)=\mu(Y, X)$ for any $X, Y\in\fr{m}$).  Next  we prove that when $0\neq \mu\in  {\rm Hom}_{K}(\Sym^{2}\fr{m}, \fr{m})$  is non-trivial, then the induced connection cannot preserve  the naturally reductive metric $\langle \ , \ \rangle$. We start with the non-symmetric case.
 \begin{lemma}\label{nicea}
Let $(M=G/K, g)$ be a  connected,  compact, non-symmetric, naturally reductive space of a compact Lie group $G$ modulo a compact subgroup $K$. Assume that the transitive  $G$-action is effective and let us denote by  $\fr{g}=\fr{k}\oplus\fr{m}=\tilde{\fr{g}}$ the associated reductive decomposition. If $\nabla$ is an invariant metric connection induced by some $\mu \in {\rm Hom}_{K}(\Sym^{2}\fr{m}, \fr{m})$, then $\mu=0$ and $\nabla$ coincides with the canonical connection $\nabla^{c}$ associated to $\fr{m}$. 
 \end{lemma}
 \begin{proof}
 Assume that such an invariant connection $\nabla=\nabla^{\bold s}$ exists, i.e. 
\[
\langle \Lambda^{\bold s}(X)Y, Z\rangle+\langle \Lambda^{\bold s}(X)Z, Y\rangle=0,
\]
which is equivalent to say that $\hat{\mu}\in \fr{m}\otimes\Lambda^{2}(\fr{m})$, i.e. $\hat{\mu}(X, Y, Z)+\hat{\mu}(X, Z, Y)=0$, for any $X, Y, Z\in\fr{m}$. Then, since $\Lambda^{\bold s}(X)Y=\Lambda^{\bold s}(Y)X$,  its torsion coincides with the torsion of the canonical connection, $T^{\bold s}(X, Y)=-[X, Y]_{\fr{m}}=T^{c}(X, Y)$. Since metric connections are determined by their torsion tensor, we conclude that  $\nabla^{s}=\nabla^{c}$ and $\mu=0$. 
\end{proof}
 In fact,   the non-existence of  an invariant metric connection $\nabla^{\bold s}$ corresponding to a non-trivial element  $0\neq \mu \in {\rm Hom}_{K}(\Sym^{2}\fr{m}, \fr{m})$ can be proved also as follows. By the condition $T^{\bold s}(X, Y)=-[X, Y]_{\fr{m}}$ and since   $\langle \ , \ \rangle$ is naturally reductive with respect to $G$, one concludes that  $\nabla^{\bold s}$ is an invariant connection with skew-torsion, which according to Lemma \ref{mainuse1} is equivalent to say  that $\Lambda^{\bold s}(X)X=0$ for any $X\in\fr{m}$. But then, it is also $\La^{\bold s}(X+Y)(X+Y)=0$ for any $X, Y\in\fr{m}$, i.e.    $\Lambda^{s}(X)Y=-\Lambda^{s}(Y)X$, which gives rise to a contradiction (since  $\mu\neq 0$).   Let us now explain also the compact symmetric case.
\begin{lemma}\label{symmera}
Given a connected Riemannian symmetric space $(M=G/K, g)$ of Type I   (resp. $(M=(G\times G)/\Delta(G), \rho)$ of Type II, for some compact, connected, simple Lie group $G$ with a  bi-invariant metric $\rho$), then  the unique $G$-invariant (resp. bi-invariant) metric connection which is induced by a symmetric $\Ad(K)$-equivariant mapping on $\fr{m}$ (resp, symmetric $\Ad(G)$-equivariant mapping on $\fr{g}$), is the  torsion-free   Levi-Civita connection.
\end{lemma}
\begin{proof}
  Consider first a symmetric space  $(M=G/K, g)$ of Type I, endowed with  a $G$-invariant affine connection $\nabla^{\mu}$  associated to an element  $\mu\in{\rm Hom}_{K}(\Sym^{2}\fr{m}, \fr{m})$.  Then,  $T^{\mu}(X, Y)=0$ for any $X, Y\in\fr{m}$, i.e. $\hat{\mu}\in\Sym^{2}\fr{m}\otimes\fr{m}$.  Hence, assuming in addition that $\nabla$ is metric with respect to $g$, the fundamental theorem in Riemannian geometry implies   the identification of $\nabla$ with the  unique torsion-free metric connection on $(M=G/K, g)$, i.e. the Levi-Civita connection, or  the canonical connection associated to $\fr{m}$. In particular, $\mu=0$ is trivial. The same conclusions, related this time to  bi-invariant metric connections corresponding to maps $\mu\in{\rm Hom}_{G}(\Sym^{2}\fr{g}, \fr{g})$,  hold for a compact, connected, simple Lie group $G\cong (G\times G)/\Delta G$, endowed with a bi-invariant metric $\rho$.
\end{proof}
 \begin{remark}\textnormal{Laquer   proved in \cite{Laq2}  the existence of  (irreducible) compact symmetric spaces $(M=G/K, g)$ which admit   invariant affine connections induced by non-trivial  elements  $\mu\in {\rm Hom}_{K}(\Sym^{2}\fr{m}, \fr{m})$. And indeed,   by \cite{AFH}  we know that these  $G$-invariant connections  are not metric with respect to $g=-B|_{\fr{m}}$, as it should be according to Lemma \ref{symmera}. The same is true for compact simple Lie groups, such  as $\SU_{n}$, see \cite{Laq1, AFH}.  }
 \end{remark}
 
  Let us now consider invariant connections whose torsion is a 3-form.   We show that   on an effective, non-symmetric, compact,   naturally reductive space   $(M=G/K, g)$ the $G$-invariant metric connections whose torsion is a 3-form  necessarily correspond to instances  of the  trivial representation  inside the space $\Lambda^{3}\fr{m}$, and conversely. In particular, the torsion form is a $G$-invariant 3-form.
 Let us denote by $\cal{\ell}$  the multiplicity of the (real) trivial representation inside  $\Lambda^{3}\fr{m}$.  
   \begin{lemma} \label{dimskew}
   Let  $(M=G/K, g)$  be a naturally reductive manifold as in Lemma \ref{nicea}. Then,  the dimension of the  affine space of $G$-invariant metric connections on $M$  which have the same geodesics with the Levi-Civita connection $\nabla^{g}$, i.e. $\Lambda(X)X=0$, or equivalent whose torsion form $T$ is a non-trivial  $G$-invariant 3-form, is equal to $\ell$. In particular,  
 \[
 1\leq \cal{\ell}\leq {\cal N}_{\rm mtr}={\bold a} \leq \cal{N}.
 \]
   \end{lemma}
 \begin{proof}
First notice that   $1\leq \cal{\ell}\leq {\cal{N}}_{\rm mtr}$. This follows   since the induced $\Ad(K)$-invariant inner product   $\langle \ , \ \rangle$ on $\fr{m}$ satisfies the  naturally reductive property and hence the torsion of the canonical connection $T^{c}(X, Y, Z)=-\langle [X, Y]_{\fr{m}}, Z\rangle\neq 0$ is a non-trivial $G$-invariant  3-form. Then, according to Lemma \ref{mainuse1}, the family $\nabla^{\al}=\nabla^{c}+\La^{\al}$ induces a 1-parameter family of metric connections with skew-torsion $T^{\al}:=\al T^{c}\neq 0$.  Now,  any instance  of the  trivial representation inside $\Lambda^{3}(\fr{m})$ induces a $G$-invariant (global) 3-form on $M=G/K$, say $0\neq T\in\Lambda^{3}(\fr{m})^{K}$.  If $\ell\geq 2$, then we can also assume that $T\neq T^{\al}$. But then,  one can define a 1-parameter family of metric connections  with skew-torsion, say $2sT$, given by $\nabla^{s}=\nabla^{g}+sT$.  Obviously, this family is $G$-invariant and preserves the metric.  On the other hand, if $M=G/K$ admits  a  $G$-invariant metric connection $\nabla$ with skew-torsion $T$ such that $T\neq T^{\al}$, then $T$ must be a global $G$-invariant 3-form and hence it corresponds to a new copy of the trivial representation inside $\Lambda^{3}\fr{m}$.
 \end{proof}

For a complete description of all $G$-invariant metric connections on $(M=G/K, g)$, one has to encode  the ``defect''  
\[
\epsilon:={\cal N}_{\rm mtr}-\ell\geq 0.
\]
 For this,    it is useful to consider the  tensor product 
\[
\otimes^{3}\fr{m}=\fr{m}\otimes\fr{m}\otimes\fr{m}\cong (\Lambda^{2}\fr{m}\oplus\Sym^{2}\fr{m})\otimes\fr{m}\cong \big( \Lambda^{2}\fr{m}\otimes\fr{m} \big) \oplus \big( \Sym^{2}\fr{m}\otimes\fr{m}\big)
\]
 and its decomposition in terms of Young diagrams:
 \begin{eqnarray*}
 \otimes^{3}\fr{m}&=&  {\small\begin{Young}
 \cr
\cr
\cr
\end{Young} }\quad\quad  \ \   \oplus \quad\quad  2 \ \begin{Young} 
&    \cr
\cr
 \end{Young} \quad \oplus\quad  \ \ \begin{Young}
& &  \cr
 \end{Young} \\
 &=&\Sym^{3}\fr{m} \ \ \oplus\quad  \  \ \   \cal{L}(\fr{m})\quad\quad  \oplus \quad \ \ \ \Lambda^{3}(\fr{m}),
\end{eqnarray*}
where $\cal{L}(\fr{m}):=\ker(P_{\bold {sym}})\cap \ker(P_{\bold {skew}})$ is the section of the the kernels of the equivariant projections
\[P_{\bold {sym}} :  \otimes^{3}\fr{m}\to \Sym^{3}(\fr{m}), \quad P_{\bold {skew}} : \otimes^{3}\fr{m}\to\Lambda^{3}(\fr{m}).\]
 Notice that the kernel of the natural maps $\Sym^{2}\fr{m}\otimes\fr{m}\to\Sym^{3}\fr{m}$ and $\Lambda^{2}(\fr{m})\otimes\fr{m}\to\Lambda^{3}(\fr{m})$ are isomorphic irreducible $\Gl(\fr{m})$-modules of real dimension $n(n-1)(n+1)/3$, where $n:=\dim_{\bb{R}}\fr{m}=\dim_{\bb{R}}M$ (for an example see \cite[p.~246]{Simon}). Moreover, there is an  equivariant  isomorphism 
\[
  \cal{L}(\fr{m})\cong \oplus^{2}\ker(\fr{m}\otimes\Lambda^{2}(\fr{m})\to\Lambda^{3}(\fr{m})).
\]
The intersection of $\cal{L}(\fr{m})$ with the $K$-module $\fr{m}\otimes \La^2\fr{m}$ consists of metric connections  and is isomorphic to the so-called  (2,1)-plethysm of the $K$-representation $\fr{m}$:
\[
P_{\fr{m}}(2, 1):=\cal{L}(\fr{m})\cap \big( \fr{m}\otimes \La^2\fr{m}\big).
\]


\begin{theorem}\label{plethysm}
Let $(M^{n}=G/K, g=\langle \ , \ \rangle, \fr{g}=\fr{k}\oplus\fr{m})$ as in Lemma \ref{nicea}.  The existence of the trivial representation inside the $n(n-1)(n+1)/3$-dimensional (2, 1)-plethysm $P_{\fr{m}}(2, 1)$ of $\fr{m}$, gives rise to a $G$-invariant metric connection $\nabla=\nabla^{\mu}$ on $M=G/K$ corresponding to a $K$-intertwining bilinear mapping $\mu : \fr{m}\otimes\fr{m}\to\fr{m}$ which has both non-trivial symmetric and skew-symmetric part, i.e. $\mu= \mu^{\rm skew}+\mu^{\rm sym}$, with   $0\neq \mu^{\rm skew}\in{\rm Hom}_{K}(\Lambda^{2}\fr{m}, \fr{m})$ and $0\neq\mu^{\rm sym}\in {\rm Hom}_{K}(\Sym^{2}\fr{m}, \fr{m})$, respectively.   In particular, the  torsion of $\nabla^{\mu}$ is not totally skew-symmetric and  the defect $\epsilon:={\cal N}_{\rm mtr}-\ell\geq 0$ coincides with  the multiplicity of the trivial representation inside $P_{\fr{m}}(2, 1)$.  
\end{theorem}
\begin{proof}  
 The trivial representation inside $P_{\fr{m}}(2, 1)$ induces an $\Ad(K)$-equivariant  3-tensor $\hat{\mu}$ on $\fr{m}$ which is   skew-symmetric with respect the last two indices, i.e. $\hat{\mu}\in\fr{m}\otimes\Lambda^{2}(\fr{m})$.  Since the $K$-module $\fr{m}\otimes\La^{2}\fr{m}$ corresponds to the set of  $\fr{so}(\fr{m})$-valued Nomizu maps on $M=G/K$ with respect to $\langle \ , \ \rangle$,   the induced invariant connection $\nabla=\nabla^{\mu}$ is necessarily metric.  In order to prove that its torsion is not a 3-form  we rely on  the  definition of $P_{\fr{m}}(2, 1)$ and the orthogonal decomposition 
  \begin{equation}\label{mainuse3}
  \otimes^{3}\fr{m}=\Sym^{3}\fr{m} \oplus \cal{L}(\fr{m})\oplus\Lambda^{3}(\fr{m}).
  \end{equation}
  Indeed, since the 3-tensor  $\hat{\mu}(X, Y, Z)=\langle \mu(X, Y), Z\rangle$ is induced by the trivial representation inside $P_{\fr{m}}(2, 1):=\cal{L}(\fr{m})\cap \big( \fr{m}\otimes \La^2\fr{m}\big)$,  the direct sum decomposition  (\ref{mainuse3}) together with  Lemma \ref{dimskew},   shows  that  the torsion $T^{\mu}$ of $\nabla^{\mu}$ cannot be totally skew-symmetric.  We use  now (\ref{hom1}) and  write  $\mu=\mu^{\rm skew}+\mu^{\rm sym}$ for the corresponding $K$-intertwining bilinear mapping  $\mu \in {\rm Hom}_{K}(\fr{m}\otimes\fr{m}, \fr{m})$. Since $T^{\mu}$ is not a 3-form,    $\mu^{\rm sym}$ cannot be trivial, $\mu^{\rm sym}\neq 0$. Indeed, if $\mu^{\rm sym}=0$, then    $\mu=\mu^{\rm skew}$ and hence $\mu(X, X)=0$ for any $X\in\fr{m}$. But then,   using  Lemma \ref{mainuse1}, (i) we get a  contradiction.  Assume now that $\mu$ is given by a (non-trivial) symmetric $K$-intertwining bilinear mapping, i.e $\mu^{\rm skew}=0$ and $\mu=\mu^{\rm sym}$ where $0\neq \mu^{\rm sym} : \Sym^{2}\fr{m}\to\fr{m}$.  
  Then,  according to  Lemma \ref{nicea} our connection $\nabla^{\mu}$ cannot be metric with respect to $\langle \ , \ \rangle$, which contradicts to $\hat{\mu}\in\fr{m}\otimes\Lambda^{2}(\fr{m})$.  This shows that $\mu^{\rm skew}\neq 0$, as well. 
Now, the identification of the defect  $\epsilon:={\cal N}_{\rm mtr}-\ell\geq 0$ with the multiplicity of the trivial representation in $P_{\fr{m}}(2, 1)$ is a direct consequence of (\ref{mainuse3}) and Lemmas \ref{ourr}, \ref{dimskew}. 
\end{proof}

 We mention that one cannot drop the naturally reductive assumption in Theorem \ref{plethysm}, due to the fact that the proof relies   on Lemmas \ref{nicea} and \ref{dimskew}.

 \begin{remark}\label{skewremark} 
 \textnormal{On a compact simple Lie group $G$, bi-invariant connections which are compatible with the Killing form are induced by the copy of $\fr{g}$ inside $\Lambda^{2}(\fr{g})$. Indeed, recall that
\[
\fr{so}(\fr{g})\cong \Lambda^{2}(\fr{g})=\fr{g}\oplus\fr{g}^{\perp}, \quad \fr{g}^{\perp}:=\ker \delta_{\fr{g}},
\]
where $\delta_{\fr{g}} : \Lambda^{2}(\fr{g})\to\fr{g}$ is given by $\delta_{\fr{g}}(X\wedge Y):=[X, Y]$.  Since $\delta_{\fr{g}}$ is surjective, $\fr{g}$  always lies  inside $\Lambda^{2}(\fr{g})$.
 However, the module $P_{\fr{g}}(2, 1):=\cal{L}(\fr{g})\cap \big( \fr{g}\otimes \La^2\fr{g}\big)$, where  $\cal{L}(\fr{g})$ is similarly defined by
$\cal{L}(\fr{g}):=\oplus^{2}\ker\big(\fr{g}\otimes\Lambda^{2}(\fr{g})\to\Lambda^{3}(\fr{g})\big)$, never contains the trivial summand. In contrast,  as we noticed in Remark \ref{revised}, for a compact Lie group the    case can be  different. Let us focus for example   on $G=\U_{n}$ $(n\geq 3)$.}
\end{remark}

\subsection{Bi-invariant metric connections on the compact  Lie group $\U_{n}$}   According to Laquer \cite{Laq1}, for $n\geq 3$ the space of bi-invariant affine connections on $\U_{n}$  is 6-dimensional. In particular, the following $\Ad(G)$-equivariant  bilinear  mappings form a basis of ${\rm Hom}_{\fr{g}}(\fr{g}\otimes\fr{g}, \fr{g})$ for    $\fr{g}=\fr{u}(n)$:
\begin{equation}
\left.
\begin{tabular}{lll}
$\mu_{1}(X, Y)=[X, Y],$ &  $\mu_{2}(X, Y)=i(XY+YX),$ & $\mu_{3}(X, Y)=i\tr(X)Y$ \\
$\mu_{4}(X, Y)=i\tr{(Y)}X,$ & $\mu_{5}(X, Y)=i\tr{(XY)}\Id,$ & $\mu_{6}(X, Y)=i\tr{(X)}\tr{(Y)}\Id$
\end{tabular}\right\},\label{laqun}
\end{equation}
where $XY$ denotes multiplication of matrices and $\Id$ is the identity matrix. 
We also consider the linear combinations  
\begin{eqnarray*}
\nu(X, Y)&:=&\mu_{3}(X, Y)-\mu_{4}(X, Y)=i(\tr(X)Y-\tr{(Y)}X)\in {\rm Hom}_{G}(\La^{2}\fr{g}, \fr{g}),\\ 
\vartheta(X, Y)&:=&\mu_{3}(X, Y)+\mu_{4}(X, Y)=i(\tr(X)Y+\tr{(Y)}X)\in {\rm Hom}_{G}(\Sym^{2}\fr{g}, \fr{g}).
\end{eqnarray*}

\begin{theorem}\label{mtr1}
(1) The connection induced by the $\Ad(\fr{u}(n))$-equivariant bilinear mapping $\mu=\mu_{4}-\mu_{5}$,  i.e. $\mu(X, Y):=i(\tr{(Y)}X-\tr{(XY)}\Id)$ for any $X, Y\in\fr{u}(n)$, is a bi-invariant metric connection on $\U_{n}$ $(n\geq 3)$ with respect to the bi-invariant metric induced by $\langle  \ , \ \rangle=-\tr(XY)$. The symmetric and skew-symmetric part of $\mu=\mu^{\rm skew}+\mu^{\rm sym}$ are given by
\[
\mu^{\rm skew}(X, Y)=-(1/2)\nu(X, Y), \quad\text{and}\quad \mu^{\rm sym}(X, Y)=(1/2)\vartheta(X, Y)+i\langle X, Y\rangle\Id,
\]
respectively, and its torsion has the form $T^{\mu}(X, Y)=-\nu(X, Y)+T^{c}(X, Y)$. In particular, the induced tensor $T^{\mu}(X, Y, Z)=\langle T^{\mu}(X, Y), Z\rangle$ is not totally skew-symmetric. \\ 
(2) Consequently,  for $n\geq 3$ the Lie group $\U_{n}$  carries a 2-dimensional space of bi-invariant metric connections, i.e.  ${\cal{N}}_{\rm mtr}=\epsilon+\ell=2$.
\end{theorem}
\begin{proof}
(1) The module $\cal{L}(\fr{g})$  associated to the adjoint representation of $\fr{g}=\fr{u}(n)=\bb{R}\oplus\fr{su}(n)$ contains the trivial representation twice. The one copy  corresponds to the  invariant   3-tensor $\hat{\nu}(X, Y, Z)=\langle \nu(X, Y), Z\rangle$ which is skew-symmetric only with respect to the first two indices, i.e. $\hat{\nu}\in \cal{L}(\fr{g})\cap (\La^{2}\fr{g}\otimes\fr{g})$ and thus $\nu=\mu_{3}-\mu_{4}$ fails to induce a bi-invariant  connection on $\U_{n}$, preserving $\langle \ , \ \rangle$.  The second copy corresponds to the invariant 3-tensor $\hat{\mu}(X, Y, Z)=\langle \mu(X, Y), Z\rangle$, where $\mu : \fr{g}\otimes\fr{g}\to\fr{g}$ is given by $\mu=\mu_{4}-\mu_{5}$.  We will show that $\hat{\mu}$ is indeed   inside the (2, 1) plethysm   $P_{\fr{g}}(2, 1)=\cal{L}(\fr{g})\cap (\fr{g}\otimes\La^{2}\fr{g})$, i.e. $\epsilon=1$, and hence the associated connection $\nabla^{\mu}$ gives rise to 1-dimensional family of bi-invariant metric connections on $\U_{n}$. For simplicity, we set    $\cal{O}(X, Y, Z):=\langle \mu(X, Y), Z\rangle+\langle Y, \mu(X, Z)\rangle$, for any $X, Y, Z\in\fr{u}(n)$. Then we get that
\begin{eqnarray*}
\cal{O}(X, Y, Z)&=&\langle i(\tr{(Y)}X-\tr{(XY)}\Id), Z\rangle+\langle Y, i(\tr{(Z)}X-\tr{(XZ)}\Id)\rangle\\
&=&i\big(\tr(Y)\langle X, Z\rangle-\tr(XY)\langle \Id, Z\rangle+\tr(Z)\langle Y, X\rangle-\tr(XZ)\langle Y, \Id\rangle\big)\\
&=&i\big(-\tr(Y)\tr(XZ)+\tr(XY)\tr(Z)-\tr(Z)\tr(XY)+\tr(XZ)\tr(Y)\big)=0, 
\end{eqnarray*}
for any $X, Y, Z\in\fr{u}(n)$ and this proves our assertion.  Now, according to Theorem \ref{plethysm}, $\mu$ has both non-trivial  symmetric and skew-symmetric part, namely   $\mu^{\rm sym}(X, Y)=\frac{1}{2}[\mu(X, Y)+\mu(Y, X)]$ and  $\mu^{\rm skew}(X, Y)=\frac{1}{2}[\mu(X, Y)-\mu(Y, X)]$, respectively, and a small computation  completes the proof.\\
(2)  For the second statement, the mapping $a\mu_{1}(X, Y)$ $(a\in\bb{R})$ induces a 1-parameter family of bi-invariant metric connections on $\U_{n}$ with skew-torsion and this is the unique family of bi-invariant metric connections with skew-torsion (since  the multiplicity of the trivial represenatation inside $\La^{3}\fr{g}$ is one, i.e. $\ell=1$, see also the remark below). Then, according to Theorem \ref{plethysm} it must be $\cal{N}_{\rm mtr}=\epsilon+\ell=1+1=2$, which also fits with the conclusion that $\mu$ is a new bi-invariant metric connection on $\U_{n}$, and finally also with  Lemma \ref{ourr}.  This induces a 1-parameter family of  bi-invariant metric connections  $\nabla^{b}$ $(b\in\bb{R})$, corresponding to the bilinear mapping  $\mu^{b}(X, Y):=b [i(\tr{(Y)}X-\tr{(XY)}\Id)]=b\mu(X, Y)$,  with  $X, Y\in\fr{u}(n)$. 
The torsion $T^{b}\in \La^{2}\fr{u}(n)\otimes\fr{u}(n)$ is not totally skew-symmetric. Indeed, the torsion of the mapping $\mu=\mu_{4}-\mu_{5}$ is given by $T^{\mu}(X, Y)=-\nu(X, Y)+T^{c}(X, Y)$. It  is not totally skew-symmetric since for example  $\mu(X, X)\neq 0$ and $\langle \ , \ \rangle$ is a naturally reductive metric. Similarly for $\mu^{b}$. This finishes the proof.
 \end{proof}
\begin{remark}\label{3forms}\textnormal{For  a verification of the fact  $\ell=1$ for $\U_{n}$, one can use the \text{LiE} program (and stability arguments), or even apply the following. First,  for dimensional reasons notice that
\[
\Lambda^{3}(\fr{g})=\Lambda^{3}(\fr{u}(n))=\La^{3}(\bb{R}\oplus\fr{su}(n))=\La^{3}(\fr{su}(n))\oplus \big(\bb{R}\otimes \La^{2}\fr{su}(n)).\quad (\ast)
\]
Using (\ref{l2}) we also see that   $\bb{R}\otimes \La^{2}\fr{su}(n)$ doesn't  contain the trivial representation. For the decomposition of $\La^{3}(\fr{su}(n))$, recall  first that  any compact  simple lie group $\hat{G}$ admits a non-trivial global $\hat{G}$-invariant 3-from, the so-called Cartan 3-form  $\omega_{\hat{\fr{g}}}(X, Y, Z)=B([X, Y], Z)$, where  $B$ denotes the Killing form on the Lie algebra $\hat{\fr{g}}$. On the other hand, the   $\Ad(\hat{G})$-equivariant differential $d_{\hat{\fr{g}}} : \La^{k}(\hat{\fr{g}})\to\La^{k+1}(\hat{\fr{g}})$ on $\hat{\fr{g}}$ is defined by $d_{\hat{\fr{g}}}(\psi\wedge\varphi)=d_{\hat{\fr{g}}}(\psi)\wedge\varphi+(-1)^{\rm deg\psi}\psi\wedge d_{\hat{\fr{g}}}(\varphi)$ with $d_{\hat{\fr{g}}}(\varphi)=\sum_{i}(Z_{i}\lrcorner \omega_{\hat{\fr{g}}})\wedge (Z_{i}\lrcorner\varphi)$ for some $(-B)$-orthonormal basis $\{Z_{i}\}$ of $\hat{\fr{g}}$.  In these terms, in \cite{Le} it was shown that the splitting $ \La^{3}(\hat{\fr{g}})=\Span_{\bb{R}}\{\omega_{\hat{\fr{g}}}\} \oplus \delta_{\hat{\fr{g}}}(\La^{4}(\fr{g}))\oplus d_{\hat{\fr{g}}}(\hat{\fr{g}}^{\perp})$ defines  an equivariant  orthogonal decomposition of  $\La^{3}(\hat{\fr{g}})$, where $\delta_{\hat{\fr{g}}}$ is the adjoint operator of $d_{\hat{\fr{g}}}$ with respect to $-B$ (see also Remark \ref{skewremark}). From this decomposition, one deduces that $\ell=1$ for any compact simple Lie group $\hat{G}$,  and since $\fr{u}(n)=\bb{R}\oplus\hat{\fr{g}}$ with $\hat{\fr{g}}=\fr{su}(n)$,   by $(\ast)$ we conclude the same for $\U_{n}$.}
\end{remark}
  We  finally observe that $\mu:=\mu_{4}-\mu_{5}$ does not  induces a derivation on $\fr{m}$ (apply for example Proposition \ref{derin} or Theorem \ref{deriv2}). In particular, \cite[Thm.~2.9]{Chrysk} holds  only for $G$ compact and simple (the direct claim is   true even in the compact case, but the converse direction fails  for non-simple Lie groups, since  \cite[Lem.~3.1]{AFH}, or    \cite[Thm.~2.1]{Chrysk}, is valid only  for a compact simple Lie group). 


  \subsection{Characterization of the types of  invariant metric connections}
Given an effective naturally reductive Riemannian manifold $(M=G/K, g)$, our aim now  is to characterize   the possible invariant connections with respect to  their torsion type (for skew-torsion, see  \cite{Agr03} or  Lemma \ref{mainuse1}). 
  We remark that next  is not necessary to assume the compactness of $M=G/K$. 
  \begin{prop}\label{kriton}   
  Let $(M^{n}=G/K, g)$ be a   homogeneous Riemannian manifold which is naturally reductive with respect  to a closed subgroup $G\subseteq \Iso(M, g)$ of the isometry group and let $\fr{g}=\fr{k}\oplus\fr{m}$ be the associated reductive decomposition. Assume that the transitive $G$-action is effective, $\fr{g}=\tilde{\fr{g}}$ and denote by $\nabla\equiv \nabla^{\mu}$   a  $G$-invariant metric connection corresponding to $\mu\in{\rm Hom}_{K}(\fr{m}\otimes\fr{m}, \fr{m})$.   Set $\hat{\mu}(X, Y, Z)=\langle \mu(X, Y), Z\rangle$, $A(X, Y)=\nabla_{X}Y-\nabla^{g}_{X}Y$ and $A(X, Y, Z)=\langle A(X, Y), Z\rangle$ for any $X, Y, Z\in\fr{m}$,  where $\nabla^{g}$ is the Levi-Civita connection. Then, the following hold:\\
  (1) $\nabla$ is of vectorial type, i.e. $A\in\cal{A}_{1}$, if and only if  there is a global $G$-invariant 1-form $\varphi$ on $M$ such that
  \[
  \hat{\mu}(X, Y, Z)=\frac{1}{2}\langle [X, Y]_{\fr{m}}, Z\rangle+\langle X, Y\rangle\varphi(Z)-\langle X, Z\rangle\varphi(Y), \quad \forall \  X, Y, Z\in\fr{m}.\]
  (2) $\nabla$ is of Cartan type or traceless cyclic, i.e. $A\in\cal{A}_{2}$, if any only if the following two conditions are simultaneously satisfied:
  \[
  \begin{tabular}{ll}
  $(\al)$ & $\fr{S}_{X, Y, Z}\hat{\mu}(X, Y, Z)=\frac{3}{2}\langle [X, Y]_{\fr{m}}, Z\rangle, \quad    \forall \  X, Y, Z\in\fr{m}$, \\
$(\be)$ & $\sum_{i}\mu(Z_{i}, Z_{i})=0$,
\end{tabular}
  \]
  where $Z_{1}, \ldots, Z_{n}$ is an  arbitrary $\langle  \ , \ \rangle$-orthonormal basis of $\fr{m}$.\\
  (3) $\nabla$ is cyclic, i.e. $A\in\cal{A}_{1}\oplus\cal{A}_{2}$, if and only if $\fr{S}_{X, Y, Z}\hat{\mu}(X, Y, Z)=\frac{3}{2}\langle [X, Y]_{\fr{m}}, Z\rangle$,  $\forall \ X, Y, Z\in\fr{m}$.\\
     (4) $\nabla$ is traceless, i.e. $A\in\cal{A}_{2}\oplus\cal{A}_{3}$, if and only if $\sum_{i}\mu(Z_{i}, Z_{i})=0$.
  \end{prop}
  \begin{remark}\textnormal{Before proceed with the proof, let us first describe a useful formula. Recall that the torsion of $\nabla$ is given by $T(X, Y)=\mu(X, Y)-\mu(Y, X)-[X, Y]_{\fr{m}}$,  or in other words $T(X, Y, Z)=\hat{\mu}(X, Y, Z)-\hat{\mu}(Y, X, Z)-\langle [X, Y]_{\fr{m}}, Z\rangle$ for any $X, Y, Z\in\fr{m}$. Therefore, a short application of (\ref{relay}) gives rise to \begin{eqnarray*}
2A(X, Y, Z)&=&T(X, Y, Z)-T(Y, Z, X)+T(Z, X, Y)\\
&=& \hat{\mu}(X, Y, Z)-\hat{\mu}(Y, X, Z)-\langle [X, Y]_{\fr{m}}, Z\rangle -\hat{\mu}(Y, Z, Z)+\hat{\mu}(Z, Y, X)+\langle [Y, Z]_{\fr{m}}, X\rangle \\
&& +\hat{\mu}(Z, X, Y)-\hat{\mu}(X, Z, Y)-\langle [Z, X]_{\fr{m}}, Y\rangle\\
&=&2\hat{\mu}(X, Y, Z)-\langle [X, Y]_{\fr{m}}, Z\rangle,
\end{eqnarray*}
since $\hat{\mu}(X, Y, Z)+\hat{\mu}(X, Z, Y)=0$ for any $X, Y, Z\in\fr{m}$ and $\langle  \ , \ \rangle$ is naturally reductive. Thus 
\begin{equation}\label{relate}
A(X, Y, Z)=\hat{\mu}(X, Y, Z)-\frac{1}{2}\langle [X, Y]_{\fr{m}}, Z\rangle, \quad \forall \ X, Y, Z\in\fr{m}.
\end{equation}
}
  \end{remark}
  \begin{proof}
  (1) Assume that $M=G/K$ carries a $G$-invariant metric connection $\nabla$ whose torsion is of vectorial type. 
  Next we shall identify $\fr{m}\cong T_{o}M$ and for any $X\in\fr{m}\subset\fr{g}$  we shall write $X^{*}$ for the (Killing) vector field  on $M$ induced by $\exp(-tX)$.  Recall that  $[X^{*}, Y^{*}]_{o}=-[X, Y]^{*}_{o}=-[X, Y]_{\fr{m}}$.   Since $\nabla$ is a $G$-invariant connection, identifying $(\nabla_{X^{*}}Y^{*})_{o}=\nabla_{X}Y$,  we can write
 \begin{eqnarray*}
 \langle \nabla_{X}Y, Z\rangle&=&\langle\nabla^{c}_{X}Y, Z\rangle+\langle\mu(X, Y), Z\rangle=\langle\nabla^{c}_{X}Y, Z\rangle+\langle\Lambda^{\mu}(X)Y, Z\rangle\\
 &=&-\langle [X, Y]_{\fr{m}}, Z\rangle +\hat{\mu}(X, Y, Z), \quad (\ast)
 \end{eqnarray*}
 where $(\nabla_{X^{*}}^{c}Y^{*})_{o}=\nabla^{c}_{X}Y=-[X, Y]_{\fr{m}}=[X^{*}, Y^{*}]_{o}$ is the canonical connection with respect to $\fr{m}$  (cf. \cite{Olmos, Reg1}).   However, $\nabla$ is of vectorial type, hence there is a 1-form $\varphi$ on $M=G/K$  such that 
 \[
\fr{m}\otimes\Lambda^{2}\fr{m}\cong \cal{A}\ni A(X, Y, Z)=\langle X, Y\rangle\varphi(Z)-\langle X, Z\rangle\varphi(Y),
 \]
  for any $X, Y, Z\in\fr{m}$. Using that $\langle  \ , \ \rangle$  is naturally reductive with respect to $G$ and $\fr{m}$, we compute $(\nabla^{g}_{X^{*}}Y^{*})_{o}=\frac{1}{2}[X^{*}, Y^{*}]_{o}=-\frac{1}{2}[X, Y]_{\fr{m}}$ and
  \[
\langle \nabla_{X}Y, Z\rangle=\langle \nabla^{g}_{X}Y,  Z\rangle  +A(X, Y, Z)
=-\frac{1}{2}\langle [X, Y]_{\fr{m}}, Z\rangle+\langle X, Y\rangle\varphi(Z)-\langle X, Z\rangle\varphi(Y).
\]
 Hence,  a small combination with $(\ast)$ gives rise to 
 \[
 \hat{\mu}(X, Y, Z)=\frac{1}{2}\langle [X, Y]_{\fr{m}}, Z\rangle+\langle X, Y\rangle\varphi(Z)-\langle X, Z\rangle\varphi(Y). \quad (\ast\ast)
  \]
However, $\hat{\mu}$ is an $\Ad(K)$-invariant tensor  (or in other words, it corresponds to a $G$-invariant tensor field on $M=G/K$), and hence by $(\ast\ast)$ we conclude that $\varphi$ must be a (global) $G$-invariant 1-form on $M$. This proves the one direction. Assume now that  $(M=G/K, g)$ is endowed with a $G$-invariant tensor $\hat{\mu}\in\fr{m}\otimes\Lambda^{2}\fr{m}$ satisfying $(\ast\ast)$ for some  $G$-invariant 1-form $\varphi$ on $M$ and let us denote by $\nabla$ the associated $G$-invariant metric connection. Then,  a combination of (\ref{relate}) and $(\ast\ast)$  yields that $A\in\cal{A}_{1}$, which  completes the proof of (1).\\
(2) Assume that $M=G/K$ carries a $G$-invariant metric connection $\nabla$ which is traceless cyclic. This means that  the invariant tensor $A(X, Y, Z)$ must satisfy the conditions 
\[
\fr{S}_{X, Y, Z}A(X, Y, Z)=0 \ \ \text{and} \ \ \sum_{i}A(Z_{i}, Z_{i}, Z)=0, \quad (\dag)
\] 
where  $\{Z_{i}\}$ is an orthonormal basis of $\fr{m}$ with respect to $\langle  \ , \ \rangle$.    By (\ref{relate}) we see that 
\[
\sum_{i}A(Z_{i}, Z_{i}, Z)=0 \ \Leftrightarrow \ \sum_{i}\hat{\mu}(Z_{i}, Z_{i}, Z)=0.
\]
However, $\sum_{i}\hat{\mu}(Z_{i}, Z_{i}, Z)=\sum_{i}\langle \mu(Z_{i}, Z_{i}), Z\rangle=\langle \sum_{i}\mu(Z_{i}, Z_{i}),  Z\rangle
=\langle \sum_{i}\Lambda(Z_{i})Z_{i}, Z\rangle$, where  $\Lambda\equiv\Lambda^{\mu}: \fr{m}\to\fr{so}(\fr{m})$ is the associated  connection map. Thus,  the traceless condition in $(\dag)$ holds if and only if $\sum_{i}\mu(Z_{i}, Z_{i})=\sum_{i}\Lambda(Z_{i})Z_{i}=0$.  Now, for the cyclic condition in $(\dag)$, using (\ref{relate})
we obtain  the relation
\[
\fr{S}_{X, Y, Z}A(X, Y, Z)=\fr{S}_{X, Y, Z}\hat{\mu}(X, Y, Z)-\frac{3}{2}\langle [X, Y]_{\fr{m}}, Z\rangle
\]
and in this way we conclude the second stated relation. In fact, this   follows also by the cyclic sum $\fr{S}_{X, Y, Z}T(X, Y, Z)=0$, where $T$ is the torsion of $\nabla$.   \\
(3) Parts (3) and (4)  are immediate  due to the description given in (2) and the definition of the  classes $\cal{A}_{1}\oplus\cal{A}_{2}$, and $\cal{A}_{2}\oplus\cal{A}_{3}$. 
\end{proof}
 \begin{remark}\label{symm1}
 \textnormal{If  $(M=G/K, g)$ is an effective Riemannian symmetric space endowed with a   $G$-invariant metric connection $\nabla\equiv \nabla^{\mu}$  corresponding to some $\mu\in{\rm Hom}_{K}(\fr{m}\otimes\fr{m}, \fr{m})$, then  the conclusions in Proposition  \ref{kriton} are simplified, i.e.  for  the  tensor $A=\nabla^{\mu}-\nabla^{g}$ we deduce that 
\begin{itemize}
\item $A\in\cal{A}_{1}$, i.e. $\nabla$ is vectorial, if and only if $\exists$  a global $G$-invariant 1-form $\varphi$ on $M$ such that 
 \[
  \hat{\mu}(X, Y, Z)=\langle X, Y\rangle\varphi(Z)-\langle X, Z\rangle\varphi(Y), \quad \forall \  X, Y, Z\in\fr{m}.
\]
\item $A\in \cal{A}_{2}$, i.e. $\nabla$ is traceless cyclic,  if any only if  $\fr{S}_{X, Y, Z}\hat{\mu}(X, Y, Z)=0$ and $\sum_{i}\Lambda(Z_{i})Z_{i}=0$.
\item  $A\in\cal{A}_{1}\oplus\cal{A}_{2}$, i.e. $\nabla$ is cyclic, if and only if $ \fr{S}_{X, Y, Z}\hat{\mu}(X, Y, Z)=0$ for any $X, Y, Z\in\fr{m}$.
  \end{itemize}
Because on a compact Riemannian symmetric space $(M=G/K, g)$ of Type I, the $G$-invariant metric connections are exhausted by the  torsion-free canonical connection $\nabla^{c}=\nabla^{g}$ associated to $\fr{m}$, in the compact case the above conditions are of particular interest  for compact connected (non-simple) Lie groups endowed with a bi-invariant metric, where $A=\nabla^{\mu}-\nabla^{g}$ can be non-trivial.  For example,  below we  apply these considerations for the Lie group $\U_{n}$.  Finally notice that considering a naturally reductive space as in Proposition \ref{kriton} (or even a symmetric space as above), it is easy to certify  that any $G$-invariant metric connection    of type $\cal{A}_{3}$ is also   of type $\cal{A}_{2}\oplus\cal{A}_{3}$, any $G$-invariant metric connection of type $\cal{A}_{1}$ it is also of type $\cal{A}_{1}\oplus\cal{A}_{2}$, etc.}
   \end{remark}

  \begin{prop}\label{un}
  For $n\geq 3$, the bi-invariant metric connection $\nabla^{\mu}$ on $(\U_{n}, \langle \ , \ \rangle)$ induced by the map $\mu:=\mu_{4}-\mu_{5}$ of Theorem \ref{mtr1} has torsion of vectorial type.    \end{prop}
  \begin{proof}
The Lie group $\U_{n}$ has 1-dimensional center $Z$; hence the quotient  $(\U_{n}\times \U_{n})/\Delta\U_{n}$ is not yet effective, but the expression $(\U_{n}/Z)/(\Delta \U_{n}/\Delta Z)$ satisfies this condition. From now on we shall identify $\U_{n}\cong (\U_{n}\times \U_{n})/\Delta\U_{n}\cong (\U_{n}/Z)/(\Delta \U_{n}/\Delta Z)$ and   write $\fr{u}(n)\oplus\fr{u}(n)=\Delta\fr{u}(n)\oplus\fr{m}$ for the associated symmetric reductive decomposition, where 
\[
\Delta\fr{u}(n):= \{(X,X) \in\fr{u}(n)\oplus\fr{u}(n) : X \in\fr{u}(n)\}, \quad \fr{m} := \{(X, -X) \in\fr{u}(n)\oplus\fr{u}(n) : X \in\fr{u}(n)\}
\]
are both isomorphic to $\fr{u}(n)$ as $\U_{n}$-modules. 
 Because any compact connected Lie group $G$ endowed with a bi-invariant metric  is a compact normal homogeneous space and moreover a compact symmetric space,  the condition $\fr{g}=\tilde{\fr{g}}$ of Proposition \ref{kriton} is satisfied and we can apply the considerations of Remark \ref{symm1}.  Consider  the Lie algebra $\fr{u}(n)$  endowed with the bilinear mapping $\mu(X, Y)=i(\tr(Y)X-\tr(XY)\Id)$, given in Theorem \ref{mtr1}.  Since $\langle X, Y\rangle=-\tr(XY)$ we conclude  that
 \begin{eqnarray}
 \hat{\mu}(X, Y, Z):=\langle \mu(X, Y), Z\rangle&=&i\tr(Y)\langle X, Z\rangle -i\tr(XY)\langle \Id, Z\rangle\nonumber\\
 &=&i\tr(Y)\langle X, Z\rangle+i\langle X, Y\rangle \langle \Id, Z\rangle\nonumber\\
 &=&i\tr(Y)\langle X, Z\rangle -i\tr(Z)\langle X, Y\rangle,\label{uncase}
  \end{eqnarray}
 for any $X, Y, Z\in\fr{u}(n)$. 
Consider now  the  1-form $\varphi : \fr{u}(n)\to \bb{R}$,  $Y\mapsto \varphi(Y):=-i\tr(Y)$. It is easy to see that $\varphi$ is a $\U_{n}$-invariant 1-form with kernel $\fr{su}(n)$. But then,  based on (\ref{uncase}) we obtain that
 \[
 \hat{\mu}(X, Y, Z)=-\langle X, Z\rangle \varphi(Y)+\langle X, Y\rangle\varphi(Z),
 \]
 for any $X, Y, Z\in\fr{u}(n)$ and using Remark \ref{symm1} we conclude that $\cal{A}^{\mu}:=\nabla^{\mu}-\nabla^{g}\in \cal{A}_{1}$.
 \end{proof}
   
   \begin{remark}
   \textnormal{By Theorem \ref{mtr1}, the group  $\U_{n}$ $(n\geq 3)$  is   equipped with a two dimensional space of bi-invariant metric connections $\nabla^{f}$, given by the bilinear map $f:=a\mu_{1}+b\mu$ $(a, b\in\bb{R})$ where $\mu_{1}$ and $\mu$ are given by (\ref{laqun}) and Theorem  \ref{mtr1}, respectively.  In general, $\nabla^{f}$ is of mixed type $\cal{A}_{1}\oplus\cal{A}_{3}$,  but the conditions that the type of $\nabla^{f}$ is either purely $\cal{A}_{3}$  or purely $\cal{A}_{1}$, naturally defines the one dimensional subfamilies $a\mu_{1}$ and $b\mu$, respectively. Thus, we can express  the space of bilinear mappings inducing bi-invariant metric connections on $\U_{n}$ as a direct sum of these families.}
   \end{remark}
     
  \subsection{The curvature tensor and the Ricci tensor}  Let us now examine the curvature tensor. 
    \begin{prop}\label{curvt}
  Let $(M=G/K, g)$ be a naturally reductive  Riemannian manifold as in Proposition \ref{kriton}. 
   Then, the curvature  tensor $R^{\nabla^{\mu}}\equiv R^{\nabla}$ associated to    a $G$-invariant metric  connection $\nabla\equiv \nabla^{\mu}$  on $M=G/K$,  induced by some  $\mu\in{\rm Hom}_{K}(\fr{m}\otimes\fr{m}, \fr{m})$, satisfies the following relation
  \begin{eqnarray*}
  R^{\nabla}(X, Y)Z&=&R^{g}(X, Y)Z+A(X, \mu(Y, Z))-A(Y, \mu(X, Z))-A([X, Y]_{\fr{m}}, Z)\\
  &&+\frac{1}{2}\Big([X, A(Y, Z)]_{\fr{m}}-[Y, A(X, Z)]_{\fr{m}}\Big),
  \end{eqnarray*}
  for any $X, Y, Z\in\fr{m}$, where the tensor $A$ is defined by the difference $A=\nabla-\nabla^{g}$ and $R^{g}$ is the Riemannian curvature tensor. If  $(\fr{g}, \fr{k})$ is a symmetric pair,  then the last three terms in the previous relation are canceled. 
  \end{prop}
   \begin{proof}
  The proof relies on a straightforward computation using the formulas 
  \[
  R^{\nabla}(X, Y)Z=\mu(X, \mu(Y, Z))-\mu(Y, \mu(X, Z))-\mu([X, Y]_{\fr{m}})Z-[[X, Y]_{\fr{k}}, Z],
  \]
  and $A(X, Y)=\mu(X, Y)-\mu^{g}(X, Y)=\Lambda(X)Y-\Lambda^{g}(X)Y$ where $\mu^{g}(X, Y)=\Lambda^{g}(X)Y=\frac{1}{2}[X, Y]_{\fr{m}}$ is the bilinear map associated to the Levi-Civita connection on $M=G/K$, see also (\ref{relate}).  The last conclusion relies  on the symmetric reductive decomposition, in particular (\ref{relate}) reduces to    $A(X, Y, Z)=\hat{\mu}(X, Y, Z)$ for any $X, Y, Z\in\fr{m}$.  
  \end{proof}
 
  
Consider now a $G$-invariant metric connection $\nabla$ of vectorial type.  Let us denote by $\varphi$ the associated $\Ad(K)$-invariant 1-form on $\fr{m}$ and by  $\xi\in\fr{m}$  the dual vector with respect to $\langle  \ , \ \rangle$. 
 If    $\|\xi\|^{2}\neq0$, then   $\nabla$ is called a  {\it $G$-invariant connection of non-degenerate vectorial type}.    In this case,   by  applying  \cite[Corol.~3.1]{AK}    or by a direct calculation
based on Proposition \ref{curvt},  we get that
\begin{corol}\label{ricvec1}
Let $(M=G/K, g)$ be a   naturally reductive manifold  as in  Proposition  \ref{kriton}, endowed with  a  $G$-invariant metric connection $\nabla\equiv \nabla^{\mu}$  of  non-degenerate vectorial type.      Then\\
(1)
For   any $X, Y\in\fr{m}$, the Ricci tensor $\Ric^{\nabla}$ associated  to $\nabla$ satisfies  the relation
  \begin{equation}\label{ricvec}
  \Ric^{\nabla}(X, Y)=\Ric^{g}(X, Y)+(n-2)\langle X, \xi\rangle\langle Y, \xi\rangle+(2-n)\|\xi\|^{2}\langle X, Y\rangle+\frac{2-n}{2}\langle [X, Y]_{\fr{m}}, \xi\rangle.
  \end{equation}
  \noindent  (2) $\Ric^{\nabla}$ is symmetric if and only if $(\fr{g}, \fr{k})$ is a symmetric pair and this is equivalent to say that $\varphi$ is a closed invariant 1-form. 
    \end{corol}
    \begin{proof}
  We  prove only the second claim.  By (\ref{ricvec})  it follows that  
  \[
  \Ric^{\nabla}(X, Y)-\Ric^{\nabla}(Y, X)=(n-2)\langle [X, Y]_{\fr{m}}, \xi\rangle, \quad  \forall \ X, Y\in\fr{m}.
  \]
   Hence, $\Ric^{\nabla}$ is symmetric  if and only if $\langle [X, Y]_{\fr{m}}, \xi\rangle=0$. But since $\xi\neq 0$, this  is equivalent to say that $(\fr{g}, \fr{k})$ is a symmetric pair, i.e. $[\fr{m}, \fr{m}]\subset\fr{k}$.  By the definition of the differential of an invariant form (cf. \cite[pp.248-250]{Wolf2}), or by \cite[Prop.~3.2]{AK} we get the last correspondence.  
    \end{proof}

  Specializing to the Lie group $\U_n$ we conclude that 
  \begin{corol}\label{ricun}
  Consider the Lie group $\U_{n}$ $(n\geq 3)$ endowed with the   bi-invariant metric connection $\nabla^{\mu}$ induced by the map $\mu=\mu_{4}-\mu_{5}$, as described in Theorem \ref{mtr1}.  Then, the Ricci tensor $\Ric^{\mu}$ associated to $\nabla^{\mu}$ is given by the following symmetric invariant bilinear form on $\fr{u}(n)$:
  \[
  \Ric^{\mu}(X, Y)=\frac{1}{2}\Big\{(n-4)\tr XY+(5-2n)\tr X\tr Y\Big\}= -\frac{(n-4)}{2}\langle X, Y\rangle +\frac{(5-2n)}{2}\beta(X, Y)
  \]
for any $X, Y\in\fr{m}\cong\fr{u}(n)$, where   $\beta(X, Y):=\tr X\tr Y$.
  \end{corol}
  \begin{proof}
We use  the   notation   of  Proposition \ref{un}  and   view $\U_{n}$ as  an effective symmetric space endowed with the bi-invariant metric induced by $\langle X, Y\rangle=-\tr(XY)$.  Consider  the Nomizu map   
\[
\Lambda^{\mu}(X)Y:=i(\tr{(Y)}X-\tr{(XY)}\Id), \quad   \forall \ X, Y\in\fr{m}\cong\fr{u}(n).
\]
  By Proposition \ref{un}  we know  that the bi-invariant metric  connection $  \nabla^{\mu}_{X}Y=\nabla^{c}_{X}Y+\La^{\mu}(X)Y$   has torsion of vectorial type, associated to the  $\U_{n}$-invariant linear form  $\varphi(Z)=-i\tr(Z)=i\langle \Id, Z\rangle$. The dual vector $\xi\in\fr{m}$ is defined by $\varphi(Z)=\langle Z, \xi\rangle$ for any $Z\in\fr{m}$ and hence we conclude that $\xi=i\Id$, in particular $0\neq \langle \xi, \xi\rangle=1=\|\xi\|^{2}$. Thus, the vectorial structure is non-degenerate and   we can apply Corollary \ref{ricvec1}, i.e.
\begin{eqnarray*}
  \Ric^{\mu}(X, Y)&=&\Ric^{g}(X, Y)+(n-2)\big(\langle X, \xi\rangle\langle Y, \xi\rangle-\langle X, Y\rangle\big)\\
&=&\Ric^{g}(X, Y)+(n-2)\big(\tr(XY)-\tr X\tr Y\big).
 \end{eqnarray*}
Now,  $\langle \ , \ \rangle$ is a bi-invariant inner product and hence  $\Ric^{g}(X, Y)=-\frac{1}{4}B(X, Y)$   for any  $X, Y\in\fr{u}(n)$, where $B(X, Y)=2n\tr XY-2\tr X\tr Y$ is the Killing form of $\U_{n}$ (cf. \cite{Bes, Arvanito} where the statement is given for a  compact   semi-simple  Lie group,  but notice that  $\Ric^{g}$  satisfies the same formula  for any bi-invariant metric $g$ on a Lie group $G$). 
Thus, a small computation in combination with the formula given above yields the result.
  \end{proof}
  
  \begin{remark}
  \textnormal{For $n=3$, Corollary \ref{ricun} gives rise to the remarkable expression
  \[
  \Ric^{\mu}(X, X)=-\frac{1}{2}\big(\tr X^{2}+(\tr X)^{2}\big), \quad \forall X\in\fr{u}(3).
  \]
  Thus, in this case we conclude that  $\Ric^{\mu}(X, X)>0$ is always positive   for any non-zero left-invariant vector field $0\neq X\in\fr{u}(3)$.   Recall  that  $\Ric^{g}(X, X)\geq 0$ for any $0\neq X\in\fr{su}(3)$ with $\Ric^{g}(X, X)=0$, if and only $X\in Z(\fr{u}(3))\cong\bb{R}$.  Finally, on $\U_{4}$  the Ricci tensor   is degenerate, $\Ric^{\mu}(X, Y)=-\frac{3}{2}\beta(X, Y)$.}
  \end{remark}
  
  \section{Classification of invariant connections on  non-symmetric SII spaces}\label{II} 
  \subsection{Strongly isotropy irreducible spaces (SII)}  
 Consider a compact, connected, effective, non-symmetric SII homogeneous space $M=G/K$. Since $G$ is a compact simple Lie group  (see \cite[p.~62]{Wolf}), any such manifold is a standard homogeneous Riemannian manifold.  
 Passing to a covering $\tilde{G}$ of $G$, if $G/K$ is not simply-connected but $G$ is connected, then $\tilde{G}$ acts transitively on the universal covering of $G/K$ with connected isotropy group, say $K'$, and  it turns out that  $G/K$ is SII if and only if $\tilde{G}/K'$ is. Hence, whenever necessary we can   assume that $G/K$ is a compact, connected and  simply-connected,  effective, non-symmetric SII space, with  $G$ being   compact, connected and simple   and $K\subset G$ compact and connected. In  this setting, the strongly isotropy irreducible condition is equivalent  to an (almost effective)  irreducible action of  the Lie algebra $\fr{k}=T_{e}K$ on $\fr{m}\cong T_{o}G/K$.  
    For a list of non-symmetric SII spaces  we refer to  \cite[Tables 5, 6, p.~203]{Bes}.  We remark however  that   there are  misprints  in Table 6 of \cite{Bes}, related to SII homogeneous spaces $M=G/K$  of $G=\Sp_{n}$ (compare for example with \cite[Thm.~7.1]{Wolf}). We correct these errors  in our Table \ref{table:five}  below.

  \begin{prop}\label{right1}
Let $(M=G/K, g=-B|_{\fr{m}})$ be an effective, non-symmetric (compact)  SII homogeneous Riemannian manifold, endowed with a  $G$-invariant affine connection $\nabla^{\mu}$ compatible with the Killing metric $\langle \ , \ \rangle=-B|_{\fr{m}}$, where $\mu\in{\rm Hom}_{K}(\fr{m}\otimes\fr{m}, \fr{m})$. Then, the torsion $T^{\mu}$ of $\nabla^{\mu}$ does not carry a component of vectorial type.
\end{prop}
  \begin{proof}
 Assume that $M=G/K$ carries a $G$-invariant metric connection $\nabla$ whose torsion is of vectorial type and let $\fr{g}=\fr{k}\oplus\fr{m}$ be the reductive decomposition  with respect to  the Killing metric.  Then, by Proposition \ref{kriton}  (1),   we have that  $ \hat{\mu}(X, Y, Z)=\frac{1}{2}\langle [X, Y]_{\fr{m}}, Z\rangle+\langle X, Y\rangle\varphi(Z)-\langle X, Z\rangle\varphi(Y)$, for some   $G$-invariant 1-form $\varphi$  on $M=G/K$.  However,  $\fr{m}$ is  a self-dual and (strongly)   irreducible $K$-module over $\bb{R}$; thus global $G$-invariant 1-forms do not exist, since  dually the isotropy representation needs to preserve some vector field $\xi$ and hence a 1-dimensional subspace of $\fr{m}$, spanned by $\xi$.  
   \end{proof}
 \begin{corol}\label{traceless}
 Let $(M=G/K, g)$ be an effective, non-symmetric (compact) SII homogeneous Riemannian manifold,  endowed with a non torsion-free $G$-invariant metric connection $\nabla$.  Then,  the torsion $0\neq T$ of $\nabla$ is totally skew-symmetric, $T\in{\cal A}_{3}\cong\Lambda^{3}TM$, or traceless cyclic $T\in\cal{A}_{2}$, or of mixed type $T\in\cal{A}_{2}\oplus\cal{A}_{3}$, i.e. traceless.  
 \end{corol}

 \subsection{An application in the spin case}
   Consider  an effective,   non-symmetric (compact) SII homogeneous Riemannian manifold $(M^{n}=G/K, g)$. Assume that $M=G/K$ admits a $G$-invariant spin structure, i.e. a $G$-homogeneous $\Spin(\fr{m})$-principal  bundle $P\to M$   and a double covering morphism $\Lambda : P\to \SO(M, g)$ compatible with the principal groups' actions. Recall that an invariant spin structure corresponds to a lift  of the isotropy representation $\chi$  into    the spin group $\Spin(\fr{m})\equiv \Spin_{n}$, i.e. a homomorphism  $\widetilde{\chi} : K\to\Spin(\fr{m})$ such that  $\chi=\lambda\circ\widetilde{\chi}$, where $\lambda : \Spin(\fr{m})\to\SO(\fr{m})$ is the double covering of   $\SO(\fr{m})\equiv \SO_{n}$.  We shall denote by $\kappa_{n} : \Cl^{\bb{C}}(\fr{m})\stackrel{\sim}{\rightarrow}\Ed(\Delta_\fr{m})$  the Clifford representation and by     $\Cl(X\otimes \phi):=\kappa_{n}(X)\psi=X\cdot \psi$  the Clifford multiplication between vectors and spinors,  see \cite{Agr03} for more details.   Set $\rho:=\kappa\circ \widetilde{\chi} : K\to\Aut(\Delta_{\fr{m}})$, where  $\kappa=\kappa_{n}|_{\Spin(\fr{m})} : \Spin(\fr{m})\to\Aut(\Delta_{\fr{m}})$  is the   spin representation.  The spinor bundle $\Sigma\to G/K$ is the homogeneous  vector bundle  associated to  $P:=G\times_{\widetilde{\chi}}\Spin(\fr{m})$ via the representation $\rho$, i.e. $\Sigma=G\times_{\rho}\Delta_{\fr{m}}$.   Therefore  we may   identify sections of $\Sigma$ with smooth functions $\varphi : G\to\Delta_{\fr{m}}$ such that  $\varphi(gk)=\kappa\big(\widetilde{\chi}(k^{-1})\big)\varphi(g)=\rho(k^{-1})\varphi(g)$ for any $g\in G, k\in K$. 
   
Choose a  $G$-invariant metric connection $\nabla$ on $G/K$, corresponding to a connection map $\Lambda\in{\rm Hom}_{K}(\fr{m}, \fr{so}(\fr{m}))$. The lift    $\widetilde{\Lambda}:=\lambda_{*}^{-1}\circ\Lambda  : \fr{m}\to\fr{spin}(\fr{m})$ induces   a  covariant derivative on spinor fields (which we still denote by the same symbol) $\nabla  :\Gamma(\Sigma)\to \Gamma(T^{*}(G/K)\otimes\Sigma)$,   given by $\nabla_{X}\psi = X(\psi)+\widetilde{\Lambda}(X)\psi$.  Here, the vector  $X\in\fr{m}$ is considered  as a left-invariant vector field in $G$ and $\widetilde{\Lambda}(X)\psi$ as an equivariant function   $\widetilde{\Lambda}(X)\psi : G\to\fr{m}$.  The  Dirac operator  $D:=\Cl\circ \nabla : \Gamma(\Sigma)\to\Gamma(\Sigma)$ associated to  $\nabla$  is defined as follows (cf. \cite{Agr03}):
 \[
D(\psi):=\sum_{i}\kappa_{n}(Z_{i})\{Z_{i}(\psi)+  \widetilde{\Lambda}(Z_{i})\psi\}= \sum_{i} Z_{i}\cdot \{Z_{i}(\psi)+  \widetilde{\Lambda}(Z_{i})\psi\},
\]
where $Z_{i}$ denotes a $\langle \ , \ \rangle$-orthonormal basis of $\fr{m}$.  

\begin{remark}
\textnormal{Given a spin Riemannian manifold $(M, g)$ endowed with a metric connection $\nabla$,  basic properties of the induced Dirac operator $D=\Cl\circ\nabla$ are reflected in the type of the torsion of $\nabla$.  For example, by a result of Th.~Friedrich \cite{Fr79} (see also  \cite{FrS, Pfa}), it is known that  the formal self-adjointness of the Dirac operator $D=\Cl\circ\nabla$   is equivalent to the condition $A\in\cal{A}_{2}\oplus \cal{A}_{3}$, where $A=\nabla-\nabla^{g}$. Hence, in our case as an immediate consequence of  Corollary \ref{traceless} we obtain that}
\end{remark} 
 \begin{corol}\label{spin}
Let $(M=G/K, g)$ be an effective, non-symmetric (compact) SII homogeneous Riemannian manifold, endowed with a $G$-invariant metric connection $\nabla$ and a $G$-invariant spin structure.    Then, the Dirac operator  $D$ associated to $\nabla$ is formally self-adjoint. 
 \end{corol}
Note that the classification of invariant spin structures on non-symmetric SII spaces is  an open problem (see \cite{Cahen} for invariant spin structures on symmetric spaces and \cite{spin} for a more recent study of spin structures on reductive homogeneous spaces).
 
  \subsection{Classification results on invariant connections}

    For the presentation of the classification results, we use the notation of  \cite[p.~299]{Oni2}. In particular,  for a compact simple Lie group $G$   we  shall denote by $R(\pi)$  the complex  irreducible representation of highest weight   $\pi$.  We mention that the  isotropy representation  of a compact, non-symmetric,  effective SII space  turns out to be of either real or complex type.    In fact,  fixing a reductive decomposition $\fr{g}=\fr{k}\oplus\fr{m}$, whenever the complexification $\fr{m}^{\bb{C}}$ splits into   two complex submodules, these  are never equivalent representations (see also \cite{Wolf}).  Hence, by Schur's lemma  we have the identification  ${\rm Hom}_{K}(\fr{m}, \fr{m})=\bb{C}$ for complex type and   ${\rm Hom}_{K}(\fr{m}, \fr{m})=\bb{R}$ for real type.  In the first case, the endomorphism  $J$ induced by $i\in\bb{C}$ makes $G/K$ a  homogeneous almost complex manifold.   Note that the same conclusions are  true for a symmetric space, see \cite{Laq1, Laq2} (recall that the adjoint representation of a compact simple Lie group is always of real type).  

\begin{remark}\label{wCley}
\textnormal{The multiplicities that we describe below  have also been  presented in the PhD thesis  \cite{Cleyton}  (see Tables I.3.1--I.3.4, pp.~77--79), for a different however aim, namely the description of the components of   the intrinsic torsion associated to (irreducible) $G$-structures over  non-symmetric compact SII spaces (see also \cite{CSwan}).   We remark that there are a few errors/omissions in \cite{Cleyton}, related with some low-dimensional cases, namely:
\begin{itemize}
\item the case  $p=2, q\geq 3$ of the family  $\SU_{pq}/\SU_{p}\times \SU_{q}$,
\item the case $n=5$ of the family $\SU_{\frac{n(n-1)}{2}}/\SU_{n}$,
\item the case $n=6$ of the  family $\SO_{\frac{(n-1)(n+2)}{2}}/\SO_{n}$ (due to isomorphism $\fr{so}(6)=\fr{su}(4)$).  
\end{itemize}
  In these mentioned cases, the general decompositions of $\La^{2}\fr{m}$ or $\Sym^{2}\fr{m}$ change and most times this affects to  multiplicities that we are interested in.  Notice also that   for the manifold $\SO_{4n}/\Sp_{n}\times\Sp_{1}$  the enumeration in \cite{Wolf, Bes} starts for $n\geq 2$ (as we do), but in \cite{Cleyton}  it is written $n\geq 3$.   We correct these errors in our Table  \ref{table:four} (they are indicated by an asterisk). 
 Notice finally  that the author  of this thesis uses the \text{LiE} program (as we do) and for infinite families he is based on stability arguments, see \cite[Rem.~I.3.9]{Cleyton}.  Below    we also give examples of how such families can be treated even without the aid of a computer.}
  \end{remark}
\begin{remark}\label{knowu}
\textnormal{Given a  reductive  homogeneous space $M=G/K$  of a classical simple Lie group $G$, there is a simple method for the computation of   the associated isotropy representation $\chi : K\to\Aut(\fr{m})$, given as follows.  Let us  denote  by $\rho_{n} : \SO_{n}\to \Aut(\bb{R}^{n})$, $\mu_{n} : \SU_{n}\to\Aut(\bb{C}^{n})$ and $\nu_{n} :  \Sp_{n}\to\Aut(\bb{H}^{n})$ be the standard   representations of $\SO_{n}$, $\SU_{n}$ (or $\U_{n}$), and $\Sp_{n}$, respectively. Recall that the complexified adjoint representation $\Ad_{G}^{\bb{C}}=\Ad_{G}\otimes\bb{C}$, satisfies
 \[
 \Ad_{\SO_{n}}^{\bb{C}}=\Lambda ^{2}\rho_{n}, \quad  \Ad_{\U_{n}}^\bb{C}=\mu_{n}\otimes\mu_{n}^{*}, \quad \Ad_{\SU_{n}}^\bb{C}\oplus 1=\mu_{n}\otimes\mu_{n}^{*}, \quad  \Ad_{\Sp_{n}}^\bb{C}={\Sym}^{2}\nu_{n},
  \]
  where  $\mu_{n}^{*}$ is the dual representation of $\mu_{n}$ and $1$ denotes the trivial 1-dimensional representation.
Let $G$ be one of the Lie groups $\SO_{n}, \SU_{n}, \Sp_{n}$ and let $\pi : K\to G$ be an (almost) faithful representation of a compact connected subgroup $K$. Using the identity  $ \Ad\big|_{K}=\Ad_{K}\oplus\chi$, we see that the isotropy representation $\chi$ of $G/\pi(K)$ is determined by   $\Lambda^{2}\pi=\ad_{\fr{k}}\oplus\chi$ in the orthogonal case, by $\pi\otimes\pi^{*}=1\oplus\ad_{\fr{k}}\oplus\chi$ in the unitary case and finally by 
  $\Sym^{2}\pi=\ad_{\fr{k}}\oplus\chi$  in the symplectic   case (cf. \cite{Wolf, Wa5}). }
  \end{remark}
 
  \begin{theorem}\label{class}
  Let $(M=G/K, g=-B|_{\fr{m}})$ be an effective,  non-symmetric (compact) SII  homogeneous space. Consider the $B$-orthogonal reductive decomposition $\fr{g}=\fr{k}\oplus\fr{m}$.  Then, the complexified isotropy representation $\fr{m}^{\bb{C}}$ and    the multiplicities ${\bold a}$, ${\bold s}$, $\cal{N}$ and $\ell$ are given in Tables \ref{table:four} and \ref{table:five}.
  \end{theorem}
  
  \subsection{On the  Theorems A.1, A.2  and B -- Conclusions}
 The results in Tables \ref{table:four} and \ref{table:five} allows us to deduce that several non-symmetric  SII spaces are carrying {\it new} families of invariant  metric connections, in the sense that they are different from  the   Lie bracket  family  $\eta^{\al}(X, Y)=\frac{1-\al}{2}[X, Y]_{\fr{m}}$  (see Lemma \ref{mainuse1}).   
 In combination with Lemma \ref{dimskew}, we  also certify the existence of  compact, effective, non-symmetric SII quotients $M=G/K$ which are endowed with additional  families of $G$-invariant metric connections with skew-torsion, besides $\eta^{\al}$.   In full details,  this occurs in the following two situations:
\begin{itemize}
\item when $2 \leq \ell \leq {\bold a}$ and the isotropy representation  is not of complex type (since for complex type we may have $\ell=2={\bold a}$, but due to  Schurs's lemma all these invariant connections must be exhausted  by the family $\eta^{\al}(X, Y)=\frac{1-\al}{2}[X, Y]_{\fr{m}}$ with $\al\in\bb{C}$), or
\item  when the isotropy representation is of complex type but  $\ell$ (and hence ${\bold a}$) is  strictly greater than 2.
\end{itemize}
This observation, in combination with Lemmas \ref{ourr}, \ref{dimskew} and the results in Tables \ref{table:four} and \ref{table:five}, yields Theorems A.1 and A.2.    Theorem B  it is also a direct conclusion  of  the multiplicity ${\bold s}$  given  in Tables \ref{table:four} and \ref{table:five} and Lemma \ref{nicea}.  In fact, for affine connections induced by symmetric elements $\mu\in{\rm Hom}_{K}(\Sym^{2}\fr{m}, \fr{m})$, we also conclude that  
\begin{corol}\label{symmsun}
Let $(M=G/K, g=-B|_{\fr{m}})$ be an effective, non-symmetric, SII homogeneous space associated to the Lie group $G=\SU_{n}$. Then, there is always  a copy of $\fr{m}$ inside $\Sym^{2}\fr{m}$, induced by the restriction of the $\Ad(\SU_{n})$-invariant symmetric bilinear mapping 
\[
\eta : \fr{su}(n)\times\fr{su}(n)\to\fr{su}(n), \quad \eta(X, Y):=i\big\{XY+YX-\frac{2}{n}\tr(XY)\Id\big\}
\]
     on the corresponding reductive complement $\fr{m}$.  If $M$ is isometric to one of the manifolds
\[
\SU_{10}/\SU_{5}, \quad \SU_{2q}/\SU_{2}\times\SU_{q} \ (q\geq 3), \quad \SU_{16}/\Spin_{10},
\]
then the 1-parameter family of $\SU_{n}$-invariant affine connections on $M=\SU_{n}/K$ associated to  the restriction $\eta|_{\fr{m}} : \fr{m}\times\fr{m}\to\fr{m}$, exhausts all   $\SU_{n}$-invariant affine connections  induced by some $0\neq\mu\in{\rm Hom}_{K}(\Sym^{2}\fr{m}, \fr{m})$.
\end{corol}
\begin{proof}
The first part is based on  \cite[Thm.~6.1]{Laq2}. Notice that $\eta$ is known by \cite[p.~550]{Laq1}. Now, using the  results of Tables \ref{table:four} and \ref{table:five} about the multiplicity ${\bold s}$ of $\fr{m}$ inside $\Sym^{2}\fr{m}$, we obtain the result.
\end{proof}
  
{\small 
 \begin{table}
 \caption{The multiplicities ${\bold a}$, ${\bold s}$, $\cal{N}$ and $\ell$ for  (non-symmetric) SII homogeneous spaces--Classical families}
 \label{table:four}
\begin{tabular}{lllccccc}
  \multicolumn{8}{l}{Classical families and their associated low-dimensional cases}  \\
  \thickline
\ $G$ & $M=G/K$   &   $\fr{m}^{\bb{C}}$ & ${\bold a}$ & ${\bold s}$ & ${\cal{N}}$ & $\ell$ & {\rm Type}  \\
   \thickline 
  \multirow{15}{*}{}     $\SU_{n}$ 
       & $1) \  \SU_{\frac{n(n-1)}{2}}/\SU_{n}$ \ $(n\geq 6)$   & $R(\pi_{2}+\pi_{n-2})$ & 1& 2 & 3 & 1    & r   \\ \cmidrule(l){2-3}
            & $1_{\al}^{*}) \  \SU_{10}/\SU_{5}$    &$R(\pi_{2}+\pi_{3})$ & 1 & 1 & 2 & 1    & r \\ \cmidrule(l){2-3}
              & $2) \ \SU_{\frac{n(n+1)}{2}}/\SU_{n}$ \ $(n\geq 3)$    & $R(2\pi_{1}+2\pi_{n-1})$ & 1 & 2 & 3 & 1   &  r  \\ \cmidrule(l){2-3}
& $3) \  \SU_{pq}/\SU_{p}\times \SU_{q}$ \ $(p, q\geq 3)$ &  $R(\pi_{1}+\pi_{p-1})\hat{\otimes} R(\pi_{1}+\pi_{q-1})$ & 2 & 2 &  4 & 2 & r \\\cmidrule(l){2-3}
& $3_{\al}^{*}) \    \SU_{2q}/\SU_{2}\times\SU_{q}$ \ $(q\geq 3)$ & $R(2\pi_{1})\hat{\otimes} R(\pi_{1}+\pi_{q-1})$ & 1 & 1 & 2 & 1  & r  \\\midrule[0.08em]
\  $\SO_{n}$  &
  $4) \ \SO_{n^{2}-1}/\SU_{n}$ \ $(n\geq 4)$   & $R(2\pi_{1}+\pi_{n-2})\oplus R(\pi_{2}+2\pi_{n-1})$  & 6   & 2 & 8 & 4   & c        \\ \cmidrule(l){2-3}
  & $4_{\al}) \  \SO_{8}/\SU_{3}$ & $R(3\pi_{1})\oplus R(3\pi_{2})$ & 2 & 0  & 2 & 2   &   c  \\ \cmidrule(l){2-3}
  & $5) \ \SO_{\frac{n(n-1)}{2}}/\SO_{n}$ \ $(n\geq 9)$ & $R(\pi_{1}+\pi_{3})$ & 3  & 1 &  4 & 2    & r   \\ \cmidrule(l){2-3}
   & $5_{\al}) \  \SO_{21}/\SO_{7}$ & $R(\pi_{1}+2\pi_{3})$ & 3 & 1 & 4 & 2 &  r \\ \cmidrule(l){2-3}
    & $5_{\be}) \ \SO_{28}/\SO_{8}$ & $R(\pi_{1}+\pi_{3}+\pi_{4})$ & 4 & 3 & 7 & 2 &   r \\ \cmidrule(l){2-3}
    & $6) \ \SO_{\frac{(n-1)(n+2)}{2}}/\SO_{n}$ \ $(n\geq 7)$ & $R(2\pi_{1}+\pi_{2})$  & 3  & 1 & 4 &  2   &   r  \\ \cmidrule(l){2-3}
    &  $6_{\al}) \ \SO_{14}/\SO_{5}$ & $R(2\pi_{1}+2\pi_{2})$ & 3 & 1 & 4 & 2   & r  \\ \cmidrule(l){2-3}
& $6_{\be}^{*}) \ \SO_{20}/\SO_{6}$ & $R(2\pi_{1}+\pi_{2}+\pi_{3})$ & 3 & 2 & 5 & 2   & r  \\ \cmidrule(l){2-3}
& $7) \ \SO_{(n-1)(2n+1)}/\Sp_{n}$ \ $(n\geq 4)$ &  $R(\pi_{1}+\pi_{3})$ & 3 & 1  & 4 & 2   & r    \\ \cmidrule(l){2-3}
   & $7_{\al}) \ \SO_{14}/\Sp_{3}$ & $R(\pi_{1}+\pi_{3})$ & 1 & 0 &  1 & 1   & r \\ \cmidrule(l){2-3}
 & $8) \ \SO_{n(2n+1)}/\Sp_{n}$ \ $(n\geq 3)$ & $R(2\pi_{1}+\pi_{2})$ & 3  & 1 & 4  &  2   & r    \\ \cmidrule(l){2-3}
& $8_{\al}) \ \SO_{10}/\Sp_{2}$ & $R(2\pi_{1}+\pi_{2})$ & 2 & 1 & 3 & 1   & r \\ \cmidrule(l){2-3}
 & $9^{*})  \ \SO_{4n}/\Sp_{n}\times\Sp_{1}$ \ $(n\geq 2)$ &  $R(\pi_{2})\hat{\otimes} R(2\pi_{1})$ & 1 & 0  & 1 & 1 &    r  \\\midrule[0.08em]
\ $\Sp_{n}$ 
& $10) \ \Sp_{n}/\SO_{n}\times\Sp_{1}$ \ $(n\geq 5)$ & $R(2\pi_{1})\hat{\otimes} R(2\pi_{1})$ & 1 & 0 & 1 & 1 &     r \\ \cmidrule(l){2-3}
& $10_{\al}) \ \Sp_{3}/\SO_{3}\times\Sp_{1}$  & $R(4\pi_{1})\hat{\otimes} R(2\pi_{1})$ & 1 & 0 & 1 & 1 & r \\\cmidrule(l){2-3}
 & $10_{\be}) \ \Sp_{4}/\SO_{4}\times\Sp_{1}$  & $R(2\pi_{1}+2\pi_{2})\hat{\otimes} R(2\pi_{1})$ & 1 & 0 & 1 & 1 &  r \\\bottomrule
 \end{tabular}
\end{table}
}
  

 {\small   \begin{table}
    \caption{The multiplicities ${\bold a}$, ${\bold s}$, $\cal{N}$ and $\ell$   for (non-symmetric) SII  homogeneous spaces--Exceptions}
    \label{table:five}
\begin{tabular}{lllccccc}
  \multicolumn{8}{l}{Exceptions (``exceptions'' in terms of \cite[p.~203]{Bes})}  \\
  \thickline
\  $G$ & $M=G/K$   &   $\fr{m}^{\bb{C}}$ & ${\bold a}$ & ${\bold s}$ & ${\cal{N}}$ & $\ell$ & {\rm Type}  \\
   \thickline 
  \multirow{15}{*}{}   $\SU_{n}$ &$\SU_{16}/\Spin_{10}$ & $R(\pi_4 + \pi_5)$ & 1 & 1 & 2 & 1   & r   \\ \cmidrule(l){2-3}
                             & $\SU_{27}/\E_6$ & $R(\pi_1 + \pi_6)$ & 1 & 2 & 3 & 1 & r  \\\midrule[0.08em]
     \   $\SO_{n}$    &$\SO_{7}/\G_2$ & $R(\pi_1)$ & 1 & 0 & 1 & 1 & r  \\ \cmidrule(l){2-3}
                             & $\SO_{14}/\G_2$ & $R(3\pi_1)$ & 2 & 0 & 2 & 2  & r \\ \cmidrule(l){2-3}
                             & $\SO_{16}/\Spin_{9}$ & $R(\pi_3)$ & 1 & 0 & 1 & 1 & r \\ \cmidrule(l){2-3}
                             & $\SO_{26}/\F_4$ & $R(\pi_3)$ & 2 & 0 & 2 & 2 & r \\ \cmidrule(l){2-3}
                              & $\SO_{42}/\Sp_{4}$ & $R(2\pi_3)$ & 2 & 0 & 2 & 2 & r \\ \cmidrule(l){2-3}
                              &$\SO_{52}/\F_4$ & $R(\pi_2)$ & 2 & 0 & 2 & 2 & r \\ \cmidrule(l){2-3}
                              &$\SO_{70}/\SU_{8}$ & $R(\pi_3 + \pi_5)$ & 2 & 1 & 3 & 2 & r \\ \cmidrule(l){2-3}
                            &$\SO_{248}/\E_8$ & $R(\pi_7)$ & 2 & 0 & 2 & 2 & r \\ \cmidrule(l){2-3}
                             &$\SO_{78}/\E_6$ & $R(\pi_4)$ & 2 & 0 & 2 & 2 & r \\ \cmidrule(l){2-3}
                             &$\SO_{128}/\Spin_{16}$ & $R(\pi_6) $ & 2 & 0 & 2 & 2 & r \\ \cmidrule(l){2-3}
                             &$\SO_{133}/\E_7$ & $R(\pi_3) $ & 2 & 0 & 2 & 2 & r     \\\midrule[0.08em]
\ $\Sp_{n}$  &        $\Sp_{2}/\SU_{2}$ & $R(6\pi_{1})$ & 1 & 0 & 1 & 1 & r \\ \cmidrule(l){2-3}
                             & $\Sp_{7}/\Sp_{3}$ & $R(2\pi_{3})$ & 1 & 0 & 1 & 1 &  r \\ \cmidrule(l){2-3}
                            & $\Sp_{10}/\SU_{6}$ & $R(2\pi_{3})$ & 1 & 0 & 1 & 1 & r \\ \cmidrule(l){2-3}
                            & $\Sp_{16}/\Spin_{12}$ & $R(2\pi_{6})$ or $R(2\pi_{5})$  & 1 & 0 & 1 & 1 & r \\ \cmidrule(l){2-3}
                            &$\Sp_{28}/\E_7$ & $R(2\pi_7)$& 1 & 0 & 1 & 1 & r  \\\midrule[0.08em]
\ $\G_2$                & $\G_2 / \SU_{3}$ & $R(\pi_1) \oplus R(\pi_2)$ & 2 & 0 & 2 & 2 & c \\ \cmidrule(l){2-3}
                            &$\G_2/\SO_3$ & $R(10\pi_{1})$ & 1 & 0 & 1 & 1 & r \\\midrule[0.08em]
\ $\F_4$                & $\F_4/(\SU_3^{1}\times\SU_3^{2})$ & $(R(2\pi_{1})\hat{\otimes} R(\omega_{1}))\oplus (R(2\pi_{2})\hat{\otimes} R(\omega_{2}))$ & 2 & 0 & 2 & 2 & c  \\ \cmidrule(l){2-3}
                           & $\F_4/(\G_2\times\SU_2)$ & $R(\pi_{1})\hat{\otimes} R(4\omega_{1})$  & 1 & 0 & 1 & 1 & r  \\\midrule[0.08em]
\ $\E_6$                & $\E_6/\SU_{3}$ & $R(4\pi_1 + \pi_2)\oplus R(\pi_{1}+4\pi_{2})$ & 6 & 4 & 10 & 4  & c \\ \cmidrule(l){2-3}
                           & $\E_6/(\SU_{3}\times\SU_3\times\SU_3)$ & $(R(\pi_{1})\hat{\otimes} R(\omega_{1})\hat{\otimes} R(\theta_{1}))\oplus $ & 2 & 0 & 2 & 2 & c \\ 
                           & & $(R(\pi_{2})\hat{\otimes} R(\omega_{2})\hat{\otimes} R(\theta_{2})) $ & & & &   &   \\ \cmidrule(l){2-3}
                           & $\E_6/\G_2$ & $R(\pi_1 + \pi_2)$& 1 & 1& 2 & 1 & r \\ \cmidrule(l){2-3} 
                           & $\E_6/(\G_2\times\SU_3)$ & $R(\pi_{1})\hat{\otimes} R(\omega_{1}+\omega_{2})$ & 1 & 1 & 2 & 1 & r   \\\midrule[0.08em]
\ $\E_7$                 & $\E_7/\SU_3$ & $R(4\pi_1 +4\pi_2)$ & 2 & 3 & 5 & 2  & r \\ \cmidrule(l){2-3} 
                             & $\E_7/(\SU_3\times\SU_6)$ & $(R(\pi_{1})\hat{\otimes} R(\omega_{2}))\oplus (R(\pi_{2})\hat{\otimes} R(\omega_{5}))$ & 2& 0 & 2 & 2 & c \\ \cmidrule(l){2-3}
                             & $\E_7/(\G_2\times\Sp_{3})$ & $R(\pi_{1})\hat{\otimes} R(\omega_{2})$ & 1 & 0 & 1 & 1 & r \\ \cmidrule(l){2-3}
                             & $\E_7/(\F_4\times\SU_2)$ & $ R(\pi_{4})\hat{\otimes} R(2\omega_{1})$ & 1 & 0 & 1 & 1 &  r \\\midrule[0.08em]
\ $\E_8$                 & $\E_8/\SU_9$ & $R(\pi_{3})\oplus R(\pi_{6})$ & 2 & 0 & 2 & 2 & c \\ \cmidrule(l){2-3} 
                             & $\E_8/(\F_4\times\G_2)$ & $R(\pi_{4})\hat{\otimes} R(\omega_{1})$ & 1 & 0 & 1 & 1 &  r \\ \cmidrule(l){2-3} 
                              & $\E_8/(\E_6\times\SU_3)$ & $(R(\pi_{1})\hat{\otimes} R(\omega_{1}))\oplus (R(\pi_{6})\hat{\otimes} R(\omega_{2}))$ & 2 & 0 & 2 & 2 & c  \\\bottomrule
                               \end{tabular}
\end{table}}
        
\subsection{Some explicit  examples}\label{examples}
 Let us now  compute the desired multiplicities ${\bold a}$, ${\bold s}$ and $\ell$, for general families of (non-symmetric) SII spaces, without the aid of computer.  For this we need first to recall some preliminaries of representation theory (for more details we refer to \cite{Broker, Simon, Cleyton}).

If $\pi$ is a complex representation of a compact Lie algebra $\fr{k}$, then $\overline{\pi}\cong\pi^{*}$, where $\overline{\pi}$ denotes the complex conjugate representation and $\pi^{*}$ the dual representation.
 If $\pi$ is a complex representation of $\fr{k}$ on $V$, then  there is a   symmetric (resp. skew-symmetric) non-degenerate bilinear form on $V$ invariant under $\pi$, if and only if there is a   anti-linear intertwining map $\tau$ with $\tau^{2}=\Id$ (resp. $\tau^{2}=-\Id$) \cite[Prop. ~6.4.]{Broker}. If $\pi$ is irreducible then Schur's lemma ensures the uniqueness of such a bilinear form. 
 A complex representation carrying a   conjugate linear intertwining map $\tau$ with $\tau^{2}=\Id$  (resp. $\tau^{2}=-\Id$)  is called of real type (resp. quaternonic type).  
  Finally we call a  complex representation $\pi : \fr{k}\to\fr{gl}(V)$  of complex type if it is not self dual, i.e. $V\ncong V^{*}$.

 Let $(\pi, V)$ and $(\pi', W)$ be representations of a connected (not necessarily compact) Lie group $H$ on two vector spaces $V$ and $W$, respectively.  It  is important to note that even if $\pi$ and $\pi'$ are irreducible, then the   tensor product representation $V\otimes W$, defined by $\pi\otimes\pi' : H\to \Aut(V\otimes W)$,  $(\pi\otimes\pi')(h)(u\otimes w)=\pi(h)u\otimes\pi'(h)w$,  is always reducible. 
Let us  denote by $\Lambda^{k}\pi$  and  $\Sym^{k}\pi$ the $k$-th exterior power  and $k$-th symmetric power, respectively. For $k=2$,  it is easy to prove that
\begin{equation}\label{usef1}
\left\{\begin{tabular}{l l l}
$\Lambda^{2}(V\oplus W)$&$=$&$\Lambda^{2}V\oplus (V\otimes W)\oplus\Lambda^{2}W$,\\
$\Sym^{2}(V\oplus W)$&$=$&$\Sym^{2}V\oplus (V\otimes W)\oplus\Sym^{2}W$.
\end{tabular}\right.
\end{equation} 
If $(\pi, V)$ and $(\pi', W)$  are representations  of two connected  Lie groups $H$ and $H'$, respectively, then the  vector space $V\otimes W$ carries a representation of the product group $H\times H'$, say $(\pi\hat{\otimes}\pi', V\otimes W)$, given   by $\pi\hat{\otimes}\pi'(h, h')(u\otimes w)=\pi(h)u\otimes \pi'(h')w$.  This representation is called the   external tensor product  of $\pi$ and $\pi'$.  In the finite-dimensional case, $\pi\hat{\otimes}\pi'$ is an irreducible representation of $H\times H'$, if and only $\pi$ and $\pi'$ are both irreducible. In particular,  if $H, H'$ are compact Lie groups, then a representation of $H\times H'$ in ${\rm GL}(\bb{C}^{n})$ is irreducible if and only if it is the tensor product of an irreducible  representation of $H$ with one of $H'$.
 Finally,  one has the following equivariant isomorphisms:
 \begin{equation}\label{usef2}
\left\{\begin{tabular}{l l l}
$(V\hat{\otimes}W)\otimes (V'\hat{\otimes}W')$&$=$&$(V\otimes V')\hat{\otimes}(W\otimes W')$\\
$\Lambda^{2}(V\hat{\otimes}W)$&$=$&$(\Lambda^{2}V\hat{\otimes}\Sym^{2}W)\oplus(\Sym^{2}V\hat{\otimes}\Lambda^{2}W)$\\
$\Sym^{2}(V\hat{\otimes}W)$&$=$&$(\Sym^{2}V\hat{\otimes}\Sym^{2}W)\oplus(\Lambda^{2}V\hat{\otimes}\Lambda^{2}W)$
\end{tabular}\right.
\end{equation} 
 We finally remark that if  $V, W$ are  complex irreducible representations of two compact Lie groups $H$ and $H'$ respectively, then  $V\hat{\otimes}W$ is of real type if $V, W$ are both of real type or both of quaternionic type,  $V\hat{\otimes}W$ is if complex type if at least one of $V, W$ are of complex type and finally,  $V\hat{\otimes}W$ is of quaternonic type if one of $V, W$ is of real type and the other one of quaternionic type.
\begin{lemma}\label{erref} 
Consider the homogeneous space $M_{p, q}:=G/K=\SU_{pq}/\SU_{p}\times \SU_{q}$ with $p, q>2, p+q>4$. Then, the multiplicities of the  isotropy representation   $\fr{m}=\Ad_{\SU_{p}}\hat{\otimes}\Ad_{\SU_{q}}$  inside $\Lambda^{2}\fr{m}$ and $\Sym^{2}\fr{m}$    are given as follows: for $p=2, q\geq 3$ it is ${\bold a}={\bold s}=1$, while for   $p, q\geq 3$  it is ${\bold a}={\bold s}=2$. Moreover, the dimension  of the trivial  submodule $(\La^{3}\fr{m})^{K}$   is   $\ell=1$ for $p=2,q\geq 3$ and $\ell=2$ for $p,q\geq 3$.
\end{lemma}
\begin{proof}
 The inclusion $\pi : K\to G$ is given by the external tensor product of the standard representations $\mu_{p}$ and $\mu_{q}$ of $\SU_{p}$ and $\SU_{q}$, respectively.  Thus, a  short application of Remark \ref{knowu} in combination with the relation $\Ad^{\bb{C}}_{A_{n}}=R(\pi_{1}+\pi_{n})$,  yields  that
 \[
 \fr{m}^{\bb{C}}=(\Ad^{\bb{C}}_{\SU_{p}}\hat{\otimes}\Ad^{\bb{C}}_{\SU_{q}})=R(\pi_{1}+\pi_{p-1})\hat{\otimes}R(\pi_{1}+\pi_{q-1}).
 \]
    Consequently,   the    isotropy representation $\fr{m}=\Ad_{\SU_{p}}\hat{\otimes}\Ad_{\SU_{q}}$ of $M_{p, q}=\SU_{pq}/\SU_{p}\times \SU_{q}$ is   irreducible over $\bb{R}$ and   is of real type, since it is the external  tensor product of two representations of real type. 
   Now, by \cite{Oni2} we also know that 
   \begin{eqnarray}
   \Lambda^{2}\Ad_{A_{n}}^{\bb{C}}&=&R(2\pi_{1}+\pi_{n-1})\oplus R(\pi_{2}+2\pi_{n})\oplus\Ad^{\bb{C}}_{A_{n}}, \quad n\geq 3 \label{l2}  \\
    \Sym^{2}\Ad_{A_{n}}^{\bb{C}}&=& \left\{
  \begin{tabular}{l   l}
  $R(2\pi_{1}+2\pi_{n})\oplus R(\pi_{2}+\pi_{n-1})\oplus\Ad^{\bb{C}}_{A_{n}}\oplus 1$, & if $n\geq 3,$ \\
  $R(2\pi_{1}+2\pi_{2})\oplus\Ad^{\bb{C}}_{A_{n}}\oplus 1$, & if $n=2$.
  \end{tabular}\right.\label{s2}
   \end{eqnarray}
Certainly, for $\fr{su}_{2}=\fr{so}_{3}=\fr{sp}_{1}$  one gets a 3-dimensional irreducible representation $\Lambda^{2}(\Ad_{\SU_{2}}^\bb{C})\cong\Ad^{\bb{C}}_{\SU_{2}}=R(2\pi_{1})$.  Moreover, it is $\Sym^{2}(\Ad_{\SU_{2}}^\bb{C})=R(4\pi_{1})\oplus 1$.  Notice  also that $\Lambda^{2}(\Ad_{\SU_{3}}^\bb{C})=R(3\pi_{1})\oplus R(3\pi_{2})\oplus \Ad_{\SU_{3}}^\bb{C}$, since $\Ad_{\SU_{3}}^\bb{C}=R(\pi_{1}+\pi_{2})$.
Due to this small speciality of $\SU_{2}$ and   the different decomposition of $\Sym^{2}\Ad_{A_{n}}^{\bb{C}}$ (for $n=2$ and $n\geq 3$, respectively) one has to separate the examination into   two cases:

 \noindent { \bf Case A:}  $p=2$, $q\geq 3$. Then we have $M_{2q}=\SU_{2q}/\SU_{2}\times\SU_{q}$  and 
 \[
 \fr{m}^{\bb{C}}=\Ad^{\bb{C}}_{\SU_{2}}\hat{\otimes}\Ad^{\bb{C}}_{\SU_{q}}=R(2\pi_{1})\hat{\otimes} R(\pi_{1}+\pi_{q-1}), \quad (q\geq 3).
 \] 
 Hence, a combination  of  (\ref{usef2}), (\ref{l2}), (\ref{s2}) and $\Lambda^{2}(\Ad_{\SU_{2}}^\bb{C})=\Ad^{\bb{C}}_{\SU_{2}}=R(2\pi_{1})$ shows that
 \begin{eqnarray*}
 \Lambda^{2}(\fr{m}^{\bb{C}})&=&\big(R(2\pi_{1})\hat{\otimes} \Sym^{2}\Ad_{\SU_{q}}^{\bb{C}}\big)\oplus\big((R(4\pi_{1})\oplus 1)\hat{\otimes}\Lambda^{2}\Ad_{\SU_{q}}^{\bb{C}}\big)\\
 &=&\big( R(2\pi_{1})\hat{\otimes} R(2\pi_{1}+2\pi_{q-1})\big)\oplus \big(R(2\pi_{1})\hat{\otimes} R(\pi_{2}+\pi_{q-2}) \big)\oplus \big(\Ad^{\bb{C}}_{\SU_{2}}\hat{\otimes} \Ad^{\bb{C}}_{\SU_{q}}\big)\oplus R(2\pi_{1}) \\
 &&\oplus\big( R(4\pi_{1})\hat{\otimes} R(2\pi_{1}+\pi_{q-2})\big)\oplus \big(R(4\pi_{1})\hat{\otimes} R(\pi_{2}+2\pi_{q-1})\big)\oplus \big(R(4\pi_{1})\hat{\otimes} \Ad^{\bb{C}}_{\SU_{q}}\big)\\\
 &&\oplus   R(2\pi_{1}+\pi_{q-2})\oplus R(\pi_{2}+2\pi_{q-1})\oplus\Ad^{\bb{C}}_{\SU_{q}}.
 \end{eqnarray*}
 We deduce that $\fr{m}^{\bb{C}}$ appears once inside  $\Lambda^{2}(\fr{m}^{\bb{C}})$ and since $\fr{m}$ is of real type, it follows that ${\bold a}=1$.\\
\noindent Let us treat  the decomposition of the second symmetric power. We start with the low dimensional case  $p=2, q=3$, i.e. $\fr{m}^{\bb{C}}=\Ad_{\SU_{2}}^{\bb{C}}\hat{\otimes}\Ad^{\bb{C}}_{\SU_{3}}$. A combination of (\ref{usef2}),
  (\ref{l2}) and (\ref{s2}), yields
 \begin{eqnarray*}
 \Sym^{2}(\fr{m}^{\bb{C}})&=&\Big(\big(R(4\pi_{1})\oplus 1\big)\hat{\otimes}\big(R(2\pi_{1}+2\pi_{2})\oplus\Ad^{\bb{C}}_{\SU_{3}}\oplus 1\big)\Big)\oplus \Big(\Ad_{\SU_{2}}^{\bb{C}}\hat{\otimes}\big(R(3\pi_{1})\oplus R(3\pi_{2})\oplus\Ad^{\bb{C}}_{\SU_{3}}\big)\Big)\\
 &=&\big(R(4\pi_{1})\hat{\otimes} R(2\pi_{1}+2\pi_{2})\big) \oplus \big(R(4\pi_{1})\hat{\otimes}\Ad^{\bb{C}}_{\SU_{3}}\big)\oplus R(4\pi_{1})\oplus R(2\pi_{1}+2\pi_{2})\oplus\Ad^{\bb{C}}_{\SU_{3}}\oplus 1\\
 &&\oplus \big(\Ad_{\SU_{2}}^{\bb{C}}\hat{\otimes}R(3\pi_{1})\big) \oplus  \big(\Ad_{\SU_{2}}^{\bb{C}}\hat{\otimes}R(3\pi_{2})\big) \oplus  \big(\Ad_{\SU_{2}}^{\bb{C}}\hat{\otimes}\Ad^{\bb{C}}_{\SU_{3}}\big).
 \end{eqnarray*}
Hence, there is a copy of  $\fr{m}^{\bb{C}}$ inside $\Sym^{ 2}(\fr{m}^{\bb{C}})$ and as above we conclude that ${\bold s}=1$.
In a similar way, for  $p=2$ and $q\leq 4$, we get that
  \begin{eqnarray*}
 \Sym^{2}(\fr{m}^{\bb{C}})&=&\Big(\big(R(4\pi_{1})\oplus 1\big)\hat{\otimes}\big(R(2\pi_{1}+2\pi_{q-1})\oplus R(\pi_{2}+\pi_{q-2})\oplus\Ad^{\bb{C}}_{\SU_{q}}\oplus 1\big)\\
 &&\oplus \Big(\Ad_{\SU_{2}}^{\bb{C}}\hat{\otimes} \big(R(2\pi_{1}+\pi_{q-2})\oplus R(\pi_{2}+2\pi_{q-1})\oplus\Ad^{\bb{C}}_{\SU_{q}}\big)\Big)\\
 &=&\big(R(4\pi_{1})\hat{\otimes}R(2\pi_{1}+2\pi_{q-1})\big)\oplus\big(R(4\pi_{1})\hat{\otimes}R(\pi_{2}+\pi_{q-2})\big)\oplus\big(R(\pi_{2}+\pi_{q-2})\hat{\otimes}\Ad^{\bb{C}}_{\SU_{q}}\big)\\
 &&\oplus R(4\pi_{1})\oplus R(2\pi_{1}+2\pi_{q-1})\oplus R(\pi_{2}+\pi_{q-2})\oplus\Ad^{\bb{C}}_{\SU_{q}}\oplus 1\\
 &&\oplus \big(\Ad_{\SU_{2}}^{\bb{C}}\hat{\otimes} R(2\pi_{1}+\pi_{q-2})\big)\oplus\big(\Ad_{\SU_{2}}^{\bb{C}}\hat{\otimes} R(\pi_{2}+2\pi_{q-1})\big)\oplus \big(\Ad_{\SU_{2}}^{\bb{C}}\hat{\otimes}\Ad^{\bb{C}}_{\SU_{q}}\big).
 \end{eqnarray*}
Thus again we conclude ${\bold s}=1$. 

   \noindent { \bf Case B:}  $3\leq p\leq q$.  In this case, a combination of  (\ref{usef2}), (\ref{l2}) and (\ref{s2})  yields the following decomposition for any $p\geq 3$ and $q\geq p$:
\begin{eqnarray*}
 \Lambda^{2}(\fr{m}^{\bb{C}})&=&\big(\Lambda^{2}R(\pi_{1}+\pi_{p-1})\hat{\otimes} \Sym^{2}R(\pi_{1}+\pi_{q-1})\big)\oplus\big(\Sym^{2}R(\pi_{1}+\pi_{p-1})\hat{\otimes} \Lambda^{2}R(\pi_{1}+\pi_{q-1})\big)\\
 &=&\big(R(2\pi_{1}+\pi_{p-2})\oplus R(\pi_{2}+2\pi_{p-1})\oplus\Ad^{\bb{C}}_{\SU_{p}}\big)\hat{\otimes}\big(R(2\pi_{1}+2\pi_{q-1})\oplus R(\pi_{2}+\pi_{q-2})\oplus\Ad^{\bb{C}}_{\SU_{q}}\oplus 1\big)\\
 &&\oplus \big(R(2\pi_{1}+2\pi_{p-1})\oplus R(\pi_{2}+\pi_{p-1})\oplus\Ad^{\bb{C}}_{\SU_{p}}\oplus 1\big)\hat{\otimes}\big(R(2\pi_{1}+\pi_{q-2})\oplus R(\pi_{2}+2\pi_{q-1})\oplus\Ad^{\bb{C}}_{\SU_{q}}\big).
\end{eqnarray*}
These two external tensor products each contain   one copy of $\fr{m}^{\bb{C}}$, hence we have  ${\bold a}=2$ in this case.
Passing to the second symmetric power and working in  the same way we get that
\begin{eqnarray*}
 \Sym^{2}(\fr{m}^{\bb{C}})&=&\big(\Sym^{2}\Ad^{\bb{C}}_{\SU_{p}}\hat{\otimes}\Sym^{2}\Ad^{\bb{C}}_{\SU_{q}}\big)\oplus\big(\Lambda^{2}\Ad^{\bb{C}}_{\SU_{p}}\hat{\otimes}\Lambda^{2}\Ad^{\bb{C}}_{\SU_{q}}\big)\\
 &=&\big(R(2\pi_{1}+2\pi_{p-1})\oplus R(\pi_{2}+\pi_{p-2})\oplus\Ad^{\bb{C}}_{\SU_{p}}\oplus 1\big)\hat{\otimes} \big(R(2\pi_{1}+2\pi_{q-1})\oplus R(\pi_{2}+\pi_{q-2})\oplus\Ad^{\bb{C}}_{\SU_{q}}\oplus 1\big)\\
 &&\oplus \big(R(2\pi_{1}+\pi_{p-2})\oplus R(\pi_{2}+2\pi_{p-1})\oplus\Ad^{\bb{C}}_{\SU_{p}}\big)\hat{\otimes} \big(R(2\pi_{1}+\pi_{q-2})\oplus R(\pi_{2}+2\pi_{q-1})\oplus\Ad^{\bb{C}}_{\SU_{q}}\big).
\end{eqnarray*}
It follows that there are two instances of $\fr{m}^{\bb{C}}$ inside $\Sym^{2}(\fr{m}^{\bb{C}})$, i.e. ${\bold s}=2$.

To compute the dimension of the space of invariant three-forms, consider the additional equivariant isomorphism
$$
\Lambda^3(V\otimes W) = \big(\Lambda^3 V \otimes \Sym^3 W\big) \oplus \big( P_{V}(2, 1)\otimes P_{W}(2, 1)\big) \oplus \big(\Sym^3 V \otimes \Lambda^3 W\big), 
$$
where $P_{V}(2, 1)$ and $P_{W}(2, 1)$ are the (2,1)-plethysms of $V$ and $W$, respectively. From this we deduce that any invariant $3$-form on $V\otimes W$ can be projected to component forms induced from the following cases:
\begin{itemize}
\item An invariant 3-form on $V$ and an invariant symmetric 3-tensor on $W$ 
\item The product of two invariant elements of the (2,1)-plethysms of $V$ and of $W$
\item An invariant symmetric 3-tensor on $V$ and an invariant 3-form on $W$ 
\end{itemize}
Let us first compute such objects for  $V:=\Ad^{\bb{C}}_{\SU_{p}}=R(\pi_1+\pi_{p-1})$. We have
$$
\Lambda^3R(\pi_1+\pi_{p-1})\oplus P_{R(\pi_1+\pi_{p-1})}(2, 1)=\Lambda^2R(\pi_1+\pi_{p-1})\otimes R(\pi_1+\pi_{p-1}),
$$
and by (\ref{l2}) we have $\Lambda^2\Ad^{\bb{C}}_{\SU_{p}}=\Lambda^2R(\pi_1+\pi_{p-1})=R(\pi_1+\pi_{p-1})\oplus R(2\pi_1+\pi_{p-2})\oplus R(\pi_2+2\pi_{p-1})$.
The first part of the computation is 
$$
\Ad^{\bb{C}}_{\SU_{p}}\otimes \Ad^{\bb{C}}_{\SU_{p}}=R(\pi_1+\pi_{p-1})\otimes R(\pi_1+\pi_{p-1}) = R(2\pi_1+2\pi_{p-1})\oplus R(\pi_1+\pi_{p-1})\oplus \bb{R},
$$
where the trivial term $\bb{R}$ corresponds to the tensor of the Killing form of $\fr{su}(p)$ (see Remark \ref{3forms}).  This yields the first invariant 3-form on $R(\pi_1+\pi_{p-1})$. Now we are left with the task of computing unsymmetrized tensor products of irreducible $\SU_{p}$-modules, and this can be done via the Littlewood-Richardson rule. The result is that the products $R(2\pi_1+\pi_{p-2})\otimes R(\pi_1+\pi_{p-1})$ and $R(\pi_2+2\pi_{p-1})\otimes R(\pi_1+\pi_{p-1})$, 
do not contain the trivial representation. Hence the (2,1)-plethysm of $R(\pi_1+\pi_{p-1})$ also does not admit any trivial modules.  By (\ref{s2}) we also deduce that the multiplicity of $V$ in $\Sym^2R(\pi_1+\pi_{p-1})$ is 0 for $p=2$, or 1 for $p\geq 3$. Furthermore, we have
$$
\Sym^3R(\pi_1+\pi_{p-1})\oplus P_{R(\pi_1+\pi_{p-1})}(2, 1)= \Sym^2R(\pi_1+\pi_{p-1})\otimes R(\pi_1+\pi_{p-1})
$$
and the tensor product between   $R(\pi_1+\pi_{p-1})\subset\Sym^2R(\pi_1+\pi_{p-1})$ and the rightmost factor $R(\pi_1+\pi_{p-1})$ contains a trivial submodule, corresponding as before to the tensor of the Killing form. This trivial submodule is contained in $\Sym^3 R(\pi_1+\pi_{p-1})$, since we have already shown that $P_{R(\pi_1+\pi_{p-1})}(2,1)$ contains no such.  Note that the computations for $W:=\Ad^{\bb{C}}_{\SU_{q}}=R(\pi_1+\pi_{q-1})$ are identical (except for switching the label $p$ to $q$). Therefore, we finally get a one dimensional trivial submodule in $\Lambda^3(V\otimes W)$ for the case $p=2$ (or $q$=2), and when $p,q\geq 3$ we get a two dimensional trivial submodule in $\Lambda^3(V\otimes W)$. In our notation, this means $\ell=1$ for $p=2,q\geq 3$ and $\ell=2$ for $p,q\geq 3$. This  completes the proof.
 \end{proof}

   \begin{lemma}
Consider the homogeneous space $M=G/K=\Sp_{n}/(\SO_{n}\times\Sp_{1})$ with $n\geq 3$. Then, the isotropy representation $\fr{m}=R(2\pi_1)\hat{\otimes} \fr{sp}(1)$ has multiplicity ${\bold a}=1$  inside  $\Lambda^2 \fr{m}$  and  multiplicity ${\bold s}=0$ inside  $\Sym^2 \fr{m}$. Moreover, the dimension of trivial submodule $(\La^{3}\fr{m})^{K}$ is $\ell=1$.
\end{lemma}
\begin{proof}
An embedding of a compact  Lie group $K$ into $\Sp_{n}$ is equivalent to  a (faithful) representation $\phi : K\rightarrow \Gl(\bb{H}^n)$.  This is a representation of real dimension $4n$ with an invariant quaternionic structure. Since $K=\SO_{n}\times \Sp_{1}$ is compact, the image of $\phi$ will be inside some conjugacy class of $\Sp_{n}$. We are looking for the unique isotropy irreducible embedding, which means that $\phi$ should be an irreducible representation. Let $R(\omega_1,\omega_2)=R(\omega_1)_{\SO_{n}}\hat{\otimes}_\bb{C} R(\omega_2)_{\Sp_{1}}$  denotes the associated real irreducible representation.
The obvious candidate is $\phi=R(\pi_1,\pi_1)=\bb{R}^n\hat{\otimes}_{\bb{R}} \bb{H}=\bb{R}^n\hat{\otimes}_{\bb{R}} \bb{C}^2$.  This irreducible representation  is obviously of quaternionic type. 
Recall now the adjoint representation  of $\Sp_{n}$ is the real submodule inside $\Ad^{\bb{C}}_{\Sp_{n}}=\Sym^{2}\nu_{n}=\Sym^2_{\bb{C}}\bb{H}^{n}$. Thus we must take into account the complex structure on $\phi$, which is defined by its action on the right tensor factor $\bb{H}\simeq \bb{C}^2$. By applying (\ref{usef2}), we compute
\begin{eqnarray*}
\Sym^2_{\bb{C}}\phi  &=& \Sym^2_{\bb{C}} (\bb{R}^n\hat{\otimes}_{\bb{R}} \bb{C}^2) = (\Sym^2_{\bb{R}}\bb{R}^n\hat{\otimes}_{\bb{R}} \Sym^2_{\bb{C}}\bb{C}^2)\oplus (\Lambda^2_{\bb{R}} \bb{R}^n \hat{\otimes} \Lambda^2_{\bb{C}} \bb{C}^2)\\
&=& (R(2\pi_1)\oplus R(0))\otimes \Ad_{\Sp_{1}}^{\bb{C}} \oplus \Ad_{\SO_{n}}^{\bb{C}}.
\end{eqnarray*}
This immediately yields the isotropy representation 
\[
\fr{m}=\fr{sp}(n)/(\fr{so}(n)\oplus\fr{sp}(1))=R(2\pi_1)\hat{\otimes} \fr{sp}(1) = R(2\pi_1,2\pi_1),
\]
which is irreducible. Since $\fr{m}$ has real type and its tensor factors also have real type, we can apply complex representation theory without performing any extra complexifications. We proceed with the decomposition of $\Lambda^2\fr{m}$ and $\Sym^2\fr{m}$.  For any $n>4$ note the following  decompositions of $\SO_{n}$-modules:  $\Lambda^2R(2\pi_1)=R(2\pi_1+\pi_2)\oplus R(\pi_2)$ and  $\Sym^2R(2\pi_1)=R(4\pi_1)\oplus R(2\pi_1)\oplus R(2\pi_2)\oplus R(0)$.   For $\Sp_{1}$ we have that $ \Lambda^2\fr{sp}(1)=\fr{sp}(1)$ and $\Sym^2\fr{sp}(1)=R(4\pi_1)\oplus R(0)=R(4\pi_{1})\oplus 1$. Hence we conclude that only those terms in the tensor square that contain a factor of $\Sym^2R(2\pi_1)$ and $\Lambda^2\fr{sp}(1)$ will yield  copies of $\fr{m}$. In particular,  the decomposition  
\[
\Lambda^2\fr{m}=\Lambda^2\big(R(2\pi_1)\hat{\otimes} \fr{sp}(1)\big) =\big(\Lambda^2R(2\pi_1)\hat{\otimes} \Sym^2\fr{sp}(1)\big)\oplus\big(\Sym^2R(2\pi_1)\hat{\otimes} \Lambda^2\fr{sp}(1)\big)
\]
  contains precisely one instance of $\fr{m}$, i.e.  ${\bold a}=1$. One the other hand, 
\[
\Sym^2\fr{m}=\Sym^2\big(R(2\pi_1)\hat{\otimes} \fr{sp}(1)\big) =\big(\Lambda^2R(2\pi_1)\hat{\otimes} \Lambda^2\fr{sp}(1)\big)\oplus \big(\Sym^2R(2\pi_1)\hat{\otimes} \Sym^2\fr{sp}(1)\big),
\]
hence  ${\bold s}=0$.  This proves the claim for $n>4$.  For completeness we    examine the  low-dimensional cases. 
Let first  $n=3$. The defining representation $\phi$ of $K=\SO_{3}\times \Sp_{1}$ must be of real dimension $12=3\times 4$ and hence the  only irreducible possibility is $\phi=\bb{R}^3\hat{\otimes} \bb{H}=R(2\pi_1)\hat{\otimes} R(\pi_1)$. Thus we get
\[
\Sym^2_{\bb{C}}\phi = \Sym^2_{\bb{C}}(\bb{R}^3\hat{\otimes}_{\bb{R}} \bb{C}^2)=\big(R(4\pi_1)\hat{\otimes} \fr{sp}(1)\big)^{\bb{C}}\oplus \big( \fr{sp}(1) \oplus \fr{so}(3)\big)^{\bb{C}}.
\]
Hence in this case $\fr{m}=R(4\pi_1)\hat{\otimes} \fr{sp}(1)$. As $\fr{so}(3)$-modules, we have that
\[
 \Lambda^2R(4\pi_1)=R(6\pi_1)\oplus \fr{so}(3), \quad \Sym^2R(4\pi_1)=R(8\pi_1)\oplus R(4\pi_1)\oplus R(0).
 \]
Therefore,  only products of $\Sym^2R(4\pi_1)$ and $\Lambda^2\fr{sp}(1)$ yield copies of $\fr{m}$.  Consequently, the result is the same as above, the multiplicity of $\fr{m}$ is one  in $\Lambda^2 \fr{m}$ and zero  in $\Sym^2 \fr{m}$.

Assume now that $n=4$.   The defining representation of $K=\SO_{4}\times \Sp_{1}$ is $\phi=\bb{R}^4\hat{\otimes} \bb{H}$, but $\bb{R}^4=R(\pi_1+\pi_2)$ in terms of highest weights, instead of being $R(\pi_1)$ as before, because $\SO_{4}$ is non-simple. We get
\[
\Sym^2_{\bb{C}}\phi = \Sym^2_{\bb{C}}(\bb{R}^4\hat{\otimes}_{\bb{R}} \bb{C}^2)= \big(R(2\pi_1+2\pi_2)\hat{\otimes} \fr{sp}(1)\big)^{\bb{C}}\oplus \big(\fr{sp}(1) \oplus \fr{so}(4)\big)^{\bb{C}}
\]
and thus $\fr{m}=R(2\pi_1+2\pi_2)\hat{\otimes} \fr{sp}(1)$ in this case.   As $\fr{so}(4)$-modules, we see that
\begin{eqnarray*}
 \Lambda^2R(2\pi_1+2\pi_2)&=&R(2\pi_1+4\pi_2)\oplus R(4\pi_1+2\pi_2)\oplus \fr{so}(3),\\
\Sym^2R(2\pi_1+2\pi_2)&=&R(4\pi_1+4\pi_2)\oplus R(4\pi_1)\oplus R(4\pi_2)\oplus R(2\pi_1+2\pi_2)\oplus R(0),
\end{eqnarray*}
and  the same argument as previously yields that  ${\bold a}=1$ and  ${\bold s}=0$.

Now, our assertion for $(\La^{3}\fr{m})^{K}$ can be deduced very easily as follows: Any invariant element of $\La^{3}\fr{m}$ induces an equivariant map in ${\rm Hom}_{K}(\fr{m}, \La^{2}\fr{m})$. For any $n\geq 3$ we have shown that  ${\bold a}=\dim_{\bb{R}}{\rm Hom}_{K}(\fr{m}, \La^{2}\fr{m})=1$. Thus, $\dim_{\bb{R}}(\La^{3}\fr{m})^{K}\leq 1$, but    by  Lemma \ref{dimskew}  we also get  $ \dim_{\bb{R}}(\La^{3}\fr{m})^{K}\geq 1$ and the result follows. Note that this method for the computation of the multiplicity $\ell$, applies  on any non-symmetric SII space $M=G/K$ whose isotropy representation is of real type and has ${\bold a}=1$ (or whose isotropy representation is of complex type and has ${\bold a}=2$, e.g. $\Ss^{6}\cong \G_2/\SU_3$).
  \end{proof}

    \section{Classification of Homogeneous $\nabla$-Einstein structures on SII spaces}
\label{IV}
  \subsection{Homogeneous $\nabla$-Einstein structures}   
  Similarly with invariant Einstein metrics on homogeneous Riemannian manifolds, on triples $(M^{n}, g, T)$ consisting  of a homogeneous  Riemannian manifold  $(M^{n}=G/K, g)$ endowed with a  (non-trivial)  invariant  3-form $T$, on may  speak of {\it homogeneous $\nabla$-Einstein structures}. In particular,
     
  \begin{definition}
A triple  $(M^{n}, g, T)$  of a  connected Riemannian manifold  $(M, g)$ carrying a (non-trivial) 3-form $T\in\La^{3}T^{*}M$, is called   a  {\it $G$-homogeneous   $\nabla$-Einstein manifold} (with skew-torsion) if there is  a closed subgroup  $G\subseteq\Iso(M, g)$ of the isometry group of $(M, g)$, which acts  transitively   on   $M$ and a  $G$-invariant  connection $\nabla$  compatible with $g$ and  with skew-torsion $T$,  whose  Ricci tensor    satisfies  the condition (\ref{skein}).
\end{definition}
   In this case, $g$ is a $G$-invariant metric,  the Levi-Civita connection $\nabla^{g}$ is a $G$-invariant metric connection and  since $2(\nabla-\nabla^{g})=T$,  the   torsion $T$ of $\nabla$ is given necessarily by  a  $G$-invariant 3-form  $0\neq T\in\Lambda^{3}(\fr{m})^{K}$, where $\fr{m}\cong T_{o}M$ is a reductive complement of $M=G/K$ with $K\subset G$ being the (closed) isotropy group. In particular,  the $\nabla$-Einstein condition (\ref{skein}) is   $\Ad(K)$-invariant, in the sense that the Ricci tensor $\Ric^{\nabla}$   is  a $G$-invariant covariant 2-tensor which  is described by an  $\Ad(K)$-invariant bilinear form on $\fr{m}$, and the same for its symmetric part. 
Moreover, 
   \begin{prop}\label{scalar}
  On a homogeneous Riemannian manifold $(M=G/K, g)$ carrying a $G$-invariant (non-trivial) 3-form $T\in\Lambda^{3}(\fr{m})^{K}$,  the  scalar curvature $\Sca=\Sca^{\nabla}$ associated to the $G$-invariant metric  connection   $\nabla:=\nabla^{g}+\frac{1}{2}T$ is a constant function on $M$. 
   \end{prop}
   \begin{proof}
 It is well-known that on a reductive homogeneous space, the scalar curvature $\Sca^{g}$ of the  Levi-Civita connection (related to a $G$-invariant Riemannian metric $g$, or the corresponding $\Ad(K)$-invariant inner product $\langle  \ , \ \rangle$ on the reductive complement $\fr{m}$) is independent of the point, i.e. it is a constant function on $M$ \cite{Bes, Nik}. Let $\nabla$ be a $G$-invariant metric connection on $(M=G/K, g)$ whose  skew-torsion coincides with the invariant 3-form $0\neq T\in\Lambda^{3}(\fr{m})^{K}$. Due to the identity $\Sca=\Sca^{g}-\frac{3}{2}\|T\|^{2}$ it is sufficient to prove that $\|T\|^{2}$ is constant, which is obvious  since $T$ corresponds to a  $G$-invariant tensor field.  Consequently,  $\Sca^{\nabla} : G/K\to\bb{R}$ is  constant.
  \end{proof}

    \subsection{On  the proofs of Theorems C, D and E}
 Let us focus now on an effective, non-symmetric (compact)  strongly isotropy irreducible homogeneous Riemannian manifold $(M=G/K, g=-B|_{\fr{m}})$, where $\fr{g}=\fr{k}\oplus\fr{m}$ is the associated B-orthogonal decomposition.  We denote  by $ {\cal{M}}^{sk}_{G}(\SO(G/K, g))\subseteq {\cal{M}}_{G}(\SO(G/K, g))\subseteq {\cal{A}}ff_{G}(F(G/K))$   the   space of $G$-invariant affine connections on $M=G/K$ which are compatible with the Killing metric $g=-B|_{\fr{m}}$  and  have invariant 3-forms $0\neq T\in(\Lambda^{3}\fr{m})^{K}$ as their torsion tensors.   For the corresponding set  of homogeneous  $\nabla$-Einstein structures, we will write ${\cal{E}}^{sk}_{G}(\SO(G/K, g))$. 
   As stated  in the introduction,   Lemma  \ref{dimskew}  and Schur's lemma  allow us to parametrize ${\cal{E}}^{sk}_{G}(\SO(G/K, g))$ by the space of  global $G$-invariant 3-forms. Hence,  this finally  yields the 
   identification ${\cal{E}}^{sk}_{G}(\SO(G/K, g))={\cal{M}}^{sk}_{G}(\SO(G/K, g))$.  With the aim to clarify this  identification  and give explicit proofs of  Theorems C, D and E in introduction, let us  recall  first the following important  result of  \cite{Chrysk}.

  \begin{theorem} \textnormal{(\cite[Thm.~4.7]{Chrysk})}\label{mine}
Let $(M^{n}=G/K, g)$  be an effective, compact and simply-connected, isotropy irreducible standard homogeneous Riemannian manifold $(M^{n}=G/K, g)$ of a compact connected  simple Lie group $G$, which is not a symmetric space of Type I. Then, $(M^{n}=G/K, g)$  is a $\nabla^{\al}$-Einstein manifold for  any parameter $\al\neq 0$,  where $\nabla^{\al}=\nabla^{g}+\frac{1}{2}T^{\al}=\nabla^{c}+\Lambda^{\al}$ is the 1-parameter family of $G$-invariant metric connections on $M$, with skew-torsion $0\neq T^{\al}=\al\cdot T^{c}$ (see Lemma \ref{mainuse1}). 
 \end{theorem}
 
Note that  for a symmetric space $M=G/K$ of Type I,    the associated space of $G$-invariant affine {\it metric} connections is always a point, i.e.  $\nabla^{\al}\equiv\nabla^{c}\equiv\nabla^{g}$ and no torsion appears.  On the other hand,  Theorem \ref{mine}    generalises the   well-known fact that a compact simple  Lie group  $G$ is $\nabla^{\al}$-Einstein with (non-trivial) {\it parallel torsion} for any $0\neq\al\in\bb{R}$, with  the   flat $\pm 1$-connections of Cartan-Schouten being the trivial members (see for example \cite[Lemma 1.8]{AFer} or \cite[Thm.~1.1]{Chrysk}). Notice however, that if  $M=G/K$ is not isometric to  a compact simple Lie group, then the $\nabla^{\al}$-Einstein structures described in  Theorem \ref{mine} have parallel torsion only for $\al=1$.  We finally remark that both   $\Ss^{6}=\G_2/\SU_3, \Ss^{7}=\G_2/\Spin_{7}$ are (strongly) isotropy irreducible and non-symmetric, hence they are $\nabla^{\al}$-Einstein manifolds with skew-torsion, for any $0\neq\al\in\bb{C}$, $0\neq\al\in\bb{R}$, respectively (due to the type of their isotropy representation).  The same applies for any compact, non-symmetric,  effective SII homogeneous Riemannian manifold  and this gives rise to  Theorem C which is  an immediate consequence of Theorem A.1 and Theorem \ref{mine}.
 
 Let us proceed now with a proof of   Theorems D and E.

\begin{proof} {\bf (Proof  of Theorem D)}
 If  $M=G/K$ is a manifold in Table \ref{table:two} whose isotropy representation is of real type, different than  $\SO_{10}/\Sp_{2}$, then Theorem A.2 (i) ensures the existence of a second (real) 1-parameter family of $G$-invariant connections $\nabla^{t}\neq \nabla^{a}$, compatible with the Killing metric and with skew-torsion  $T^{t}$ such that $T^{t}\neq T^{c}\sim T^{\al}$ (for any $t, \al\in \bb{R}$), where $T^{c}$ is the torsion of the (unique) canonical connection corresponding to $\fr{m}$.  Thus,  we can write $\nabla^{t}=\nabla^{g}+\frac{1}{2}T^{t}$ with $\nabla^{t}\neq \nabla^{\al}$.   Since $\Ric^{\nabla^{t}}\equiv \Ric^{t}$ is $G$-invariant, the same is  true for $\delta^{g}T^{t}=\delta^{\nabla^{t}}T^{t}$, in particular we can view the codifferential of $T^{t}\in(\La^{3}\fr{m})^{K}$ as a $G$-invariant  2-form.  However, $\chi$ is of real type, hence the trivial representation  $\bb{R}$ does not appear in $\Lambda^{2}\fr{m}$, i.e. $(\Lambda^{2}\fr{m})^{K}=0$. Hence, we deduce that  $\delta^{g}T^{t}=0=\delta^{t}T^{t}$ and since $\fr{m}$ is (strongly) isotropy irreducible over $\bb{R}$ and  the Ricci tensor $\Ric^{t}$ is  symmetric, by Schur's lemma it     must be a multiple of the Killing metric, i.e. $(M=G/K, -B|_{\fr{m}}, \nabla^{t})$ is    $\nabla^{t}$-Einstein with skew-torsion. Our final claim  follows now  in combination  with   Theorem \ref{mine}.   
  \end{proof}
      
    It is well-known that  an effective  SII homogeneous space $M=G/K$ admits an (integrable) $G$-invariant complex structure if and only if is a Hermitian symmetric space \cite{Wolf}. Moreover,  the existence of an invariant  {\it almost} complex structure $J\in \Ed(\fr{m})$ on  a strongly isotropy irreducible space implies that the isotropy representation is not of real type, hence $\chi=\phi\oplus\overline{\phi}$ for some irreducible complex representation with $\phi\ncong\overline{\phi}$.  Consequently,  any manifold which appears in Tables \ref{table:four}, \ref{table:five} and whose isotropy representation is  of complex type, is a $G$-homogeneous almost complex manifold (see also \cite[Cor.~13.2]{Wolf}).  Notice also

  \begin{lemma} \label{complex}
  Let $\fr{k}$ be a compact Lie algebra and let $\rho : \fr{k}\to\Ed(\fr{m})$ be a faithful (irreducible) representation of $\fr{k}$ over $\bb{R}$, endowed with an invariant inner product $B_{\fr{m}}$. Assume that $\dim\fr{m}\geq 2$.  If $\fr{m}$ admits an $\ad(\fr{k})$-invariant complex structure $J$ (as a vector space), then $\Lambda^{2}\fr{m}$ contains  the trivial representation $\bb{R}$.  
  \end{lemma}
  \begin{proof}
  We only  mention that since $\fr{m}$ is an irreducible   complex type representation of a compact Lie algebra $\fr{k}$, it is  unitary, therefore the $\ad(\fr{k})$-invariant  K\"ahler form $\omega(X, Y)=B_{\fr{m}}(JX, Y)$  gives rise to an  $\ad(\fr{k})$-invariant element inside $\La^{2}\fr{m}$.
   \end{proof}
Consider now the spaces $\SO_{n^{2}-1}/\SU_{n}$ $(n\geq 4)$ and  $\E_6/\SU_{3}$.  Since their isotropy representation is of complex type,  Lemma \ref{complex} certifies the existence of   $G$-invariant 2-forms.  Thus, in contrast to  Theorem D, we cannot  deduce that the Ricci tensor of all predicted $G$-invariant metric connections with skew-torsion  must be necessarily symmetric (although this is the case always for $\Ric^{\al}$).  However, since we are considering  the isotropy irreducible case, we obtain Theorem E as follows:
 \begin{proof} {\bf (Proof  of Theorem E)}
  Assume that $(M=G/K, g=-B|_{\fr{m}})$ is one of the manifolds  $\SO_{n^{2}-1}/\SU_{n}$ $(n\geq 4)$ or $\E_6/\SU_{3}$. By Theorem A.2 (ii) (see also  Table \ref{table:two}) we know that $M=G/K$  admits  a 4-dimensional space of  $G$-invariant metric connections with skew-torsion.  Now, the  $\nabla$-Einstein condition is related only with the symmetric part of the Ricci tensor associated to any such  connection. Since this tensor is $G$-invariant,   Schur's lemma ensures that the $\nabla$-Einstein equation  is satisfied  for  any available $G$-invariant metric connection $\nabla$  with skew-torsion.  Therefore,      the space of $G$-invariant $\nabla$-Einstein structures  has the same dimension with the space of $G$-invariant metric connections with skew-torsion. This  proves  Theorem E.  \end{proof}


\begin{thebibliography}{40}
   \bibitem[A]{Agr03}  I.~Agricola, {\it Connections on naturally reductive spaces, their Dirac operator and homogeneous models in string theory}, Comm. Math. Phys.  $\bold{232}$, no. 3, (2003), 535--563. 
 \bibitem[ABK]{ABK}
I.~ Agricola, J.~ Becker-Bender, H.~ Kim, {\it Twistorial eigenvalue estimates for generalized Dirac operators with torsion}, Adv. Math. $\bold{243}$ (2013), 296--329.
  \bibitem[AF]{AFer} I.~Agricola,  A.~C.~Ferreira, {\it Einstein manifolds with skew torsion}, 
 Oxford Quart. J. $\bold{65}$  (2014), 717--741.
  \bibitem[AFr1]{AF} 
 I.~Agricola, Th.~Friedrich, {\it On the holonomy of connections with 
  skew-symmetric torsion}, Math. Ann. $\bold{328}$, (2004), 711--748. 
  \bibitem[AFr2]{3Sak}
 I.~Agricola,  Th.~Friedrich, {\it    3-Sasakian manifolds in dimension seven, their spinors and $\G_2$-structures}, J. Geom. Phys. $\bold{60}$, (2010), 326--332. 
   \bibitem[AFH]{AFH}
I.~Agricola, Th.~Friedrich, J.~H\"oll,  {\it $\Sp(3)$-structures on 14-dimensional manifolds},  J. Geom. Phys. $\bold{69}$ (2013), 12--30.
   \bibitem[AKr]{AK} I.~Agricola, M.~Kraus, {\it  Manifolds with vectorial torsion},  Diff. Geom. Appl. $\bold{45}$ (2016), 130--147.
   \bibitem[A$\ell$C]{spin}
   D.~V.~Alekseevsky, I.~Chrysikos, {\it  Spin structures on compact homogeneous pseudo-Riemannian manifolds}, arXiv:1602.07968v2.
\bibitem[AVL]{AVL}  D.~V.~Alekseevsky,  A.~M.~Vinogradov,  V.~V.~Lychagin, {\it Geometry I - Basic Ideas and Concepts of Differential Geometry}, Encyclopaedia of Mathematical Sciences,Vol. 28,  Springer-Verlag, Berlin, 1991.
\bibitem[Arv]{Arvanito}  A.~Arvanitoyeorgos, {\it An Introduction to Lie Groups and the Geometry of Homogeneous Spaces}, Amer. Math. Soc., Student Math. Library, Vol. 22, 2003.
 
  \bibitem[BB]{Julia} J. Becker-Bender: {\it  Dirac-Operatoren und Killing-Spinoren mit Torsion}, Ph.D. Thesis, University of Marburg (2012).
 \bibitem[BEM]{Ben1}
 P.~Benito, A.~Elduque,  F.~Mart\'in-Herce, {\it Irreducible Lie--Yamaguti algebras}, J. Pure Appl.  Algebra, 
 $\bold{213}$, (5) (2009),  795--808.
\bibitem[B]{Bes} A. L. Besse,  {\it Einstein Manifolds},   Springer-Verlag, Berlin, 1986.   
\bibitem[BT]{Broker} T. Br\"ocker, T. Tom Dieck, {\it Representations of Compact Lie Groups},  Springer-Verlag, New York, 1985.

  \bibitem[CG]{Cahen}
  M.~Cahen, S.~Gutt, {\it Spin structures on compact simply connected Rieamannian symmetric spaces}, Simon Stevin $\bold{62}$  (1988), 291--330.
  \bibitem[\v{C}S]{Cap}
  A.~ \v{C}ap, J.~Slov\'ak, {\it  Parabolic geometries  I: Background and general theory},  Mathematical Surveys and Monographs, 154.,  A.M.S., RI, 2009.
  \bibitem[C1]{Chrysk} I.~Chrysikos, {\it Invariant connections with skew-torsion and $\nabla$-Einstein manifolds}, J. Lie Theory $\bold{26}$, (2016),   11--48.
   \bibitem[C2]{Chrysk2} I.~Chrysikos, {\it Killing and twistor spinors with torsion}, Ann. Glob.  Anal.  Geom. $\bold{49}$, (2016), 105--141.
       \bibitem[C$\ell$]{Cleyton} R.~Cleyton, {\it $G$-structures and Einstein metrics}, PhD Thesis, University of Southern Denmark, Odense, 2001, {{ftp://ftp.imada.sdu./pub/phd/2001/24.PS.gz}} 
   \bibitem[C$\ell$S]{CSwan} 
   R.~Cleyton, A.~Swann, {\it Einstein metrics via intrinsic or parallel torsion}, 
    Math.  Z.  ${\bold 247}$ (3), (2004), 513--528.
         \bibitem[DI]{Dalakov} P.~Dalakov, S.~Ivanov, {\it Harmonic spinors of Dirac operator of connection with torsion in dimension 4}, Class. Quant. Grav., $\bold{18}$, (2001), 253--265.
 \bibitem[DGP]{Draper}
  C.~A.~Draper,  A.~Garvin, F.~J.~Palomo, {\it Invariant affine connections on odd dimensional spheres}, Ann. Glob.  Anal.  Geom. $\bold{49}$, (2016), 213--251.
\bibitem[E$\ell$M]{Eldu}
 A.~Elduque, H.~C.~Myung, {\it The Reductive Pair $(B4, B3)$ and Affine Connections on $S^{15}$}, Journal of Algebra
$\bold{227}$,  (2) (2000), 504--531.
  \bibitem[Fr]{Fr79}   Th.~Friedrich, {\it Einige differentialgeometrische Untersuchungen des Dirac-Operators einer Riemannschen Mannigfaltigkeit}, Dissertation B (habilitation), Humboldt Universit\"at zu Berlin, 1979.
     \bibitem[FrIv]{FrIv}  Th.~Friedrich, S.~ Ivanov, {\it Parallel spinors and connections with skew-symmetric torsion in string theory},  Asian J. Math. $\bold{6}$ (2002), no. 2, 303--335.  
\bibitem[FrS]{FrS} Th.~Friedrich, R.~Sulanke, {\it Ein kriterium f\"ur die formale sebstadjungiertheit des Dirac operators}, Colloqium Mathematicium, (1979), 239--247.
 \bibitem[I]{Itoh}
M.~Itoh: {\it Invariant connections and Yang-Mills solutions}, Trans. Amer. Math. Soc. $\bold{267}$, (1), (1981), 229--236.
 \bibitem[IvP]{IvPap} S.~Ivanov, G.~Papadopoulos, {\it Vanishing theorems and strings backgrounds}, Class. Quant. Grav. $\bold{18}$, (2001), 1089--1110.
\bibitem[KN]{Kob2}  S. Kobayashi, K. Nomizu, {\it Foundations of Differential Geometry  Vol II}, Wiley - Interscience, New York, 1969.  
  \bibitem[K]{K} B.~Kostant, {\it A characterization of invariant affine connections}, Nagoya Math. J. $\bold{16}$, (1960), 33--50.
   \bibitem[L1]{Laqq} 
    H.~T.~Laquer, {\it Stability properties of the Yang-Mills functional near the canonical connection}, 
    Michigan Math. J. $\bold{31}$ (2) (1984), 139--159.
      \bibitem[L2]{Laq1} 
 H.~T.~Laquer, {\it Invariant affine connections on Lie groups}, Trans. Am. Math. Soc.  $\bold{331}$, (2), (1992), 541--551.
  \bibitem[L3]{Laq2} 
 H.~T.~Laquer,   {\it Invariant affine connections on symmetric spaces}, Proc. Am. Math. Soc.  $\bold{115}$, (2), (1992), 447--454.
\bibitem[Le]{Le} Hong-Van Le, {\it Geometric structures associated with a simple Cartan 3-form}, J. Geom. Phys.   $\bold{70}$, (2013), 205--223.
\bibitem[NRS]{Nik} Yu.~G.~Nikonorov, E.~D.~Rodionov, V.~V.~Slavskii, {\it Geometry of homogeneous Riemannian manifolds}, J. Math. Sci.    $\bold{146}$  (6) (2007)  6313--6390.
 \bibitem[N]{N}  K.~Nomizu, {\it  Invariant affine   connections on homogeneous spaces},  Amer.  J. Math. $\bold{76}$ (1954), 33--65.
  \bibitem[OR1]{Olmos}
    C.~Olmos, S.~Reggiani, {\it  The skew-torsion holonomy theorem and naturally reductive spaces}, J. Reine Angew. Math. $\bold{664}$ (2012), 29--53. 
    \bibitem[OR2]{OReg} 
  C.~Olmos, S.~Reggiani, {\it  A note on the uniqueness of the canonical connection of a naturally reductive space}, arXiv: 12108374. 
    \bibitem[OV]{Oni2} A.~L.~Onishchik, E.~B.~Vinberg, {\it Lie groups and algebraic groups}, Springer-Verlag, New York 1990.
     \bibitem[PS]{Pfa}  F.~Pf\"affle,  C.~Stephan, {\it On gravity, torsion and the spectral action principle}, J. Funct. Anal. $\bold{262}$ (2012), 1529--1565. 
       \bibitem[R]{Reg1}  S.~Reggiani, {\it On the affine group of a normal homogeneous manifold}, Ann. Glob. Anal. Geom. $\bold{37}$, (2010),  351--359.
\bibitem[Sim]{Simon}
 B.~Simon, {\it Representations of Finite and Compact Lie Groups},
Graduate Studies in Mathematics; V.10, AMS, 1996.
  \bibitem[TrV]{Tric1} 
  F.~Tricerri, L.~Vanhecke, {\it  Homogeneous structures on Riemannian manifolds}, London Math. Soc. Lecture Notes Series, vol. 83, Cambridge Univ. Press, Cambridge, 1983.
 \bibitem[W]{Wang}  H.~C.~Wang, {\it On invariant connections over a principal fibre bundle}, Nagoya Math. J. $\bold{13}$, (1958), 1--19. 
   \bibitem[WZ]{Wa5} M. Wang, W. Ziller, {\it Symmetric spaces and strongly isotropy irreducible spaces}, Math. Ann. $\bold{296}$, (1993), 285--326.
   \bibitem[Wo1]{Wolf}  J. A. Wolf, {\it The geometry and the structure of isotropy irreducible homogeneous spaces}, Acta. Math.  120 (1968) 59-148; correction: Acta Math.  152, (1984), 141--142. 
  \bibitem[Wo2]{Wolf2}J. A. Wolf, {\it Spaces of Constant Curvature}, Fifth Edition, Publish or Perish, Inc., U.S.A., 1984.
\end{thebibliography}
 \end{document}